\documentclass[10pt]{article}
	
	\usepackage{caption}
    \usepackage{url}
    \usepackage{verbatim}
    \usepackage{parskip}
	\usepackage{amsmath}
    \usepackage[usenames, dvipsnames]{color}
	\usepackage{amsthm}
	\usepackage{amsfonts}
	\usepackage{graphicx}
	\graphicspath{ {figures/}{figures/eps/}{figures/pdf/}{figures/png/}}
    \usepackage{setspace}
	\usepackage{bm}
	\usepackage{float}
	\usepackage{subfig}
	\usepackage{mathtools}

	\def\captionfont{\setb@se{11pt}\protect\footnotesize}
    \def\captionfont{\protect\footnotesize}

    \newcommand{\iprd}[2]{\left( #1 , #2 \right)}
    
    \newcommand{\Tauh}{\mathcal{T}_h}
    
    \newcommand{\Nodeh}{\mathcal{N}_h}
    \newcommand{\aiprd}[2]{a\left( #1 , #2 \right)}

    \newcommand{\clabel}{c}
    \newcommand{\cform}[6]{\clabel \left( #1 , #2, #3, #4; #5, #6 \right)}
    \newcommand{\eform}[4]{e\left( #1 , #2;  #3 , #4 \right)}

    \newcommand{\Sh}{S_h}

    \newcommand{\N}{\mathbb{N}}
    \newcommand{\Sp}{\mathbb{S}}
    \newcommand{\U}{\mathbb{U}}
    \newcommand{\V}{\mathbb{V}}
    \newcommand{\Vp}{\V^{\perp}}
    \newcommand{\W}{\mathbb{W}}
    \newcommand{\X}{\mathbb{X}}
    \newcommand{\Y}{\mathbb{Y}}
    \newcommand{\Nh}{\N_h}
    \newcommand{\Sph}{\Sp_h}
    \newcommand{\Uh}{\U_h}
    
    \newcommand{\Vph}{\Vp_h}
    \newcommand{\Wh}{\W_h}
    \newcommand{\Xh}{\X_h}
    \newcommand{\Yh}{\Y_h}

    \newcommand{\R}{\mathbb{R}}

	\def\norm#1#2{\left\| #1 \right\|_{#2}}

	\newcommand{\dtau}{\delta_\tau}

	\newcommand{\phih}{\phi_h}

	\newcommand{\nuh}{\nu_h}
	\newcommand{\psih}{\psi_h}
    \newcommand{\br}{{\bf r}}
	\newcommand{\bu}{{\bf u}}
	\newcommand{\bv}{{\bf v}}
	\newcommand{\bw}{{\bf w}}
	\newcommand{\bx}{{\bf x}}
	
	\newcommand{\bn}{{\bf n}}
	
	\newcommand{\bntilh}{\tilde{{\bf n}}_h}
	\newcommand{\bnh}{{\bf n}_h}
	
	\newcommand{\bI}{{\bf I}}
    \newcommand{\brh}{\br_h}
	\newcommand{\buh}{\bu_h}
	\newcommand{\bvh}{\bv_h}
	\newcommand{\bwh}{\bw_h}
	
	\newcommand{\sh}{s_h}
	\newcommand{\shave}{\overline{s}_h}
	\newcommand{\muh}{\mu_h}

    \newcommand{\Etot}{E}                     
    \newcommand{\Eerk}{E_{\mathrm{erk}}}      
    \newcommand{\Edw}{E_{\mathrm{dw}}}        
    \newcommand{\Eerktot}{J}        
    \newcommand{\dwfunc}{f}                   
    \newcommand{\Ech}{E_{\mathrm{ch}}}  
     \newcommand{\Echdw}{E_{\mathrm{chdw}}}  
      \newcommand{\Echgd}{E_{\mathrm{chp}}}       
    \newcommand{\Ewan}{E_{\mathrm{a},\bn}}      
    \newcommand{\Ewas}{E_{\mathrm{a},s}}      
    \newcommand{\Werk}{\omega_{\mathrm{erk}}} 
    \newcommand{\Wdw}{\omega_{\mathrm{dw}}}   
    \newcommand{\Wchdw}{\omega_{\mathrm{chdw}}}   
    \newcommand{\Wchgd}{\omega_{\mathrm{chp}}}   
    \newcommand{\Wwan}{\omega_{\mathrm{a},\bn}} 
    \newcommand{\Wwas}{\omega_{\mathrm{a},s}} 
    \newcommand{\bdys}{\Gamma_{s}}
    
    \newcommand{\bdybu}{\Gamma_{\bu}}
    
    \newcommand{\Admerk}{\mathbb{A}_{erk}}
    \newcommand{\Admerkh}{\mathbb{A}_{erk,h}}
    
    \newcommand{\Admall}{\mathbb{A}}
    \newcommand{\Adelta}{A_{\delta}}
    \newcommand{\Sing}{\mathcal{S}}

	\newtheorem{thm}{Theorem}[section]
	\newtheorem{prop}[thm]{Proposition}
	\newtheorem{cor}[thm]{Corollary}
	
	\newtheorem{rmk}[thm]{Remark}
	\newtheorem{lem}[thm]{Lemma}

\begin{document}
\title{A Finite Element Method for a Phase Field Model of \\ Nematic Liquid Crystal Droplets}

	\author{Amanda E. Diegel, ~Shawn W. Walker}
	
	\maketitle
	
	\numberwithin{equation}{section}

\begin{abstract}

We develop a novel finite element method for a phase field model of nematic liquid crystal droplets. The continuous model considers a free energy comprised of three components: the Ericksen's energy for liquid crystals, the Cahn-Hilliard energy representing the interfacial energy of the droplet, and a weak anchoring energy representing the interaction of the liquid crystal molecules with the surface tension on the interface (i.e. anisotropic surface tension).  Applications of the model are for finding minimizers of the free energy and exploring gradient flow dynamics.  We present a finite element method that utilizes a special discretization of the liquid crystal elastic energy, as well as mass-lumping to discretize the coupling terms for the anisotropic surface tension part.  Next, we present a discrete gradient flow method and show that it is monotone energy decreasing.  Furthermore, we show that global discrete energy minimizers $\Gamma$-converge to global minimizers of the continuous energy.  We conclude with numerical experiments illustrating different gradient flow dynamics, including droplet coalescence and break-up.

\end{abstract}

\section{Introduction}
\label{sec:intro}

The purpose of this paper is to couple Ericksen's model for nematic liquid crystals to an interfacial energy (modeled via the Cahn-Hilliard equation) in order to model liquid crystal droplets.  Interest in developing numerical methods for modeling liquid crystals or complex fluids involving liquid crystals has grown in recent years, \cite{Alouges_SJNA1997, Barrett_M2AN2006,Cohen_CPC1989, Guillen-Gonzalez_M2AN2013, Liu_SJNA2000,Lin_SJNA1989, Walkington_M2AN2011, nochetto2017finite, Nochetto_JCP2017}. One driver for this development is the large host of technological applications of liquid crystals \cite{Ackerman_NC2015, Araki_PRL2006, Bisoyi_CSR2011, Blanc_Sci2016, Humar_OE2010, Moreno-Razo_Nat2012, Musevic2011, Rahimi_PNAS2015, Shah_Small2012, Sun_SMS2014, Wang_NL2014}. Popular models representing liquid crystal substances include the Q-tensor model, the Oseen-Frank model, and Ericksen's model with a variable degree of orientation.  A common issue in any of these methods is capturing defects.
For instance, in \cite{Barrett_M2AN2006}, Barrett et.~al.~presents a fully discrete finite element method for the evolution of uniaxial nematic liquid crystals with variable degree of orientation.  An advantage of their method is that they are able to provide convergence results. However, in order to avoid the degeneracy introduced by the degree of orientation variable $s$, they use a regularization of Ericksen's model.

The use of diffuse interface theory to describe the mixing of complex fluids has likewise grown in popularity and the research group which includes J.~Zhao, X.~Yang, Q.~Wang, J.~Shen (among others) has released several papers on this subject \cite{zhao2016decoupled,yang2017numerical,zhao2016semi,zhao2017novel,zhao2017decoupled,zhao2016numerical}. Their models may be described as energy minimizing models whereby their energy functionals are composed of a kinetic energy and a free energy. The kinetic energy is based on fluid velocity coming from a fluid model, such as Stoke's flow. The free energy is then broken down into three parts: the mixing energy, the bulk free energy for liquid crystals, and an anchoring energy. For instance, in \cite{zhao2016decoupled}, Zhao et.~al.~develop an energy-stable scheme for a binary hydrodynamic phase field model of mixtures of nematic liquid crystals and viscous fluids where they use the Cahn-Hilliard energy to describe the mixing energy and the Oseen-Frank energy to describe the bulk free energy for liquid crystals.  Defects are effectively regularized by penalizing the unit length constraint.

The work presented herein is unique in the following sense: the Cahn-Hilliard energy is combined directly with Ericksen's energy in order to develop a phase field model for nematic liquid crystal droplets in a pure liquid crystal substance.  The model considers a free energy which is comprised of three components: the Ericksen's energy for liquid crystals, the Cahn-Hilliard energy representing the interfacial energy of the droplet, and a weak anchoring energy representing the interaction of the liquid crystal molecules with the surface tension on the interface (which gives rise to anisotropic surface tension).  The goal is to find minimizers of this free energy.  To this end, we present a finite element discretization of the energy and apply a modified time-discrete gradient flow method to compute minimizers.  In this way, the numerical scheme considered herein combines the finite element approximation of the Ericksen model of nematic liquid crystals in \cite{nochetto2017finite}, which captures point and line defects and requires no regularization, and the technique considered in \cite{diegel15} which follows a convex splitting gradient flow strategy for modeling the Cahn-Hilliard equation.

An outline of the paper is as follows.  Section \ref{sec:contin_energy_models} describes the continuous energy model for the liquid crystal/surface tension system.  In Section \ref{sec:energy-discretization}, we present a discretization of the total energy \eqref{eq:energy-total} followed by the development of a discrete gradient flow strategy in Section \ref{sec:gradient-flow}. In Section \ref{sec:fully-discrete-scheme}, we present a fully discrete finite element scheme based on the gradient flow strategy and prove its stability. In Section \ref{sec:gamma-convergence}, we demonstrate that the discrete energy converges to the continuous energy using the tools of $\Gamma$-convergence. We conclude with several numerical experiments in Section \ref{sec:numerical-experiments}, and some discussion in Section \ref{sec:conclusion}.

\section{Continuous Energy Models}\label{sec:contin_energy_models}

\subsection{Ericksen Energy}\label{sec:Erk_energy}

We consider the one-constant model for liquid crystals with variable degree of orientation \cite{Ericksen_ARMA1991, deGennes_book1995, Virga_book1994} (Ericksen's model) on a Lipschitz domain $\Omega \subset \mathbb{R}^d$ with $d=2,3$.  The liquid crystal state is modeled by a director field $\bn(x)$ and a scalar function $s(x)$, the so-called degree-of-orientation.  Equilibrium is attained when $(s,\bn)$ minimizes the non-dimensional energy
\begin{equation}\label{eq:ericksen_total_energy}
  \Eerktot(s,\bn) = \Eerk(s,\bn) + \Edw(s),
\end{equation}
where $\Eerk(s,\bn)$ and $\Edw(s)$ are defined by
\begin{align}
\Eerk(s,\bn) &:= \int_{\Omega} \left[ \kappa |\nabla s|^2 + s^2 |\nabla \bn|^2 \right] \, d\bx,
\label{eq:energy-ericksen}
\\
\Edw(s) &:= \int_{\Omega} \dwfunc (s(\bx)) \, d\bx,
\label{eq:energy-ericksen-dw}
\end{align}
with $\kappa > 0$ and where the double well potential $\dwfunc$ is a $C^2$ function defined on $-1/2 < s < 1$ that satisfies the following conditions \cite{Ericksen_ARMA1991, Ambrosio_MM1990a, Lin_CPAM1991}:
\begin{enumerate}
  \item $\lim_{s \rightarrow 1} \dwfunc (s) = \lim_{s \rightarrow -1/2} \dwfunc (s) = \infty$,
  \item $\dwfunc(0) > \dwfunc(s^*) = \min_{s \in [-1/2, 1]} \dwfunc(s)$ for some $s^* \in (0,1)$,
  \item $\dwfunc'(0) = 0$.
\end{enumerate}

The existence of minimizers $(s^*,\bn^*)$ of \eqref{eq:ericksen_total_energy} was shown in \cite{Ambrosio_MM1990a, Lin_CPAM1991}, along with regularity properties.  Minimizers may exhibit non-trivial defects (depending on boundary conditions) \cite{Bethuel_book1994, Blinov_book1983, Brezis_CMP1986, Lin_CPAM1991, Lin_CPAM1989, Schoen_JDG1982}.  Some analytical solutions can be found in \cite{Virga_book1994}.  The presence of $s$ in \eqref{eq:energy-ericksen} gives a \emph{degenerate} Euler-Lagrange equation for $\bn$.  This allows for line and plane defects (singularities of $\bn$) when $s$ vanishes in dimension $d=3$.  The size of defects and regularity properties of minimizers were studied in \cite{Lin_CPAM1991}.  This lead to the study of dynamics \cite{Calderer_SJMA2002} and corresponding numerics \cite{Barrett_M2AN2006}.  However, in both cases, they regularize the model to avoid the degeneracy induced by the order parameter $s$ vanishing. In \cite{nochetto2017finite}, they present a numerical method, without requiring any regularization, for computing minimizers of \eqref{eq:ericksen_total_energy} that exhibit non-trivial defect structures.

The theoretical framework follows \cite{Ambrosio_MM1990a, Lin_CPAM1991}.  We introduce an auxiliary variable $\bu := s\bn$, and rewrite Ericksen's energy \eqref{eq:ericksen_total_energy} as
\begin{align}
\Eerk(s,\bn) = \widetilde{\Eerk}(s,\bu) := \int_\Omega \left( (\kappa -1) |\nabla s|^2 + |\nabla \bu|^2 \right) d{\bf x},
\label{eq:energy-s-u}
\end{align}
which follows from differentiating the identity $|\bn|^2 = 1$.  This suggests the following admissible class for $(s,\bu)$:
\begin{align}
    \Admerk := \{(s,\bu) \in [H^1(\Omega)]^{d+1} : ~\text{ there exists $\bn$ such that \eqref{eq:struct_condition} holds} \},
\label{eq:admissible-class-s-u}
\end{align}
where
\begin{equation}\label{eq:struct_condition}
  \bu = s \bn, \quad - 1/2 < s < 1 \text{ a.e.~in } \Omega, \text{ and} \quad \bn \in \mathbb{S}^{d-1} \text{ a.e.~in } \Omega,
\end{equation}
is called the structural condition of $\Admerk$.  Note: we use an abuse of notation and define $(s,\bn)$ in $\Admerk$ to mean $(s,\bu)$ in $\Admerk$ with $\bu = s \bn$.

Moreover, to enforce boundary conditions on $(s,\bu)$, possibly on different parts of the boundary, let $(\Gamma_s, \Gamma_\bu)$ be open subsets of $\partial \Omega$ where we set Dirichlet boundary conditions for $(s,\bu)$. Then the restricted admissible class is defined by
\begin{align}\label{eq:admis_set_BCs}
\Admerk(g,{\bf r}) := \{(s,\bu) \in \Admerk : s|_{\Gamma_s} = g, \quad \bu|_{\Gamma_\bu} = {\bf r}\},
\end{align}
for some given functions $(g,{\bf r}) \in [W_\infty^1(\Omega)]^{d+1}$ that satisfy \eqref{eq:struct_condition} on $\partial \Omega$.  If we further assume
\begin{equation}\label{eq:g_pos}
g \ge \delta_0 \quad\text{ on } \partial \Omega, ~ \text{ for some } \delta_0 > 0,
\end{equation}
then $\bn$ is $H^1$ in a neighborhood of $\partial \Omega$ and satisfies $\bn=g^{-1} {\bf r} \in \mathbb{S}^{d-1}$ on $\partial \Omega$.

In the case where $s$ is a non-zero constant, \eqref{eq:energy-ericksen} effectively reduces to the Oseen-Frank energy $\int_{\Omega} |\nabla \bn|^2$.  If $s$ is variable, it may vanish in order to relax the energy of defects.  In this case, discontinuities of $\bn$ (i.e. defects) may occur in the singular set
\begin{align}\label{set:singular_s}
    \Sing := \{ x \in \Omega :\; s(x) = 0 \},
\end{align}
with \emph{finite} energy: $\Eerk(s,\bn) < \infty$.  The parameter $\kappa$ in \eqref{eq:energy-ericksen} can influence the appearance of  defects; see \cite{nochetto2017finite, NochettoWalker_MRSproc2015} for examples of this effect.

\subsection{Phase Field Energy}

The Cahn-Hilliard (CH) energy is given by \cite{cahn61, cahn58}
\begin{align}
\Ech(\phi) = \int_{\Omega} \frac{1}{4\varepsilon} \left(\phi^2-1\right)^2 + \frac{\varepsilon}{2} \int_{\Omega} |\nabla \phi |^2 \, d\bx =: \Echdw (\phi) + \Echgd (\phi),
\label{eq:energy-CH}
\end{align}
where $\varepsilon > 0$ is a small constant representing the interfacial width between the liquid crystal droplet and surrounding liquid crystal substance and $\phi$ represents a concentration field. The CH energy \eqref{eq:energy-CH} typically prefers the pure phase values $\phi = \pm 1$ and may be described as representing a competition between two different energy density terms: the double well density $\frac{1}{4\varepsilon} \left(\phi^2-1\right)^2$ which is minimized by the pure phase values of $\phi$ and the gradient energy density $\frac{\varepsilon}{2} \int_{\Omega} |\nabla \phi |^2$ which penalizes any derivatives of $\phi$.  The natural admissible class for $\phi$ is $H^1(\Omega)$.

%


The interfacial energy associated with the liquid crystal molecules interacting with the surface tension of the interface is given by a weak anchoring energy \cite{deGennes_book1995,Virga_book1994}.  We define the weak anchoring energy as
\begin{equation}
\begin{split}
\Ewan(s,\bn,\phi) &= \frac{\varepsilon}{2} \int_{\Omega} s^2 \left[ |\bn|^2 |\nabla \phi|^2 - \left(\bn \cdot \nabla \phi \right)^2 \right] \, d\bx, \\
\Ewas(s,\phi) &= \frac{\varepsilon}{2} \int_{\Omega} |\nabla \phi|^2 (s(\bx) - s^*)^2 \, d\bx,
\label{eq:energy-anch}
\end{split}
\end{equation}
as in \cite[eqn. (66)]{Mottram_arXiv2014} and where $\varepsilon$ is included to ensure that $\Ewan$ scales the same as $\Echgd$.  The total anchoring energy is then considered to be $E_{anch}(s, \bn, \phi):=\Ewan(s,\bn,\phi) + \Ewas(s,\phi)$. Note that $\Ewan(\bn,\phi)$ tries to force normal anchoring of $\bn$ (with respect to $\nabla \phi$) when minimized.  Planar anchoring can also be considered and is an obvious modification of the method presented here.


Combining the three components produces the total energy
\begin{align}
  \Etot(s,\bn,\phi) =& \, \Werk \Eerk(s,\bn) + \Wdw \Edw(s) + \Wchdw \Echdw(\phi)
  \nonumber
  \\
  &+ \Wchgd \Echgd(\phi) + \Wwan \Ewan(s,\bn,\phi) + \Wwas \Ewas(s,\phi),
  \label{eq:energy-total}
\end{align}
where $\Werk, \Wdw, \Wchdw, \Wchgd, \Wwan, \Wwas > 0$ are constants denoting various ``weights''. The total energy is then described as consisting of a liquid crystal energy (using the Ericksen model), an interfacial energy (using the Cahn-Hilliard model), and an energetic coupling term that connects the two.

\vspace{.1in}

\begin{rmk}[anisotropic surface tension]\label{rm:anisotropic_surface_energy}

Let $\tilde{J}(s,\bn,\phi) := \Ech(\phi) + \Ewan(s,\bn,\phi)$, which has the form:
\begin{align*}
  \tilde{J}(s,\bn,\phi) &= \frac{1}{4\varepsilon} \int_{\Omega} \left(\phi^2-1\right)^2 + \frac{\varepsilon}{2} \int_{\Omega} |\nabla \phi |^2 \, d\bx + \frac{\varepsilon}{2} \int_{\Omega} s^2 \left[ |\bn|^2 |\nabla \phi|^2 - \left(\bn \cdot \nabla \phi \right)^2 \right] \, d\bx \\
  &= \frac{1}{4\varepsilon} \int_{\Omega} \left(\phi^2-1\right)^2 + \frac{\varepsilon}{2} \int_{\Omega}  \nabla \phi \cdot \left[ \bI +s^2\left( \bI -  \bn \otimes \bn \right) \right] \nabla \phi \, d\bx.
\end{align*}
Thus, combining $\Ech$ with $\Ewan$ changes the effective surface tension from isotropic to \emph{anisotropic}. We note that we have taken the weights equal to 1 for simplicity but that this property holds for any weights $\Wwan, \Wchdw, \Wchgd >0$.
\end{rmk}

%
%

\section{Spatial Discretization of the Energy}
\label{sec:energy-discretization}



Assume the domain $\Omega$ is partitioned into a conforming simplicial triangulation $\Tauh = \{ K \}$.  The set of nodes (vertices) of $\Tauh$ is denoted $\Nodeh$ with cardinality $N$.  We further assume the following property on the so-called stiffness matrix entries
\begin{equation}\label{discrete_max_principle}
  k_{ij} := -\int_{\Omega} \nabla \eta_i \cdot \nabla \eta_j \, dx ,
\end{equation}
such that $k_{ij} \geq 0$ for all $i\ne j$ and where $\eta_i$ is the standard ``hat'' basis function associated with node $\bx_i \in \Nodeh$.  This is guaranteed if the mesh is weakly acute \cite{Ciarlet_CMAME1973, Strang_FEMbook2008, Korotov_MC2001, Brandts_LAA2008}.  Note: weak acuteness is guaranteed if all interior angles (dihedral angles in three dimensions) are bounded by $90^\circ$; this corresponds to a \emph{non-obtuse} mesh.


Next, we introduce the following finite element spaces:
\begin{align}
\label{eqn:discrete_spaces}
\begin{split}
  \Yh &:= \{ \phi_h \in H^1(\Omega) : \phi_h |_{K} \text{ is affine for all } K \in \Tauh \}, \\
  \Sph &:= \{ s_h \in H^1(\Omega) : s_h |_{K} \text{ is affine for all } K \in \Tauh \}, \\
  \Uh &:= \{ \bu_h \in H^1(\Omega)^d : \bu_h |_{K} \text{ is affine in
    each component for all } K \in \Tauh \}, \\
  \Nh &:= \{ \bn_h \in \Uh : |\bn_h(\bx_i)| = 1 \text{ for all nodes } \bx_i \in \Nodeh \},\\
  \Vph &:= \{\bv_h \in \Uh: \bv_h(\bx_i) \cdot \bn_h(\bx_i) = 0 \text{ for all nodes } \bx_i \in \Nodeh \},
\end{split}
\end{align}
where $\Nh$ imposes the unit length constraint \emph{at the vertices} of the mesh.  The spaces can be modified to incorporate (Dirichlet) boundary conditions:
\begin{equation}\label{eqn:discrete_spaces_BC}
\begin{split}
  \Sph (\bdys,g_h) &:= \{ s_h \in \Sph : s_h |_{\bdys} = g_h \}, \\
  \Uh (\bdybu,\br_h) &:= \{ \bu_h \in \Uh : \bu_h |_{\bdybu} = \br_h \},
\end{split}
\end{equation}
where $\bdys,\bdybu$ represent subsets of $\partial \Omega$ where Dirichlet conditions are enforced and $g_h = I_hg, \br_h = I_h\br_h$ are the Lagrange interpolations of $(g, \br)$ where $g$ and $\br$ are the traces of some $W^1_\infty(\Omega)$ functions as in \eqref{eq:admis_set_BCs}. With these definitions, we define a discrete admissible class
\begin{align}
\Admerkh(g_h,\br_h) := \{(s_h, \bu_h) \in \Sph (\bdys,g_h) \times \Uh (\bdybu,\br_h) : ~\text{ there exists $\bn_h$ such that \eqref{eq:struct_condition_discrete} holds} \},
\label{eq:admissible-class-discrete}
\end{align}
where
\begin{equation}\label{eq:struct_condition_discrete}
  \buh = I_h(s_h \bnh), \quad - 1/2 < s_h < 1 ~\text{ in } \Omega, \text{ and} \quad \bn_h \in \Nh,
\end{equation}
is the discrete structural condition of $\Admerkh$.  Again, we abuse notation and define $(s_h,\bn_h)$ in $\Admerkh$ to mean $(s_h,\bu_h)$ in $\Admerkh$ with $\bu_h = I_h(s_h \bn_h)$.


The discrete form of the Ericksen energy \eqref{eq:energy-ericksen} is given by \cite{nochetto2017finite}
\begin{equation}
\begin{split}
  \Eerk^h (s_h, \bn_h) := \frac{\kappa}{2} \sum_{i, j = 1}^N k_{ij} \left( s_h(\bx_i) - s_h(\bx_j) \right)^2 + \frac{1}{2} \sum_{i, j = 1}^N k_{ij} \left(\frac{s_h(\bx_i)^2 + s_h(\bx_j)^2}{2}\right) |\bn_h(\bx_i) - \bn_h(\bx_j)|^2,
\end{split}\label{eq:energy-ericksen-discrete}
\end{equation}
for $(s_h,\bu_h) \in \Admerkh(g_h,\br_h)$ where the second term is a first order approximation of $\int_{\Omega} s^2 |\nabla \bn|^2 \, d\bx$.  Note that it can be shown that the first term \emph{equals} $\kappa \int_{\Omega} |\nabla s_h|^2 \, d\bx$.  The discrete energy satisfies a coercivity estimate \cite[Lemma 3.5]{nochetto2017finite} which we now summarize.
\begin{lem}\label{lem:coercivity}
For any $(s_h,\bn_h) \in \Admerkh$, we have
\begin{align*}
  \Eerk^h (s_h, \bnh) \geq &\; \min\{\kappa, 1\}
  \max \left\{\int_{\Omega} |\nabla \buh|^2 dx,
  \int_{\Omega} |\nabla s_h|^2 \, d\bx \right\}, ~ \text{ where } \buh = I_h(s_h \bnh).
\end{align*}
\end{lem}

The form of \eqref{eq:energy-ericksen-discrete} is able to account for the degeneracy in $s_h$ in the limit as $h \rightarrow 0$ without regularization.  Indeed, in \cite{nochetto2017finite}, they proved a $\Gamma$-convergence result for \eqref{eq:energy-ericksen-discrete}, i.e.
\begin{equation}\label{eqn:ericksen-gamma-convergence}
  \Gamma - \lim_{h \rightarrow 0} \Eerk^h (s_h, \bn_h) = \Eerk (s, \bn).
\end{equation}

The Ericksen double well energy, the Cahn-Hilliard energy, and anchoring energy $\Ewas$ are discretized in the standard way:
\begin{align}
  \Edw^h (s_h) &:= \int_{\Omega} \dwfunc (s_h(\bx)) \, d\bx,
  \label{eq:energy-ericksen-dw-discrete}
  \\
  \Ech^h (\phih) &:= \int_{\Omega} \frac{1}{4\varepsilon} \left(\phih^2-1\right)^2 + \frac{\varepsilon}{2} \int_{\Omega} |\nabla \phih |^2 \, d\bx,
  \label{eq:energy-ch-discrete}
  \\
\Ewas^h(s_h,\phih) &= \frac{\varepsilon}{2} \int_{\Omega} |\nabla \phih|^2 (s_h(\bx) - s^*)^2 \, d\bx.
  \label{eq:energy-a,s-anch-discrete}
\end{align}
Finally, the discrete version of the weak anchoring term $\Ewan$ is given by
\begin{equation}
\Ewan(s_h,\bnh,\phih) =\sum_{T_j \subset \Tauh} \int_{T_j} I_h \left\{ (s_h)^2 \, \bnh \cdot [ (\nabla \phih \cdot \nabla \phih) \mathbf{I} - (\nabla \phih \otimes \nabla \phih) ]\bnh \right\},
\label{eq:energy-a,n-anch-discrete}
\end{equation}
where $I_h$ is the Lagrange interpolant. We note that a more detailed definition of the discretization of $\Ewan$ is given in section \ref{sec:gradient-flow}.

The (total) discrete energy is then
\begin{align}
  \Etot^h(s_h,\bnh,\phih) =& \, \Werk \Eerk^h(s_h,\bnh) + \Wdw \Edw^h(s_h) + \Wchdw \Echdw^h(\phih)
  \nonumber
  \\
  &+ \Wchgd \Echgd^h(\phih) + \Wwan \Ewan^h(s_h,\bnh,\phih) + \Wwas \Ewas^h(s_h,\phih).
  \label{eq:energy-total-discrete}
\end{align}
The discretization of time will follow a gradient flow strategy with respect to the total discrete energy \eqref{eq:energy-total-discrete}.

\section{Fully Discrete Gradient Flow Strategy}
\label{sec:gradient-flow}


We use the notation $\iprd{\cdot}{\cdot} : L^2(\Omega) \times L^2(\Omega) \rightarrow \R$ as the standard $L^2$ inner product and the notation $\aiprd{\cdot}{\cdot} : H^1(\Omega) \times H^1(\Omega) \rightarrow \R$ as the $H^1$ inner product such that each may be applied to both scalar and vector valued functions as follows:
\begin{equation*}
  \iprd{u}{v} = \int_{\Omega}  u  v\, d\bx, \quad \iprd{\bn}{\bw} = \int_{\Omega} \bn \cdot \bw\, d\bx, \quad \aiprd{u}{v} = \int_{\Omega} \nabla u \cdot \nabla v\, d\bx, \quad \aiprd{\bn}{\bw} = \int_{\Omega} \nabla \bn : \nabla \bw\, d\bx.
\end{equation*}

Next, we define a multi-linear form representing the discrete Ericksen's energy $\Eerk^h$, as well as its variational derivatives.  Specifically, we define $\eform{\cdot}{\cdot}{\cdot}{\cdot}: \Sph \times \Sph \times \Uh \times \Uh \rightarrow \R$ by
\begin{equation}
  \eform{s_h}{z_h}{\bn_h}{\bw_h} := \sum_{i, j = 1}^N k_{ij} \left(\frac{s_h(\bx_i) z_h(\bx_i) + s_h(\bx_j) z_h(\bx_j)}{2}\right) \left( \bn_h(\bx_i) - \bn_h(\bx_j) \right) \cdot \left( \bw_h(\bx_i) - \bw_h(\bx_j) \right),
  \label{eq:multi-linear-form-ericksen-discrete}
\end{equation}
which is linear in each argument, and note that
\begin{equation*}
  \Eerk^h(s_h,\bn_h) = \kappa \, \aiprd{ s_h }{z_h} + \frac{1}{2} \eform{s_h}{s_h}{\bn_h}{\bn_h}.
\end{equation*}
Furthermore, taking variational derivatives with respect to both $s_h$ and $\bnh$, we have
\begin{align}
  \delta_{\bnh} \Eerk^h [s_h, \bnh ; \bw_h] &= \eform{s_h}{s_h}{\bnh}{\bw_h}, \\
 \delta_{s_h} \Eerk^h [s_h ,\bn_h ; z_h] &= 2  \kappa \, \aiprd{ s_h }{z_h}  + \eform{s_h}{z_h}{\bn_h}{\bn_h}.
  \label{eq:discrete-ericksen-var-deriv}
\end{align}
Additionally, the variational derivative with respect to $s_h$ of the Ericksen double well energy is
\begin{align}
\delta_{s_h} \Edw^h ( s_h; z_h ) = \int_{\Omega} f^\prime(s_h) z_h \, d\bx,
\label{eq:discrete-ericksen-dw-var-deriv}
\end{align}
and the variational derivative with respect to $\phih$ of the Cahn-Hilliard energy is given by
\begin{align}
\delta_{\phih} \Ech^h (\phih; \psih) = \int_{\Omega} \frac{1}{\varepsilon} (\phih^3 - \phih) \psih \, d\bx + \varepsilon \, \aiprd{\phih}{\psih}.
\end{align}

Finally, we define a discrete inner product to capture the discrete coupling energy $\Ewan(s,\bn,\phi)$ in \eqref{eq:energy-anch}, as well as its variational derivatives.  Define the multi-linear form $\cform{\cdot}{\cdot}{\cdot}{\cdot}{\cdot}{\cdot}: \Uh \times P_0 \times \Uh \times P_0 \times \Sh \times \Sh \rightarrow \R$, where $P_0$ is the space of piecewise constant, vector-valued functions such that
\begin{equation}\label{eq:discrete_inner_prod_coupling}
\begin{split}
  \cform{\bv_h}{\nabla \phih}{\bwh}{\nabla \psih}{s_h}{z_h} &:= \\
  \sum_{T_j \subset \Tauh} |T_j| \frac{1}{d+1} &\sum_{i = 1}^{d+1} \left[ s_h z_h \Big((\nabla \phih \cdot \nabla \psih) (\bvh \cdot \bwh) - (\bvh \cdot \nabla \phih)(\bwh \cdot \nabla \psih)\Big) \Big{|}_{T_j} (\hat{\bx}^j_i) \right],
\end{split}
\end{equation}
where $\{ \hat{x}^j_i \}_{i = 1}^{d+1}$ are the vertices of the element $T_j$ in the mesh $\Tauh$; note that we restrict $\nabla \phih$, $\nabla \psih$ to $T_j$ before evaluating at $\bx = \hat{\bx}^j_i$.  Equation \eqref{eq:discrete_inner_prod_coupling} can also be written as
\begin{equation}\label{eq:discrete_inner_prod_coupling_alt}
 \cform{\bv_h}{\nabla \phih}{\bwh}{\nabla \psih}{s_h}{z_h} := \sum_{T_j \subset \Tauh} \int_{T_j} I_h \left\{ (s_h z_h) \bvh \cdot [ (\nabla \phih \cdot \nabla \psih) \mathbf{I} - (\nabla \phih \otimes \nabla \psih) ]\bwh \right\},
\end{equation}
where $I_h$ is the Lagrange interpolant; this follows because the formula in \eqref{eq:discrete_inner_prod_coupling} can be viewed as a quadrature rule that is exact for linear polynomials over each element $T_j$.  The finite element realization of \eqref{eq:discrete_inner_prod_coupling} is a $d \times d$ block matrix, where each block is an $N \times N$ \emph{diagonal} matrix.

Considering these definitions, the discrete anchoring condition can be written as
\begin{align}
\Ewan^h(\bn_h,\phi_h,s_h) &= \frac{\varepsilon}{2} \, \cform{\bnh}{\nabla \phih}{\bnh}{\nabla \phih}{s_h}{s_h}, \\
\Ewas^h(s_h,\phi_h) &= \frac{\varepsilon}{2} \iprd{ \nabla \phi_h (s_h(\bx) - s^*)}{ \nabla \phi_h (s_h(\bx) - s^*)},
\label{eq:energy-anch-discrete}
\end{align}
with the following variational derivatives
\begin{align*}
\delta_{\bnh} \Ewan^h (s_h,\bnh,\phih;\bw_h) &= \varepsilon \, \cform{\bnh}{\nabla \phih}{\bw_h}{\nabla \phih}{s_h}{s_h},
 \\
\delta_{s_h} \Ewan^h (s_h,\bnh,\phih;z_h) &= \varepsilon \,\cform{\bnh}{\nabla \phih}{\bnh}{\nabla \phih}{s_h}{z_h},
\\
\delta_{\phih} \Ewan^h (s_h,\bnh,\phih;\psi_h) &= \varepsilon \, \cform{\bnh}{\nabla \phih}{\bnh}{\nabla \psih}{s_h}{s_h},
\\
\delta_{s_h} \Ewas^h(s_h,\phih;z_h) &= \varepsilon \,\iprd{ \nabla \phi_h (s_h - s^*)}{ \nabla \phi_h z_h },
\\
\delta_{\phih} \Ewas^h(s_h,\phih;\psi_h) &= \varepsilon \, \iprd{ \nabla \phi_h (s_h - s^*)}{ \nabla \psi_h (s_h - s^*)}.
\end{align*}

An important advantage of the inner products $\eform{\cdot}{\cdot}{\cdot}{\cdot}$ and $\cform{\cdot}{\cdot}{\cdot}{\cdot}{\cdot}{\cdot}$ is that they both satisfy a projection property with respect to $\bn_h$. Specifically, we have the following lemma.
\vspace{.1in}

\begin{lem}
\label{lem:monotone_anchoring}
Let $\eform{\cdot}{\cdot}{\cdot}{\cdot}$ be defined by \eqref{eq:multi-linear-form-ericksen-discrete} and $\cform{\cdot}{\cdot}{\cdot}{\cdot}{\cdot}{\cdot}$ be defined by \eqref{eq:discrete_inner_prod_coupling_alt}. If $|\bn_h(\bx_i)| \geq 1$ at all nodes $\bx_i$ in $\Nodeh$, then
\begin{align}
 \eform{s_h}{s_h}{\bnh}{\bnh} &\ge \eform{s_h}{s_h}{\frac{\bnh}{|\bnh|}}{\frac{\bnh}{|\bnh|}},
 \label{eq:eform-monotone-prop}
   \\
\cform{\bnh}{\nabla \phih}{\bnh}{\nabla \phih}{s_h}{s_h} &\ge \cform{\frac{\bnh}{|\bnh|}}{\nabla \phih}{\frac{\bnh}{|\bnh|}}{\nabla \phih}{s_h}{s_h}.
 \label{eq:cform-monotone-prop}
\end{align}
\end{lem}
\begin{proof}
The proof of \eqref{eq:eform-monotone-prop} may be found in \cite{nochetto2017finite}. The proof of \eqref{eq:cform-monotone-prop} follows from Proposition \ref{prop:monotone_lumped_mass} (shown below) and the fact that $\Ewan^h(s_h,\bn_h,\phi_h) = m_h(\bn_h, \bn_h)$ with $H(\bx) = (s_h)^2 [ (\nabla \phih \cdot \nabla \phih) \mathbf{I} - (\nabla \phih \otimes \nabla \phih) ]$.
\end{proof}

\vspace{.1in}
\begin{prop}[monotone property for lumped mass matrix]\label{prop:monotone_lumped_mass}
Let $m_h : \Nh \times \Nh \rightarrow \R$ be a bilinear form defined by
\begin{equation*}
  m_h(\bn_h, \bw_h) = \sum_{T_j \subset \Tauh} \int_{T_j} I_h [ \bn_h \cdot H(\bx) \bw_h ] \, d\bx,
\end{equation*}
where $H$ is a $d \times d$ symmetric positive semi-definite matrix, that is piecewise discontinuous across boundaries of mesh elements but smooth inside each element.  If $|\bn_h(\bx_i)| \geq 1$ at all nodes $\bx_i$ in $\Nodeh$, then
\begin{equation*}
  m_h(\bn_h, \bn_h) \geq m_h \left( \frac{\bn_h}{|\bn_h|}, \frac{\bn_h}{|\bn_h|} \right).
\end{equation*}
\end{prop}
\begin{proof}
Rewrite $m_h(\bn_h, \bw_h)$ as
\begin{equation*}
   m_h(\bn_h, \bw_h) := \sum_{T_j \subset \Tauh} |T_j| \frac{1}{d+1} \sum_{i = 1}^{d+1} \left[ \bn_h(\hat{\bx}^j_i) \cdot H(\hat{\bx}^j_i) \bw_h(\hat{\bx}^j_i) \right].
\end{equation*}
Then, clearly,
\begin{align*}
  m_h(\bn_h, \bn_h) &= \sum_{T_j \subset \Tauh} |T_j| \frac{1}{d+1} \sum_{i = 1}^{d+1} |\bn_h(\hat{\bx}^j_i)|^2 \left[ \frac{\bn_h(\hat{\bx}^j_i)}{|\bn_h(\hat{\bx}^j_i)|} \cdot H(\hat{\bx}^j_i) \frac{\bn_h(\hat{\bx}^j_i)}{|\bn_h(\hat{\bx}^j_i)|} \right] \\
  &\geq \sum_{T_j \subset \Tauh} |T_j| \frac{1}{d+1} \sum_{i = 1}^{d+1} \left[ \frac{\bn_h(\hat{\bx}^j_i)}{|\bn_h(\hat{\bx}^j_i)|} \cdot H(\hat{\bx}^j_i) \frac{\bn_h(\hat{\bx}^j_i)}{|\bn_h(\hat{\bx}^j_i)|} \right] = m_h \left( \frac{\bn_h}{|\bn_h|}, \frac{\bn_h}{|\bn_h|} \right).
\end{align*}
\end{proof}

\section{A Fully Discrete Numerical Scheme}
\label{sec:fully-discrete-scheme}

To set up the numerical scheme presented below, we utilize an $L^2$ gradient flow strategy with respect to the director field and the orientation parameters and an $H^{-1}$ gradient flow strategy with respect to the phase field parameter.  We note that in order to guarantee energy stability, the time discretization is not solely based on a backward Euler method.  Specifically, we use two different convex splittings for the two double well potentials and the anchoring (coupling) terms must be handled appropriately.

\subsection{Scheme}

Let $M$ be a positive integer and $0=t_0 < t_1 < \cdots < t_M = T$ be a uniform partition of $[0,T]$, with $\tau = t_i-t_{i-1}$, $i=1,\ldots ,M$. The fully discrete, finite element scheme is as follows: for any $1\le m\le M$, given  $s_h^{m-1} \in \Sph, \bnh^{m-1}\in \Nh$, and $\phih^{m-1} \in \Yh$, find $s_h^{m} \in \Sph, \bnh^m \in \Nh, \phih^m \in \Yh$, and $\muh^m\in \Yh$,  such that
	\begin{subequations}
	\begin{align}
\rho \, \iprd{\bvh^{m}}{\bwh} + \Werk\, \eform{s^{m-1}_h}{s^{m-1}_h}{\bntilh^{m}}{\bwh}
\nonumber
\\
+ \Wwan \, \varepsilon \, \cform{\bntilh^{m}}{\nabla \phih^{m-1}}{\bwh}{ \nabla \phih^{m-1}}{s^{m-1}_h}{s^{m-1}_h} &= \, 0, & \forall \, \bwh \in \Vph,
	\label{eq:cherk-scheme-bn-discrete}
	\\
\iprd{\dtau s_h^{m}}{z_h} + \Werk \left[ 2 \, \kappa \, \aiprd{s_h^{m}}{z_h} + \eform{s_h^m}{z_h}{\bnh^{m}}{\bnh^{m}} \right] + \Wdw \, \delta_{s_h} \Edw^h ( s_h^{m}; z_h )  \qquad & \nonumber \\
+ \Wwas \, \varepsilon \, \iprd{(s_h^{m} - s^*) \nabla \phih^{m-1}}{z_h \nabla \phih^{m-1}} + \Wwan \, \varepsilon \, \cform{\bnh^m}{\nabla \phih^{m-1}}{\bnh^m}{\nabla \phih^{m-1}}{\shave^m}{z_h} &= \, 0, & \forall \, z_h \in \Sph,
	\label{eq:cherk-scheme-s-discrete}
	\\
\iprd{\dtau \phih^{m}}{\nuh} + \varepsilon \, \aiprd{\muh^{m}}{\nuh} &= \, 0,  & \forall \, \nuh \in \Yh ,
	\label{eq:cherk-scheme-phi-discrete}
	\\
\Wchdw \, \varepsilon^{-1} \,\iprd{\left(\phih^{m}\right)^3 - \phih^{m-1}}{\psih} + \Wchgd \, \varepsilon \,\aiprd{\phih^{m}}{\psih} - \iprd{\muh^{m}}{\psih}  \qquad & \nonumber \\
 + \Wwan \, \varepsilon \, \cform{\bn_h^{m}}{ \nabla \phih^{m}}{\bn_h^{m}}{\nabla \psih}{s^m_h}{s^m_h} + \Wwas \, \varepsilon \, \iprd{(s^{m}_h - s^*) \nabla \phih^{m}}{(s^{m}_h - s^*) \nabla \psih} &= \, 0, & \forall \, \psih \in \Yh,
	\label{eq:cherk-scheme-mu-discrete}
	\end{align}
	\end{subequations}
where $\Vph = \Vph(\bnh^{m-1})$, $\rho > 0$ is a constant, and
\begin{align*}
    \dtau s_h^{m} &:= \frac{s_h^{m} - s_h^{m-1}}{\tau}, \quad \dtau \phih^{m} := \frac{\phih^{m} - \phih^{m-1}}{\tau},
    \\
     \shave^m &:= \frac{s_h^{m}+s^{m-1}_h}{2}, \quad \delta_{s_h} \Edw^h ( s_h^{m}; z_h ) := \int_{\Omega} \left[ f_c'(s_h^{m}) - f_e'(s_h^{m-1}) \right] z_h \, d\bx,
      \\
    \bvh^m &= \dtau \bntilh^m := \frac{\bntilh^{m} - \bnh^{m-1}}{\tau}, \quad \text{and }  \bnh^m (\bx_i) := \frac{\bntilh^m (\bx_i)}{|\bntilh^m (\bx_i) |} \text{ at the nodes } \bx_i,
\end{align*}
such that $\dwfunc_c$, $\dwfunc_e$ are convex functions for all $s \in (-1/2, 1)$ and $\dwfunc (s) = \dwfunc_c(s) - \dwfunc_e(s)$. We note that the order of the method is to first solve \eqref{eq:cherk-scheme-bn-discrete}, normalize to compute $\bnh^m$ and solve \eqref{eq:cherk-scheme-s-discrete}, then solve \eqref{eq:cherk-scheme-phi-discrete} and \eqref{eq:cherk-scheme-mu-discrete}.


Due to the fact that equations \eqref{eq:cherk-scheme-bn-discrete}--\eqref{eq:cherk-scheme-s-discrete} are essentially uncoupled from equations \eqref{eq:cherk-scheme-phi-discrete}--\eqref{eq:cherk-scheme-mu-discrete}, then following similar arguments to what are given in \cite{nochetto2017finite} and \cite{diegel15}, we have the following theorem, which we state without proof:
\vspace{.1in}

	\begin{thm}
For any $1\le m\le M$, the fully discrete scheme \eqref{eq:cherk-scheme-bn-discrete}--\eqref{eq:cherk-scheme-mu-discrete} is uniquely solvable and mass conservative, \emph{i.e.}, $\iprd{\phi_h^m-\phi^0}{1}{} = 0$.
	\end{thm}

The fully-discrete scheme \eqref{eq:cherk-scheme-bn-discrete}--\eqref{eq:cherk-scheme-mu-discrete} obeys the energy law stated below.

	\begin{thm}
	\label{thm-energy-law-full-scheme}
Let $(\phih^{m}, \muh^{m}, \bnh^{m}, s_h^m) \in \Yh \times \Yh \times \Nh \times \Sph$ be the unique solution of \eqref{eq:cherk-scheme-bn-discrete}--\eqref{eq:cherk-scheme-mu-discrete}, for all $1 \leq m \leq M$.  Then the following energy law holds for any $h, \tau >0$:
	\begin{align}
&E^h\left(s_h^\ell, \bnh^{\ell},\phih^{\ell} \right) +\frac{\Werk}{2}\sum_{m=1}^\ell \left(\eform{s_h^{m-1}}{s_h^{m-1}}{\bntilh^m}{\bntilh^m} - \eform{s_h^{m-1}}{s_h^{m-1}}{\bnh^m}{\bnh^m}\right)
\nonumber
\\
&\quad+ \frac{\Wwan}{2} \sum_{m=1}^\ell \Big(\varepsilon\,\cform{\bntilh^{m}}{ \nabla \phih^{m-1}}{\bntilh^{m}}{\nabla \phih^{m-1}}{s_h^{m-1}}{s_h^{m-1}} - \varepsilon\, \cform{\bn_h^{m}}{\nabla \phih^{m-1}}{\bn_h^{m}}{\nabla \phih^{m-1}}{s_h^{m-1}}{s_h^{m-1}} \Big)
\nonumber
\\
&\quad+\tau \sum_{m=1}^\ell \left( \varepsilon \norm{\nabla\muh^{m}}{L^2}^2 + \rho \norm{\dtau \bnh^m}{L^2}^2 + \norm{\dtau s_h^m}{L^2}^2 \right) + \Wchgd \frac{\tau^2}{2} \sum_{m=1}^\ell \varepsilon \norm{\nabla \dtau \phih^m}{L^2}^2
\nonumber
\\
&\quad+ \Wchdw \frac{\tau^2}{2} \sum_{m=1}^\ell \Big( \frac{1}{2\varepsilon} \norm{\dtau( \phih^m)^2}{L^2}^2 + \frac{1}{\varepsilon}  \norm{\phih^m \dtau \phih^m}{L^2}^2 + \frac{1}{\varepsilon}  \norm{\dtau \phih^m}{L^2}^2 \Big)
\nonumber
\\
&\quad+ \Werk\frac{\tau^2}{2} \sum_{m=1}^\ell \Big(2\kappa \norm{\nabla \dtau s_h^m}{L^2}^2 +  \eform{\sh^{m-1}}{\sh^{m-1}}{\dtau \bnh^m}{\dtau \bnh^m} + \eform{\dtau \sh^{m}}{\dtau \sh^{m}}{\bnh^m}{\bnh^m} \Big)
\nonumber
\\
&\quad+ \Wwan \frac{\tau^2}{2} \sum_{m=1}^\ell \Big( \varepsilon \, \cform{\bn_h^{m}}{\nabla \dtau \phih^{m}}{\bn_h^{m}}{ \nabla \dtau \phih^m}{s_h^m}{s_h^m} +  \varepsilon \, \cform{\dtau \bnh^m}{ \nabla \phih^{m-1}}{\dtau \bnh^m}{ \nabla \phih^{m-1}}{s_h^{m-1}}{s_h^{m-1}}
\nonumber
\\
&\quad + \Wwas\frac{\tau^2}{2} \sum_{m=1}^\ell \Big( \varepsilon \, \norm{(s_h^m - s^*) \nabla \dtau \phih^m}{L^2}^2 +  \varepsilon \, \norm{\dtau s_h^m \nabla \phih^m}{L^2}^2 \Big)
\nonumber
\\
&\quad  \le E^h\left(s_h^0, \bnh^0, \phih^0\right),
	\label{eq:ConvSplitEnLaw}
	\end{align}
for all $1 \leq \ell \leq M$ and where we note that $\eform{s_h^{m-1}}{s_h^{m-1}}{\bntilh^m}{\bntilh^m} - \eform{s_h^{m-1}}{s_h^{m-1}}{\bnh^m}{\bnh^m} \ge 0$ and $\cform{\bntilh^{m}}{\nabla \phih^{m-1}}{\bntilh^{m}}{\nabla \phih^{m-1}}{s_h^{m-1}}{s_h^{m-1}} - \cform{\bn_h^{m}}{ \nabla \phih^{m-1}}{\bn_h^{m}}{\nabla \phih^{m-1}}{s_h^{m-1}}{s_h^{m-1}} \ge 0$.  Moreover, the energy is monotonically decreasing, i.e.
\begin{equation*}
  E^h\left(s_h^\ell, \bnh^{\ell},\phih^{\ell} \right) \leq E^h\left(s_h^{\ell-1}, \bnh^{\ell-1},\phih^{\ell-1} \right), \quad \text{for all } ~ 1 \leq \ell \leq M.
\end{equation*}
	\end{thm}

	\begin{proof}
Setting $\bw = \bv_h^{m} = \dtau \bnh^m = \frac{\bntilh^m - \bnh^{m-1}}{\tau}$ in \eqref{eq:cherk-scheme-bn-discrete}, $z_h = \dtau \sh^m$ in \eqref{eq:cherk-scheme-s-discrete}, $\nuh= \muh^{m}$ in \eqref{eq:cherk-scheme-phi-discrete}, and $\psih = \dtau \phih^{m}$ in \eqref{eq:cherk-scheme-mu-discrete}, gives
	\begin{align}
	\rho \, \norm{\dtau \bnh^m}{L^2}^2 + \Werk \, \eform{s^{m-1}_h}{s^{m-1}_h}{\bntilh^{m}}{\dtau \bnh^m} + \Wwan \,\varepsilon\, \cform{\bntilh^{m}}{ \nabla \phih^{m-1}}{\dtau \bnh^m}{ \nabla \phih^{m-1}}{s_h^{m-1}}{s_h^{m-1}} &= \, 0,
	\label{eq:cherk-tested-energy-bn}
	\\
\norm{\dtau s_h^{m}}{L^2}^2 + \Werk \left[ 2 \,\kappa\, \aiprd{s_h^{m}}{\dtau \sh^m} + \eform{s_h^m}{\dtau \sh^m}{\bnh^{m}}{\bnh^{m}} \right] + \Wdw \delta_{s_h} \Edw^h ( s_h^{m}; \dtau \sh^m ) \qquad &
\nonumber
\\
 + \Wwas \, \varepsilon\, \iprd{(s_h^{m} - s^*) \nabla \phih^{m-1}}{\dtau \sh^m \nabla \phih^{m-1}} +  \Wwan \,\varepsilon \, \cform{\bnh^m}{\nabla \phih^{m-1}}{\bnh^m}{\nabla \phih^{m-1}}{\shave^m}{\dtau s_h^m} &= \, 0,
	\label{eq:cherk-tested-energy-s}
	\\
\iprd{\dtau \phih^{m}}{\muh^m} + \varepsilon \norm{\nabla \muh^{m}}{L^2}^2 &= \, 0,
	\label{eq:cherk-tested-energy-phi}
	\\
 \frac{\Wchdw}{4\varepsilon} \,\iprd{\left(\phih^{m}\right)^3 - \phih^{m-1}}{\dtau \phih^m} + \Wchgd \,\varepsilon \,\aiprd{\phih^{m}}{\dtau \phih^m} - \iprd{\muh^{m}}{\dtau \phih^m} \qquad &
\nonumber
\\
+ \Wwan \, \varepsilon \, \cform{\bn_h^{m}}{ \nabla \phih^{m}}{\bn_h^{m}}{ \nabla \dtau \phih^m}{s_h^m}{s^m_h} + \Wwas \, \varepsilon \, \iprd{(s^{m}_h - s^*) \nabla \phih^{m}}{(s^{m}_h - s^*) \nabla \dtau \phih^m} &= \, 0,
	\label{eq:cherk-tested-energy-mu}
	\end{align}

We note that since $\iprd{\cdot}{\cdot}$ and $\aiprd{\cdot}{\cdot}$ are bilinear forms and since $\cform{\cdot}{\cdot}{\cdot}{\cdot}{\cdot}{\cdot}$ and $\eform{\cdot}{\cdot}{\cdot}{\cdot}$ are multi-linear forms, we obtain the following identities:
	\begin{flalign}
\aiprd{\phih^m}{\dtau\phih^m}  =&\, \frac12\,\left[\,\dtau\norm{\nabla\phih^m}{L^2}^2  +\tau \norm{\nabla\dtau \phih^m}{L^2}^2 \,\right],
	\label{eq:identity-phi-dtau}
	\\
\iprd{\left(\phih^m\right)^3-\phih^{m-1}}{\dtau \phih^m} =&\, \frac14\, \dtau \norm{\left( \phih^m\right)^2-1}{L^2}^2
\nonumber
\\
&+\frac{\tau}4 \Bigl[\norm{\dtau( \phih^m)^2}{L^2}^2 +2\norm{\phih^m \dtau \phih^m}{L^2}^2 +2\norm{\dtau \phih^m}{L^2}^2 \, \Bigr]
	\label{eq:identity-nonlinear-first-order}
	\\
 \cform{\bn_h^{m}}{ \nabla \phih^{m}}{\bn_h^{m}}{ \nabla \dtau \phih^m}{s_h^m}{s^m_h} =&\, \frac{1}{2\tau} \Big(\cform{\bn_h^{m}}{ \nabla \phih^{m}}{\bn_h^{m}}{ \nabla \phih^m}{s_h^m}{s_h^m}
 \nonumber
 \\
 &- \cform{\bn_h^{m}}{ \nabla \phih^{m-1}}{\bn_h^{m}}{ \nabla \phih^{m-1}}{s_h^m}{s_h^m} \Big)
 \nonumber
 \\
 &+ \frac{\tau}{2} \cform{\bn_h^{m}}{ \nabla \dtau \phih^{m}}{\bn_h^{m}}{ \nabla \dtau \phih^m}{s_h^m}{s_h^m} ,
	\label{eq:identity-coupled-n-phi-s-1}
	\\
 \cform{\bntilh^{m}}{ \nabla \phih^{m-1}}{\dtau \bnh^m}{ \nabla \phih^{m-1}}{s_h^{m-1}}{s_h^{m-1}} =&\, \frac{1}{2\tau} \Big(\cform{\bntilh^{m}}{ \nabla \phih^{m-1}}{\bntilh^{m}}{ \nabla \phih^{m-1}}{s_h^{m-1}}{s_h^{m-1}}
  \nonumber
 \\
 &-  \cform{\bnh^{m-1}}{ \nabla \phih^{m-1}}{\bnh^{m-1}}{ \nabla  \phih^{m-1}}{s_h^{m-1}}{s_h^{m-1}}  \Big)
 \nonumber
 \\
 &+ \frac{\tau}{2}  \cform{\dtau \bnh^m}{ \nabla \phih^{m-1}}{\dtau \bnh^m}{ \nabla \phih^{m-1}}{s_h^{m-1}}{s_h^{m-1}} ,
	\label{eq:identity-coupled-n-phi-s-2}
\\
\cform{\bnh^m}{\nabla \phih^{m-1}}{\bnh^m}{\nabla \phih^{m-1}}{\shave^m}{\dtau s_h^m} =&\,\frac{1}{2\tau} \Big(\cform{\bnh^m}{\nabla \phih^{m-1}}{\bnh^m}{\nabla \phih^{m-1}}{s_h^{m}}{s_h^{m}}
  \nonumber
 \\
 &-  \cform{\bnh^m}{\nabla \phih^{m-1}}{\bnh^m}{\nabla \phih^{m-1}}{s_h^{m-1}}{s_h^{m-1}} \Big),
 \label{eq:identity-coupled-n-phi-s-3}
\\
	 \iprd{(\sh^m - s^*) \cdot \nabla \phih^{m}}{(\sh^m - s^*)  \cdot \nabla \dtau \phih^{m}} =& \, \frac{1}{2\tau} \norm{(\sh^m - s^*)  \cdot \nabla \phih^{m}}{L^2}^2 - \frac{1}{2\tau} \norm{(\sh^m - s^*)  \cdot \nabla \phih^{m-1}}{L^2}^2
	 \nonumber
	 \\
	 &+ \frac{\tau}{2}  \norm{(\sh^m - s^*)  \cdot \nabla \dtau \phih^{m}}{L^2}^2
	\label{eq:identity-coupled-s-phi-1}
	\\
\iprd{(\sh^m - s^*)  \cdot \nabla \phih^{m-1} }{\dtau \sh^m \cdot \nabla \phih^{m-1}}
=&\, \frac{1}{2\tau}  \norm{(\sh^m - s^*)  \cdot \nabla \phih^{m-1}}{L^2}^2 - \frac{1}{2\tau}  \norm{(\sh^{m-1} - s^*)  \cdot \nabla \phih^{m-1}}{L^2}^2
\nonumber
\\
&+ \frac{\tau}{2}  \norm{\dtau \sh^m \cdot \nabla \phih^{m}}{L^2}^2,
	\label{eq:identity-coupled-s-phi-2}
	\\
	 \eform{\sh^{m-1}}{\sh^{m-1}}{\bntilh^m}{\dtau \bnh^m} =&\, \frac{1}{2\tau}  \left(\eform{\sh^{m-1}}{\sh^{m-1}}{\bntilh^m}{\bntilh^m} - \eform{\sh^{m-1}}{\sh^{m-1}}{\bnh^{m-1}}{\bnh^{m-1}} \right)
	 \nonumber
	 \\
&+ \frac{\tau}{2} \eform{\sh^{m-1}}{\sh^{m-1}}{\dtau \bnh^m}{\dtau \bnh^m},
	 \label{eq:identity-eform-1}
	 \\
 \eform{\sh^{m}}{\dtau \sh^{m}}{\bnh^m}{\bnh^m} =&\, \frac{1}{2\tau}  \left(\eform{\sh^{m}}{\sh^{m}}{\bnh^m}{\bnh^m} - \eform{\sh^{m-1}}{\sh^{m-1}}{\bnh^m}{\bnh^{m}} \right)
 \nonumber
 \\
&+ \frac{\tau}{2} \eform{\dtau \sh^{m}}{\dtau \sh^{m}}{\bnh^m}{\bnh^m} .
	 \label{eq:identity-eform-2}
	\end{flalign}
Additionally, following the procedures supplied in \cite{Wise_SJNA2009, Shen_DCDS2010, Shen_SJSC2010, nochetto2017finite}, we have
\begin{equation}
  \int_{\Omega} \dwfunc(s_h^{k+1}) \, d\bx - \int_{\Omega} \dwfunc(s_h^{k}) \, d\bx \leq \delta_{s_h} \Edw^h ( s_h^{k+1}; s_h^{k+1} - s_h^k ),
  \label{eq:convex_split_inequality}
\end{equation}
for any $s_h^{k}$ and $s_h^{k+1}$ in $\Sph$, Therefore, combining \eqref{eq:cherk-tested-energy-bn} -- \eqref{eq:cherk-tested-energy-mu}, using the identities above, and applying the operator $\tau\sum_{m=1}^\ell$ results in \eqref{eq:ConvSplitEnLaw}.
	\end{proof}

\section{$\Gamma$-Convergence of the Fully Discrete Scheme}
\label{sec:gamma-convergence}

In this section, we show that the total discrete energy \eqref{eq:energy-total-discrete} converges to the total continuous energy \eqref{eq:energy-total} in the $\Gamma$-convergence sense; this is a slightly more general result than \cite[Thm 3.7]{nochetto2017finite} which only shows that global minimizers $\Gamma$-converge.  We require the use of the following proposition whose proof may be found in \cite{nochetto2017finite}.

\vspace{.1in}

\begin{prop}
\label{prop:regularization-s-u}
Let $\Gamma_s=\Gamma_\bu =\partial \Omega, (s,\bu) \in \Admerk(g,\br)$, and let $g$ satisfy \eqref{eq:g_pos}. Then, given $\delta > 0$, there exists a pair $(s_\delta, \bu_\delta) \in \Admerk(g,\br) \cap \left[W_\infty^1(\Omega)\right]^{d+1}$ such that
\begin{equation*}
    \norm{(s,\bu) - (s_\delta, \bu_\delta)}{H^1(\Omega)} \le \delta.
\end{equation*}
Moreover, define $\bn_\delta := \bu_\delta / s_\delta$ if $s_\delta \neq 0$, and any unit vector if $s_\delta = 0$.  Then, $\bn_\delta$ is Lipschitz on $\Omega \setminus \{ |s_\delta| \geq \xi \}$, for any $\xi > 0$, where the Lipschitz constant depends on $\delta$ and $\xi$.
\end{prop}

Furthermore, in order to prove the full $\Gamma$-convergence result in Theorem \ref{thm:Gamma_convergence_main}, we also need the following lemma.
\begin{lem}[Recovery Sequence for Ericksen]\label{lem:recov_seq_Ericksen}
Let $(s,\bu) \in \Admerk(g,\br)$ where $\bu = s \bn$ with $|\bn| = 1$ a.e.  Then there exists a sequence $(s_h,\buh) \in \Admerkh(g_h,\br_h)$ converging to $(s,\bu)$ in $H^1(\Omega)$, as well as $\bnh \in \Nh$ converging to $\bn$ in $L^2(\Omega \setminus \Sing)$, such that
\begin{equation*}
  \Eerk(s,\bn) = \lim_{h \to 0} \Eerk^h(s_h,\bnh).
\end{equation*}
\end{lem}
\begin{proof}
First, note that we can assume $\Eerk (s, \bn) < \infty$ (otherwise, the result is trivial). Recall from \eqref{eq:energy-s-u} that $\Eerk(s,\bn) = \widetilde{\Eerk}(s,\bu)$ when $(s,\bu) \in \Admerk$.  By Proposition \ref{prop:regularization-s-u}, there exists $(s_\delta, \bu_\delta) \in \Admerk(g,\br) \cap \left[W_\infty^1(\Omega)\right]^{d+1}$, such that $\norm{(s,\bu) - (s_\delta, \bu_\delta)}{H^1(\Omega)} \to 0$, as $\delta \to 0$.  Ergo, with $k > 0$ being a given integer, one can choose $\delta_k > 0$ sufficiently small so that
\begin{equation*}
  \norm{(s,\bu) - (s_{\delta_k}, \bu_{\delta_k})}{H^1(\Omega)} < k^{-1}, \quad
  \left| \widetilde{\Eerk}(s_{\delta_k}, \bu_{\delta_k}) - \widetilde{\Eerk}(s,\bu) \right| < C_0 k^{-1},
\end{equation*}
where the constant $C_0 > 0$ depends on $\kappa$ and $\norm{(s,\bu)}{H^1(\Omega)}$; in fact, the last inequality follows from the first.

Next, introduce the Lagrange interpolants $s_{h} := I_{h} ( s_{\delta_k} )$, $\bu_{h} := I_{h} ( \bu_{\delta_k} )$ for some $h$ to be chosen; moreover, define
\begin{equation*}
  \bn_{h}(\bx_i) =
\begin{cases}
  \bu_{h}(\bx_i) / s_{h}(\bx_i), & \mbox{if } s_h(\bx_i) \neq 0 \\
  \text{any unit vector}, & \mbox{otherwise}.
\end{cases}
\end{equation*}
for each $\bx_i \in \Nodeh$.  So, $(s_{h}, \bn_{h}) \in \Admerkh(g_{h},\br_{h})$.  By \cite[Lemma 3.3]{nochetto2017finite}, it was shown that
\begin{equation*}
  \lim_{h \to 0} \Eerk^h (I_h (\hat{s}), I_h (\hat{\bn})) = \Eerk(\hat{s},\hat{\bn}),
\end{equation*}
for all $(\hat{s},\hat{\bu}) \in \Admerk(g,\br) \cap \left[W_\infty^1(\Omega)\right]^{d+1}$, where $\hat{\bn}$ is defined as in Proposition \ref{prop:regularization-s-u}.

Therefore, we can choose $h_k < \delta_k$ sufficiently small so that
\begin{equation*}
  \norm{(s_{\delta_k}, \bu_{\delta_k}) - (s_{h_k}, \bu_{h_k})}{H^1(\Omega)} < k^{-1}, \quad \text{and} \quad
  \left| \Eerk(s_{\delta_k}, \bn_{\delta_k}) - \Eerk^h (s_{h_k}, \bn_{h_k}) \right| < k^{-1}.
\end{equation*}
Combining the above, we obtain $\left| \Eerk^h (s_{h_k}, \bn_{h_k}) - \Eerk (s, \bn) \right| < C_1 k^{-1}$, for some constant $C_1$ that only depends on $\kappa$ and $\norm{(s,\bu)}{H^1(\Omega)}$.
Thus, there exists a sequence $(s_h,\buh) \in \Admerkh(g_h,\br_h)$ converging to $(s,\bu)$ in $H^1(\Omega)$, as well as $\bnh \in \Nh$ converging to $\bn$ in $L^2(\Omega \setminus \Sing)$, such that $\lim_{h \to 0} \Eerk^h ( s_h, \bn_h ) = \Eerk (s, \bn)$.
\end{proof}

We are now in position to prove the main convergence result.  The discrete energy $\Etot^h(s_h,\bnh,\phih)$ is defined on $\Wh := \Sph \times \Nh \times \Yh$, but convergence cannot be insured for a sequence $(s_h,\bnh,\phih) \in \Wh$, because $\bnh$ will not (in general) converge on the singular set $\Sing$.  However, we can guarantee convergence for $(s_h,\buh,\phih) \in \Xh := \Sph \times \Uh \times \Yh$, i.e. $\buh$ is well-behaved.  Thus, Theorem \ref{thm:Gamma_convergence_main} does not follow the standard definition of $\Gamma$-convergence \cite{DalMaso_book1993, Braides_book2002} but is similar; indeed, \emph{one level of indirection} is used in stating the convergence.

To this end, we define the continuous space to be $\X := L^2(\Omega) \times [L^2(\Omega)]^d \times L^2(\Omega)$, and note that $\Xh \subset \X$ \emph{and} $\Wh \subset \X$. Furthermore, we define $\Admall := \Admerk(g,\br) \times H^1(\Omega)$ and $\Admall_h := \Admerkh(g_h,\br_h) \times \Yh$.  Next, the continuous energy $\Etot : \X \to \R$ is defined as follows: $\Etot(s,\bn,\phi)$ by \eqref{eq:energy-total} if $(s,\bn,\phi) \in \Admall$, and set $\Etot(s,\bn,\phi) = \infty$ if $(s,\bn,\phi) \in \X \setminus \Admall$. Likewise, define the discrete energy $\Etot^h(s_h,\bnh,\phih)$ by \eqref{eq:energy-total-discrete} if $(s_h,\bnh,\phih)\in\Admall_h$, and set $\Etot^h(s,\bn,\phi) = \infty$ if $(s,\bn,\phi) \in \X \setminus \Admall_h$.

\medskip

\begin{thm}[$\Gamma$-convergence]\label{thm:Gamma_convergence_main}
Given $(s, \bn, \phi) \in \X$, where $|\bn|=1$ a.e., define the \emph{corresponding element} $(s, \bu, \phi) \in \X$, where $\bu := s \bn$.  In addition, given $(s_h, \bnh, \phih) \in \Wh$, define the \emph{corresponding element} $(s_h, \buh, \phih) \in \Xh$, where $\buh := I_h (s_h \bnh)$.  Let $\{ \Tauh \}$ be a sequence of weakly acute meshes and let $\gamma_0 > 0$ be some arbitrary fixed constant. Then the following properties hold for \emph{any} triple $(s, \bn, \phi)$ in $\X$, where $|\bn|=1$ a.e.~and $-1/2 + \gamma_0 \leq s \leq 1 - \gamma_0$ a.e.
 \begin{itemize}
  \item Lim-inf inequality. For every sequence $(s_h, \bnh, \phih) \in \Wh \subset \X$, such that the corresponding sequence $(s_h, \buh, \phih) \in \Xh \subset \X$ converges strongly to the corresponding triple $(s, \bu, \phi)$, we have
    \begin{align}\label{liminf}
      \Etot (s, \bn, \phi) \leq \liminf_{h \rightarrow 0} \Etot^h (s_h, \bnh, \phih);
    \end{align}
  \item Lim-sup inequality. There exists a sequence $(s_h, \bnh, \phih) \in \Wh \subset \X$ such that the corresponding sequence $(s_h, \buh, \phih) \in \Xh \subset \X$ converges strongly to the corresponding triple $(s, \bu, \phi)$, and
    \begin{align}
      \label{limsup}
      \Etot (s, \bn, \phi) \geq \limsup_{h\rightarrow 0} \Etot^h (s_h, \bnh, \phih).
    \end{align}
\end{itemize}
\end{thm}
\begin{proof}
The proof is split into two parts.

(\underline{Part 1: Lim-inf inequality})

Let $(s_h, \bnh, \phih) \in \Wh$ be any sequence such that its corresponding sequence $(s_h, \buh, \phih) \in \Xh$ converges strongly to $(s, \bu, \phi) \in \X$.  Ergo, by hypothesis, we have
\begin{align*}
\begin{split}
s_h \rightarrow s& \text{ in } L^2(\Omega), \quad \buh \rightarrow \bu \text{ in } L^2(\Omega), \quad \phih \rightarrow \phi \text{ in } L^2(\Omega),
\\
s_h \rightarrow s& \text{ a.e. in } \Omega, \quad \buh \rightarrow \bu \text{ a.e. in } \Omega, \quad \phih \rightarrow \phi \text{ a.e. in } \Omega.
\end{split}
\end{align*}

Without loss of generality, we can assume that $\Etot (s, \bn, \phi) < \infty$; note: this implies that $(s,\bn) \in \Admerk(g,\br)$.
Moreover, we can assume there exists a constant $\Lambda > 0$ such that
\begin{align}
\liminf_{h\rightarrow 0} \Etot^h(s_h, \bnh, \phih) = \liminf_{h\rightarrow 0} \Big(&\Werk \Eerk^h(s_h,\bnh) + \Wdw \Edw^h(s_h) + \Wchdw \Echdw^h(\phih) + \Wchgd \Echgd^h(\phih)
\nonumber
\\
&+ \Wwan \Ewan^h(s_h,\bnh,\phih) + \Wwas \Ewas^h(s_h,\phih) \Big) \le \Lambda;
\label{eq:liminf-assumption}
\end{align}
otherwise, the inequality \eqref{liminf} is trivial.  Assumption \eqref{eq:liminf-assumption} also implies that $(s_h,\bn_h) \in \Admerkh(g_h,\br_h)$ for $h$ sufficiently small.
Combining \eqref{eq:liminf-assumption} with Lemma \ref{lem:coercivity} (coercivity) gives the following weakly convergent subsequences (not relabeled):
\begin{equation*}
   s_h \rightharpoonup s \text{ in } H^1(\Omega), \quad \buh \rightharpoonup \bu \text{ in } H^1(\Omega), \quad \phih \rightharpoonup \phi \text{ in } H^1(\Omega).
\end{equation*}

Note: if $\Etot (s, \bn, \phi) = \infty$, then either $(s,\bn) \notin \Admerk(g,\br)$ or $\phi \notin H^1(\Omega)$. In the later case, clearly $\liminf_{h\rightarrow 0} \Etot^h(s_h, \bnh, \phih) = \infty$, which contradicts \eqref{eq:liminf-assumption}.  For the former, either  $s \notin H^1(\Omega)$ or $\bu \notin [H^1(\Omega)]^d$.  Again, this implies $\liminf_{h\rightarrow 0} \Etot^h(s_h, \bnh, \phih) = \infty$, which contradicts \eqref{eq:liminf-assumption}.  Therefore, if $\Etot (s, \bn, \phi) = \infty$, then the inequality \eqref{liminf} is trivial.

Using Fatou's Lemma, one can show that $\Edw(s) \leq \liminf_{h\rightarrow 0} \Edw^h(s_h)$.  In \cite{nochetto2017finite}, the following technical result was proved:
$\Eerk(s,\bn) \leq \liminf_{h\rightarrow 0} \Eerk^h(s_h,\bnh)$;
so we do not repeat the argument here.  We now consider the remaining terms. By weak lower semi-continuity, we have
\begin{align*}
\Echgd(\phi) = \frac{\varepsilon}{2}\int_\Omega |\nabla \phi |^2 \le \liminf_{h\rightarrow 0} \frac{\varepsilon}{2}\int_\Omega |\nabla \phih |^2 = \liminf_{h\rightarrow 0} \Echgd^h(\phih).
\end{align*}
Additionally, using the compact Sobolev embedding  $H^1(\Omega) \hookrightarrow L^4(\Omega)$, for $d = 2, 3$, there exists a subsequence $\{\phih\}$ (not relabeled) such that $\phih \rightarrow \phi$ in $L^4(\Omega)$. The lim-inf inequality relating to $\Echdw(\phi)$ then follows from Fatou's Lemma:
\begin{align*}
\Echdw(\phi) = \frac{1}{4\varepsilon} \int_\Omega (\phi^2 -1)^2 \le \liminf_{h\rightarrow 0} \frac{1}{4\varepsilon} \int_\Omega (\phih^2 -1)^2 = \liminf_{h\rightarrow 0} \Echdw^h(\phih).
\end{align*}

For the anchoring energy $\Ewas^h(s_h,\phih)$, we split the integral into two parts by adding and subtracting appropriate terms as follows:
\begin{align*}
\Ewas^h(s_h,\phih) &= \frac{\varepsilon}{2} \int_\Omega |\nabla \phih|^2 (s_h - s^*)^2
\\
&=  \frac{\varepsilon}{2} \int_\Omega |\nabla \phih|^2 (s - s^*)^2 +  \frac{\varepsilon}{2} \int_\Omega |\nabla \phih|^2 \left[(s_h - s^*)^2 - (s - s^*)^2 \right].
\end{align*}
By Egorov's Theorem, given $\delta> 0$, there exists a subset $\Adelta \subset \Omega$ such that $(s_h - s^*)^2 \rightarrow (s - s^*)^2$ uniformly on $\Adelta$ and $\left| \Omega \setminus \Adelta \right| \le \delta$. Hence,
\begin{align*}
\lim_{h\rightarrow0} \left| \int_{\Adelta} |\nabla \phih|^2 \left[(s_h - s^*)^2 - (s - s^*)^2 \right] \right| &\le \lim_{h\rightarrow0} \norm{(s_h - s^*)^2 - (s - s^*)^2}{L^\infty(\Adelta)} \int_{\Adelta} |\nabla \phih |^2
\\
&\le \lim_{h\rightarrow0} \norm{(s_h - s^*)^2 - (s - s^*)^2}{L^\infty(\Adelta)} \int_\Omega |\nabla \phih |^2
\\
&\le C\Lambda \lim_{h\rightarrow0} \norm{(s_h - s^*)^2 - (s - s^*)^2}{L^\infty(\Adelta)}
\\
&=0.
\end{align*}
Thus,
\begin{align*}
 \liminf_{h\rightarrow0} \int_{\Omega}  |\nabla \phih|^2  (s_h - s^*)^2 &\ge \liminf_{h\rightarrow0} \int_{\Adelta} |\nabla \phih|^2 (s_h - s^*)^2
 \\
 &= \liminf_{h\rightarrow0} \int_{\Adelta} |\nabla \phih|^2 (s - s^*)^2 + \liminf_{h\rightarrow0} \int_{\Adelta} |\nabla \phih|^2 \left[(s_h - s^*)^2 - (s - s^*)^2 \right]
\\
&\ge \liminf_{h\rightarrow0} \int_{\Adelta} |\nabla \phih|^2 (s - s^*)^2
\\
&\ge \int_{\Adelta} |\nabla \phi|^2 (s - s^*)^2
\end{align*}
for all $\delta > 0$, where we have used weak lower semi-continuity \cite{evans1988partial}. Using Lebesgue's dominated convergence theorem and allowing $\delta \rightarrow 0$ gives the desired result.

To show the lim-inf inequality for the weak anchoring energy $\Ewan$, we begin by noting that, by using the same notation defined in Section \ref{sec:contin_energy_models} and the auxiliary variable $\bu := s\bn$, the weak anchoring energy $\Ewan$ can be rewritten as:
\begin{align}
\Ewan(s,\bn,\phi) = \Ewan(\bu,\phi)= \frac{\varepsilon}{2} \int_\Omega |\bu|^2 |\nabla \phi|^2 - (\bu \cdot \nabla \phi)^2.
\label{eq:energy-an-with-u}
\end{align}
Furthermore, we consider the discrete weak anchoring energy in the form of \eqref{eq:discrete_inner_prod_coupling_alt} and note the following equivalences:
\begin{align}
\Ewan^h(s_h,\bnh,\phih)&:= \frac{\varepsilon}{2}\sum\limits_{T_j \subset \Tauh} \int\limits_{T_j} I_h \left\{ s_h^2 \, \bnh \cdot [ (\nabla \phih \cdot \nabla \phih) \mathbf{I} - (\nabla \phih \otimes \nabla \phih) ]\bnh \right\}
\nonumber
\\
&= \frac{\varepsilon}{2} \sum\limits_{T_j \subset \Tauh} \int\limits_{T_j}  I_h \left\{ |\buh|^2 |\nabla \phih|^2\right\} - I_h \left\{(\nabla \phih \cdot \buh)^2 \right\} =: \Ewan^h(\buh,\phih).
\label{eq:energy-an-with-u_h}
\end{align}

By interpolation theory, we have
\begin{align*}
 \norm{I_h \left\{ |\buh|^2 \right\} - |\buh|^2}{L^2(\Omega)} \le C h \norm{\nabla \buh |\buh|}{L^2(\Omega)} \le C h \norm{\nabla \buh }{L^2(\Omega)} \le C \Lambda^{1/2} h.
\end{align*}
Similarly,
\begin{align*}
\norm{I_h \left\{ \buh \otimes \buh \right\} - \buh \otimes \buh}{L^2} &\le C \Lambda^{1/2} h.
\end{align*}
Therefore, since $\buh \to \bu$ in $L^2(\Omega)$, we have the following convergence results,
\begin{align*}
\left| I_h \left\{ |\buh|^2 \right\} - |\bu|^2 \right| \rightarrow 0, \text{ in } L^2(\Omega), &\quad \left| I_h \left\{ |\buh|^2 \right\} - |\bu|^2 \right| \rightarrow 0, ~ a.e. \text{ in } \Omega,
\\
\left| I_h \left\{ \buh \otimes \buh \right\} - \bu \otimes \bu \right| \rightarrow 0, \text{ in } L^2(\Omega), &\quad \left| I_h \left\{ \buh \otimes \buh \right\} - \bu \otimes \bu \right| \rightarrow 0, ~ a.e. \text{ in } \Omega.
\end{align*}

Due to the fact that $\nabla \phih$ is constant on each element, the discrete energy $\Ewan^h$ can be written as follows:
\begin{align*}
\Ewan^h(\buh,\phih) &= \frac{\varepsilon}{2} \int_\Omega  |\nabla \phih|^2 I_h \left\{|\buh|^2 \right\} - \nabla \phih \cdot \left(I_h \left\{ \buh \otimes \buh\right\} \right) \nabla \phih
\\
&= \frac{\varepsilon}{2} \int_\Omega |\nabla \phih|^2 |\bu|^2 - (\nabla \phih \cdot \bu)^2 + \frac{\varepsilon}{2} \int_\Omega |\nabla \phih|^2 \left[I_h \left\{|\buh|^2\right\} - |\bu|^2 \right]
\\
&\quad -  \frac{\varepsilon}{2} \int_\Omega  \nabla \phih \cdot \left[I_h \left\{ \buh \otimes \buh \right\} - \bu \otimes \bu\right] \nabla \phih.
\end{align*}
By Egorov's Theorem, given $\delta> 0$, there exists a subset $\Adelta \subset \Omega$ such that $I_h \{|\buh|^2 \} \rightarrow |\bu|^2$ uniformly on $\Adelta$ and $\left| \Omega \setminus \Adelta \right| \le \delta$. Hence,
\begin{align*}
\lim_{h\rightarrow0} \left| \int_{\Adelta} |\nabla \phih|^2 \left[ I_h \left\{ |\buh|^2 \right\}- |\bu|^2 \right] \right| &\le \lim_{h\rightarrow0} \norm{ I_h \left\{ |\buh|^2 \right\} - |\bu|^2}{L^\infty(\Adelta)} \int_{\Adelta} |\nabla \phih |^2
\\
&\le \lim_{h\rightarrow0} \norm{ I_h \left\{ |\buh|^2 \right\} - |\bu|^2}{L^\infty(\Adelta)} \int_\Omega |\nabla \phih |^2
\\
&\le C\Lambda \lim_{h\rightarrow0} \norm{ I_h \left\{ |\buh|^2 \right\} - |\bu|^2}{L^\infty(\Adelta)}
\\
&=0.
\end{align*}
Similarly, there exists a subset $\tilde{\Adelta} \subset \Omega$ such that $I_h \left\{ \buh \otimes \buh\right\} \rightarrow \bu \otimes \bu$ uniformly on $\tilde{\Adelta}$ and $| \Omega \setminus \tilde{\Adelta} | \le \delta$. Hence
\begin{align*}
\lim_{h\rightarrow0} \left| \int_{\tilde{\Adelta}} \nabla \phih \cdot \left[ \bu \otimes \bu -  I_h \left\{ \buh \otimes \buh \right\} \right]\nabla \phih \right| &\le \lim_{h\rightarrow0} \norm{\bu \otimes \bu -  I_h \left\{  \buh \otimes \buh\right\}}{L^\infty(\tilde{\Adelta})} \int_{\tilde{\Adelta}} |\nabla \phih |^2
\\
&\le \lim_{h\rightarrow0} \norm{\bu \otimes \bu -  I_h \left\{ \buh \otimes \buh\right\}}{L^\infty(\tilde{\Adelta})} \int_\Omega |\nabla \phih |^2
\\
&\le C\Lambda \lim_{h\rightarrow0} \norm{\bu \otimes \bu -  I_h \left\{ \buh \otimes \buh\right\}}{L^\infty(\tilde{\Adelta})}
\\
&=0.
\end{align*}
Let $B_\delta := \Adelta \cap \tilde{\Adelta}$. Then $|\Omega \setminus B_\delta| =  |(\Omega \setminus \Adelta) \cup (\Omega \setminus \tilde{\Adelta})| \le 2\delta$. Hence,
\begin{align*}
\liminf_{h\rightarrow0} \int_{\Omega} I_h \left\{ |\buh|^2 |\nabla \phih|^2 - (\nabla \phih \cdot \buh)^2 \right\}&\ge \liminf_{h\rightarrow0} \int_{B_\delta}I_h \left\{|\buh|^2 |\nabla \phih|^2 - (\nabla \phih \cdot \buh)^2 \right\}
\\
=&\, \liminf_{h\rightarrow0} \int_{B_\delta} |\nabla \phih|^2 |\bu|^2 - (\nabla \phih \cdot \bu)^2
\\
&+ \liminf_{h\rightarrow0} \int_{B_\delta} |\nabla \phih|^2 \left[I_h \left\{|\buh|^2\right\} - |\bu|^2 \right]
\\
&-\liminf_{h\rightarrow0} \int_{B_\delta}  \nabla \phih \cdot \left[ I_h \left\{\buh \otimes \buh \right\} - \bu \otimes \bu\right] \nabla \phih
\\
\ge&\, \liminf_{h\rightarrow0} \int_{B_\delta} |\nabla \phih|^2 |\bu|^2 - (\nabla \phih \cdot \bu)^2
\\
\ge& \int_{B_\delta} |\nabla \phi|^2 |\bu|^2 - (\nabla \phi \cdot \bu)^2,
\end{align*}
for all $\delta > 0$, where we have used weak lower semi-continuity \cite{evans1988partial}. Using Lebesgue's dominated convergence theorem and allowing $\delta \rightarrow 0$ gives the desired result.

(\underline{Part 2: Lim-sup inequality})

For the lim-sup inequality, we will construct a sequence that verifies the inequality \eqref{limsup}.  Indeed, we will actually show equality with a limit.

Invoking Lemma \ref{lem:recov_seq_Ericksen}, there exists sequences $(s_h,\buh) \in \Admerkh(g_h,\br_h)$ and $\bnh \in \Nh$ such that
\begin{equation*}
  \norm{(s_h,\buh) - (s,\bu)}{H^1(\Omega)} \to 0, \quad \norm{\bnh - \bn}{L^2(\Omega \setminus \Sing)} \to 0,
\end{equation*}
and $\Eerk(s,\bn) = \lim_{h \to 0} \Eerk^h(s_h,\bnh)$.  For the Ericksen double-well, $\Edw(s)$, since $-1/2 + \gamma_0 \leq s \leq 1 - \gamma_0$, $|\dwfunc (s(\bx))| \leq M$ for a.e.~$\bx \in \Omega$ for some positive constant $M$ (recall Section \ref{sec:Erk_energy}).  Thus, by Lebesgue's dominated convergence theorem, we have $\lim_{h \to 0} \Edw^h ( s_h ) = \Edw (s)$.  We shall use the sequence $(s_h,\buh)$ below to prove convergence of the weak anchoring terms.

For the phase variable $\phi \in H^1(\Omega)$, we let $\phih$ be the elliptic projection of $\phi$, i.e. $\phih$ solves
\begin{equation*}
  (\nabla \phih, \nabla \psih) = (\nabla \phi, \nabla \psih), ~~ \text{for all } \psih \in \Yh,
\end{equation*}
which implies that $\norm{\phih - \phi}{H^1(\Omega)} \to 0$.  Considering
\begin{align*}
\Ech(\phi) - \Ech(\phih) = \int_\Omega \frac{\Wchdw}{4\varepsilon} \left[(\phi^2 -1)^2 - (\phih^2 -1)^2 \right] + \frac{\Wchgd \, \varepsilon}{2} \left[ |\nabla \phi|^2 -  |\nabla \phih|^2\right],
\end{align*}
we see that $\Ech(\phih) \to \Ech(\phi)$, where we used the Sobolev embedding $H^1(\Omega) \hookrightarrow L^4(\Omega)$, for $d = 2, 3$,

Next, since $(s_h - s^*)^2 \rightarrow  (s - s^*)^2$ a.e.~in $\Omega$, and $s$ is bounded a.e.~in $\Omega$, then $|\nabla \phi|^2 (s_h - s^*)^2 \rightarrow |\nabla \phi|^2 (s - s^*)^2$ a.e.~in $\Omega$. So, by Lebesgue's Dominated Convergence theorem, we have
\begin{align*}
\lim_{h \rightarrow 0} \left| \Ewas(s,\phi) - \Ewas(s_h,\phih) \right| &= \lim_{h \rightarrow 0} \frac{\varepsilon}{2}\int_\Omega \left||\nabla \phi |^2 - |\nabla \phih|^2  \right| (s_h - s^*)^2 + \frac{\varepsilon}{2} \int_\Omega |\nabla \phi |^2 \left| (s-s^*)^2 - (s_h - s^*)^2 \right|
\\
&\le \lim_{h \rightarrow 0} \frac{\varepsilon}{2}\int_\Omega \left||\nabla \phi |^2 - |\nabla \phih|^2  \right| (s_h - s^*)^2
\\
&\quad+ \lim_{h \rightarrow 0} \frac{\varepsilon}{2} \int_\Omega |\nabla \phi |^2 \left| (s-s^*)^2 - (s_h - s^*)^2 \right|
\\
&\le \lim_{h \rightarrow 0}  \frac{C \varepsilon}{2} \int_\Omega \left| |\nabla \phi |^2 - |\nabla \phih|^2  \right| + 0 = 0.
\end{align*}
Similarly, we find that
\begin{align*}
\lim_{h \rightarrow 0} \left| \Ewan(\bu,\phi) - \Ewan^h(\buh,\phih) \right| \le&\, \lim_{h \rightarrow 0} \frac{\varepsilon}{2} \int_\Omega \left| | \nabla \phi |^2 - |\nabla \phih|^2  \right| |\bu|^2
\\
&+ \lim_{h \rightarrow 0} \frac{\varepsilon}{2} \int_\Omega |\nabla \phih|^2 \left| |\bu|^2 - I_h\{|\buh|^2\}  \right|
\\
&+ \lim_{h \rightarrow 0} \frac{\varepsilon}{2} \int_\Omega \left| \nabla \phi  + \nabla \phih \right| |\bu \otimes \bu| \left| \nabla \phi  - \nabla \phih \right|
\\
&+ \lim_{h \rightarrow 0} \frac{\varepsilon}{2} \int_\Omega |\nabla \phih| \left| \bu \otimes \bu - I_h\{\buh \otimes \buh\} \right| | \nabla \phih | = 0,
\end{align*}
where we used the earlier results: $\norm{ I_h \{ |\buh|^2 \} - |\buh|^2 }{L^2(\Omega)} = O(h)$, $\norm{ I_h \{ \buh \otimes \buh \} - \buh \otimes \buh }{L^2(\Omega)} = O(h)$.
\end{proof}

We now obtain the following corollary about convergence of global minimizers \cite{Braides_book2014, DalMaso_book1993}.
\begin{cor}[convergence of global discrete minimizers]\label{cor:convergence_global_mins}
Let $\{ \Tauh \}$ be a sequence of weakly acute meshes.
If $(s_h,\bnh,\phih) \in \Admerkh(g_h,\brh)$ is a sequence of global minimizers of $\Etot^h (s_h, \bnh, \phih)$ in \eqref{eq:energy-total-discrete}, then
every cluster point is a global minimizer of the continuous energy $\Etot(s , \bn, \phi)$ in \eqref{eq:energy-total}.
\end{cor}
\begin{proof}
First note that, because of the form of the energy (both continuous and discrete), we can always truncate $s$ and $s_h$ with the function
\begin{equation*}
  \Theta(f) := \max \{ -1/2 + \gamma_0, \min \{ 1 - \gamma_0, f \} \},
\end{equation*}
for some fixed constant $\gamma_0 > 0$ sufficiently small.  Since the boundary condition $g$ (for $s$) is bounded away from $-1/2$ and $1$ (recall \eqref{eq:struct_condition} and \eqref{eq:g_pos}), one can show that
\begin{equation*}
\begin{split}
  \Eerk(\Theta(s),\bn) &\leq \Eerk(s,\bn), \quad \Edw(\Theta(s)) \leq \Edw(s), \\
  \Ewan(\Theta(s),\bn,\phi) &\leq \Ewan(s,\bn,\phi), \quad \Ewas(\Theta(s),\phi) \leq \Ewas(s,\phi),
\end{split}
\end{equation*}
where we use the fact that $(\Theta (s) - s^*)^2 \leq (s - s^*)^2$ provided $s^*$ is bounded away from $-1/2$ and $1$.  The same holds for the discrete energies as well.  Thus, without loss of generality, we assume the discrete minimizers obey $-1/2 + \gamma_0 \leq s_h \leq 1 - \gamma_0$ for $\gamma_0$ sufficiently small.

Next, we take $\Etot^h (s_h, \bnh, \phih) \leq \Lambda$ for all $h > 0$, where $0 < \Lambda < \infty$ is a fixed constant.  Using \cite[Lem 3.6]{nochetto2017finite} we obtain convergent subsequences $\{\phih\}, \{s_h\}, \{\buh\}$ (not relabeled) such that
\begin{align*}
\begin{split}
s_h \rightharpoonup s& \text{ in } H^1(\Omega), \quad \buh \rightharpoonup \bu \text{ in } H^1(\Omega), \quad \phih \rightharpoonup \phi \text{ in } H^1(\Omega),
\\
s_h \rightarrow s& \text{ in } L^2(\Omega), \quad \buh \rightarrow \bu \text{ in } L^2(\Omega), \quad \phih \rightarrow \phi \text{ in } L^2(\Omega),
\\
s_h \rightarrow s& \text{ a.e. in } \Omega, \quad \buh \rightarrow \bu \text{ a.e. in } \Omega, \quad \phih \rightarrow \phi \text{ a.e. in } \Omega.
\end{split}
\end{align*}
Moreover, \cite[Lem 3.6]{nochetto2017finite} implies there is a subsequence $\{ \bnh \}$ (not relabeled), and $\bn \in L^2(\Omega)$ with $|\bn|=1$ a.e., such that $\Nh \ni \bnh \rightarrow \bn$ in $L^2(\Omega \setminus \Sing)$, $\bnh \rightarrow \bn$ a.e.~in $\Omega \setminus \Sing$, and $\bu = s \bn$ a.e.~in $\Omega$.  Thus, $(s,\bu) \in \Admerk(g,\br)$.  So the subsequence $(s_h, \buh, \phih)$ of minimizers converges to a limit in $\X$.

Therefore, $(s,\bn,\phi)$ and the corresponding $(s,\bu,\phi)$ satisfies the conditions of Theorem \ref{thm:Gamma_convergence_main}, so we obtain that $\Etot (s, \bn, \phi) \leq \liminf_{h \rightarrow 0} \Etot^h (s_h, \bnh, \phih)$.  Moreover, there exists a sequence $\{ (\tilde{s}_h, \tilde{\bn}_h, \tilde{\phi}_h) \}$, and corresponding sequence $\{ (\tilde{s}_h, \tilde{\bu}_h, \tilde{\phi}_h) \}$ such that $(\tilde{s}_h, \tilde{\bu}_h, \tilde{\phi}_h) \to (s,\bu,\phi)$ in $\X$, and
\begin{equation*}
  \Etot (s, \bn, \phi) \leq \liminf_{h \rightarrow 0} \Etot^h (s_h, \bnh, \phih) \leq \limsup_{h\rightarrow 0} \Etot^h (\tilde{s}_h, \tilde{\bn}_h, \tilde{\phi}_h) \leq \Etot (s, \bn, \phi).
\end{equation*}
Hence, $\Etot (s, \bn, \phi) = \lim_{h \rightarrow 0} \Etot^h (s_h, \bnh, \phih)$, i.e. the limit of discrete global minimizers is a global minimizer.
\end{proof}
Note: this convergence result does not yield a rate of convergence, though first order is expected in most situations (see \cite{Nochetto_JCP2017} for an example).

\section{Numerical Experiments}
\label{sec:numerical-experiments}

In the experiments to follow, we use a square domain $\Omega = (0,1)^2\subset \mathbb{R}^2$ and take ${\mathcal T}_h$ to be a regular triangulation of $\Omega$ consisting of right isosceles triangles. (We note that the analysis presented in the previous sections holds for both $d=2$ and $d=3$.) To solve the linear systems in \eqref{eq:cherk-scheme-bn-discrete}--\eqref{eq:cherk-scheme-s-discrete}, we used MATLAB's ``backslash'' command and used a standard Newton's method algorithm to solve the system \eqref{eq:cherk-scheme-phi-discrete}--\eqref{eq:cherk-scheme-mu-discrete} with a tolerance of $10^{-15}$ or a residual tolerance of $10^{-7}$, whichever is satisfied first. To solve the linear system within Newton's method, we again used MATLAB's ``backslash'' command. In each experiment, the double well potential related to the orientation parameter is defined as $f(s) = f_c(s) - f_e(s) = 63.0s^2 - (-16.0s^4 + 21.33333333333s^3 + 57.0s^2)$ with $s^* = 0.750025$. All computations are completed using the FELICITY MATLAB/C++ Toolbox \cite{FELICITY_REF}.

\subsection{Movement of a Liquid Crystal Droplet}
\label{sec:move-LC-droplet}

The first numerical experiment demonstrates the movement of a liquid crystal droplet. The initial conditions are as follows:
\begin{align*}
s_h^0 &= s^*,
\\
\bnh^0 &= \frac{(x,y) - (0.26,0.25)}{|(x,y) - (0.26,0.25)|},
\\
\phih^0 &= I_h\left\{-\tanh\left(\frac{(x-0.25)^2 / 0.02 + (y-0.25)^2 / 0.02 - 1}{2\varepsilon}\right)\right\}.
\end{align*}
The following Dirichlet boundary conditions on $\partial \Omega$ are imposed for $s$ and $\bn$:
\begin{align*}
s = s^*, \quad \bnh &= \frac{(x,y) - (0.85,0.85)}{|(x,y) - (0.85,0.85)|}.
\end{align*}
The relevant parameters are $\kappa = 1, \rho = 1, \Werk =  1, \Wdw = 100, \Wchdw = 1, \Wchgd = 1 + \Wwan + \Wwas = 41, \Wwas = 20, \Wwan = 20$. The space step size is taken to be $h = \sqrt{2}/64$ and the time step size is taken to be $\tau = 0.002$ with a final stopping time of $T=20.0$. The interfacial width parameter is taken to be $\varepsilon = 3h / \sqrt{2}$. Figure \ref{fig:droplet-moving} shows the evolution of the droplet over time. The top two rows display the evolution of the scalar degree of orientation parameter $s$. The bottom two rows show the evolution of the phase field parameter $\phi$ and the director field $\bn$.  This example shows that the droplet position can be manipulated by choosing appropriate boundary conditions.

\begin{figure}
\subfloat{\includegraphics[width = 2in]{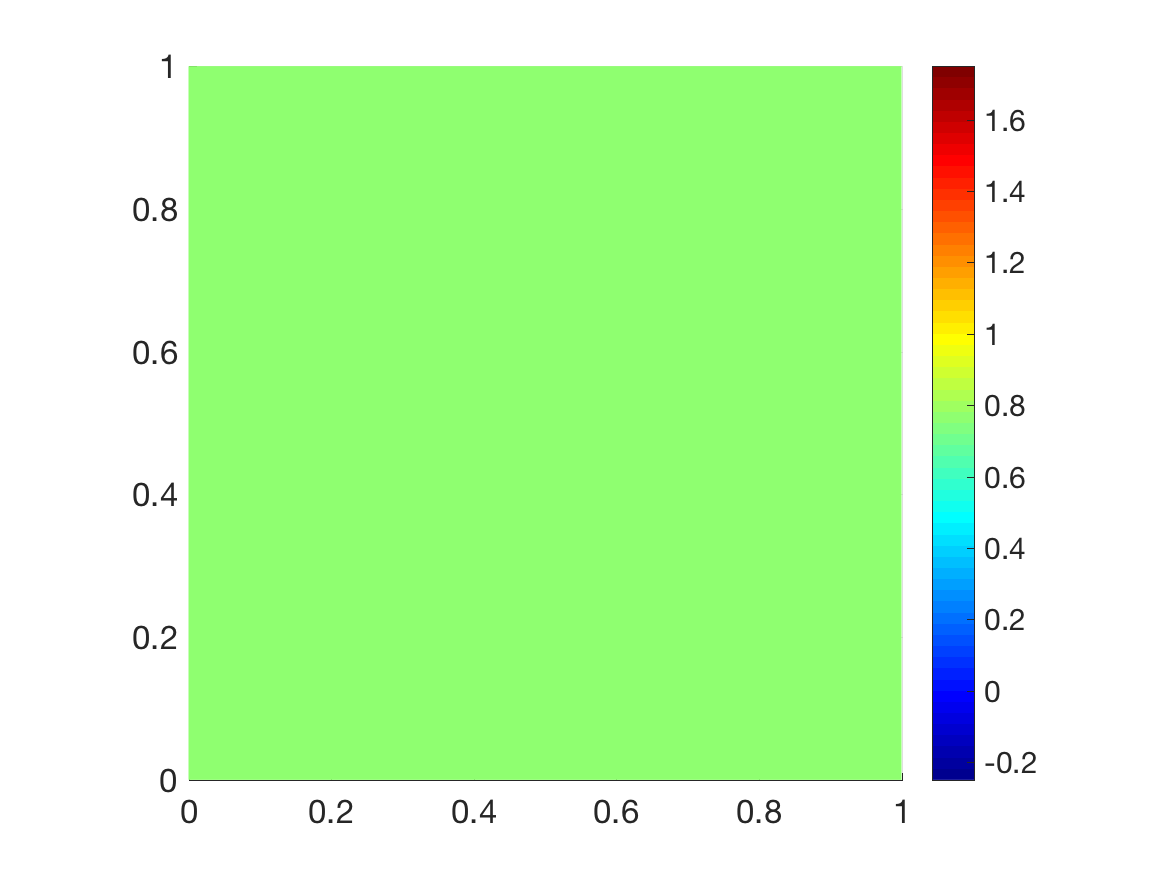}} \hspace{-2em}
\subfloat{\includegraphics[width = 2in]{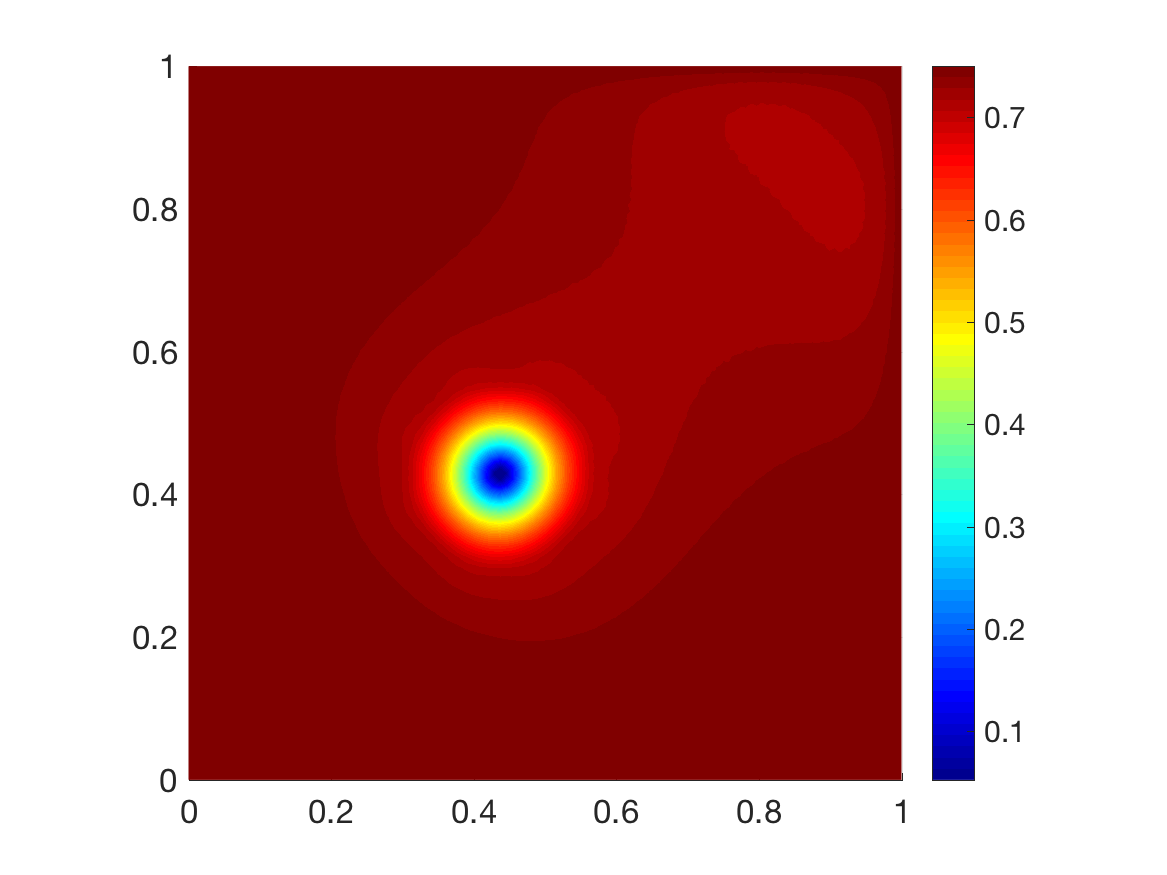}} \hspace{-2em}
\subfloat{\includegraphics[width = 2in]{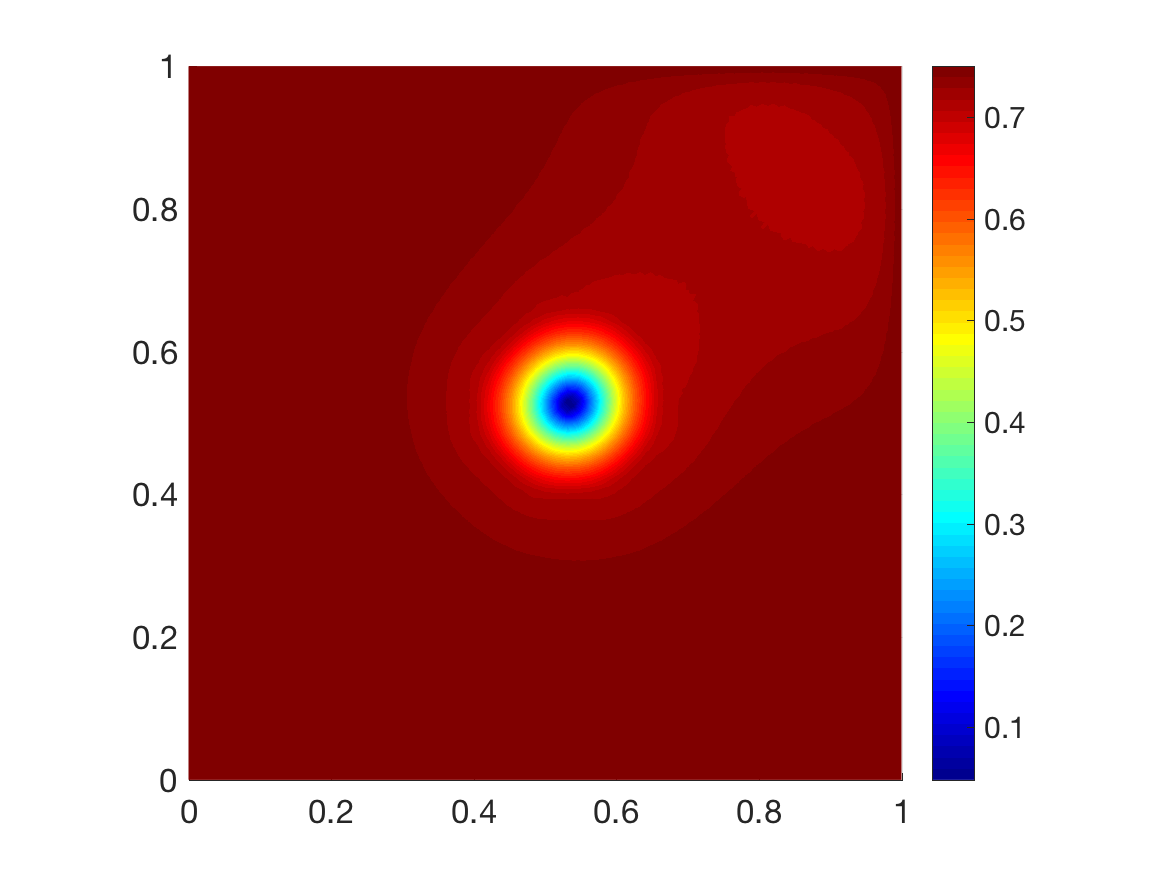}} \\[-4ex]
\subfloat{\includegraphics[width = 2in]{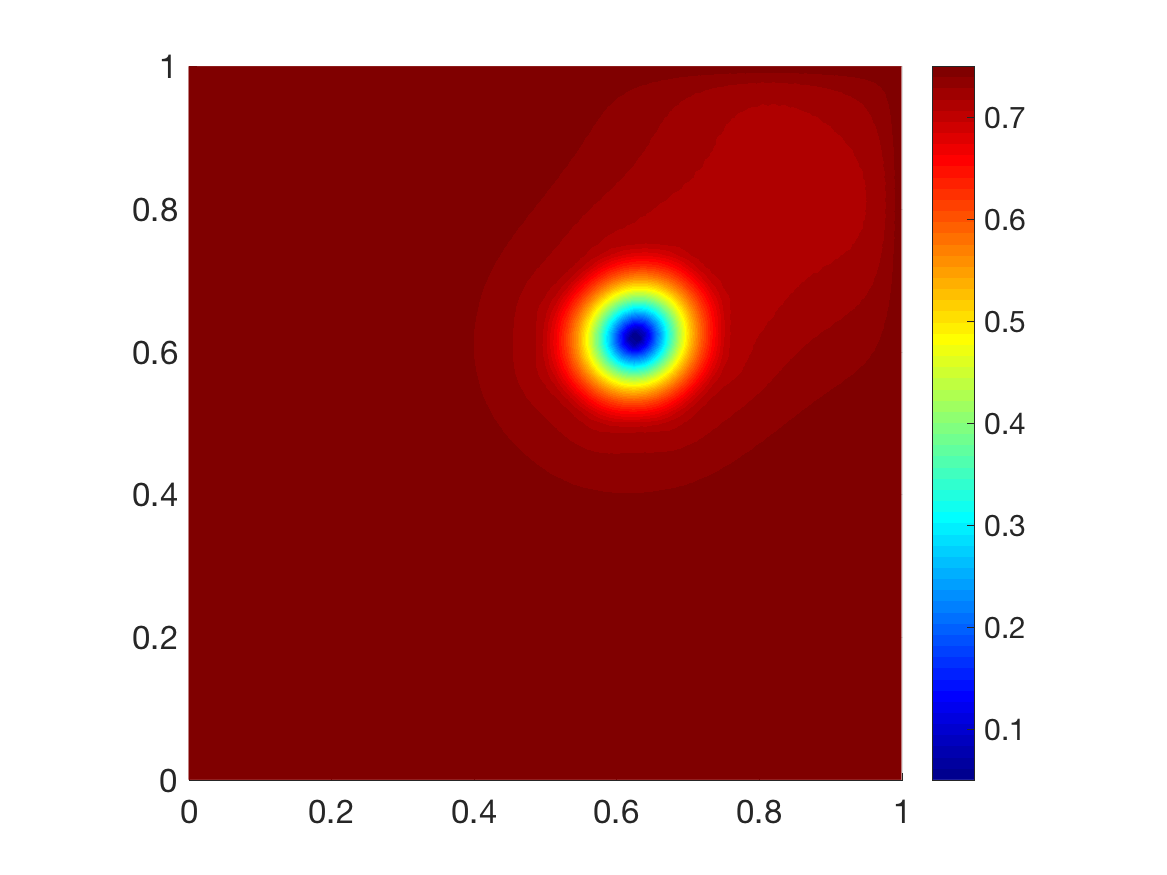}} \hspace{-2em}
\subfloat{\includegraphics[width = 2in]{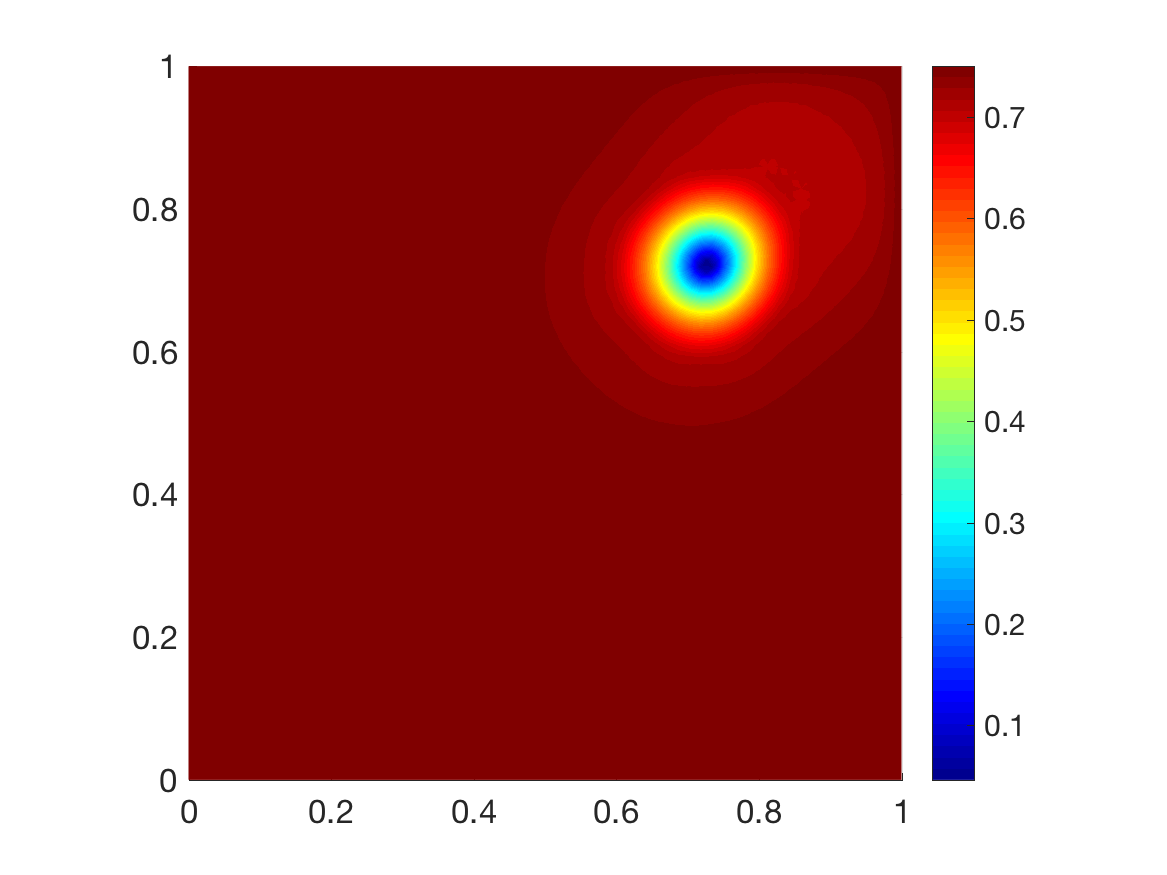}} \hspace{-2em}
\subfloat{\includegraphics[width = 2in]{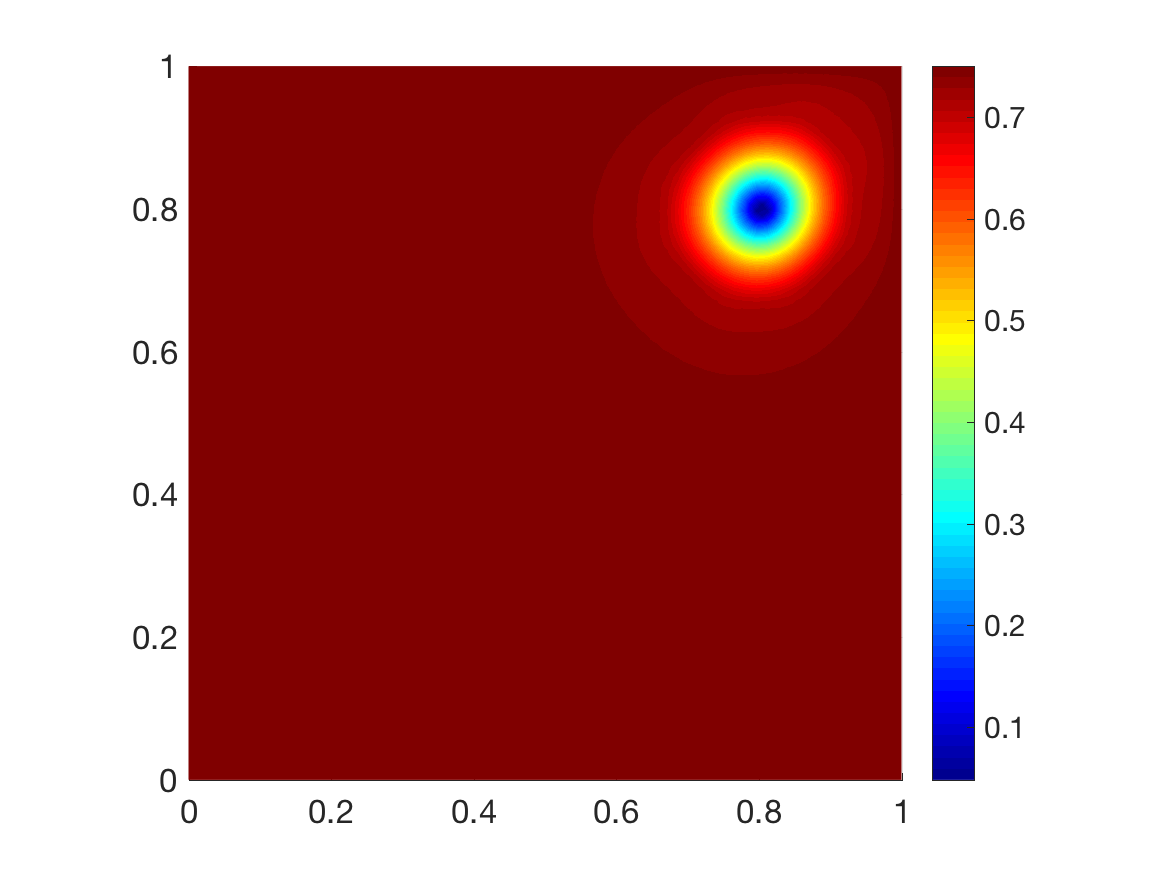}}\\[-4ex]
\subfloat{\includegraphics[width = 2in]{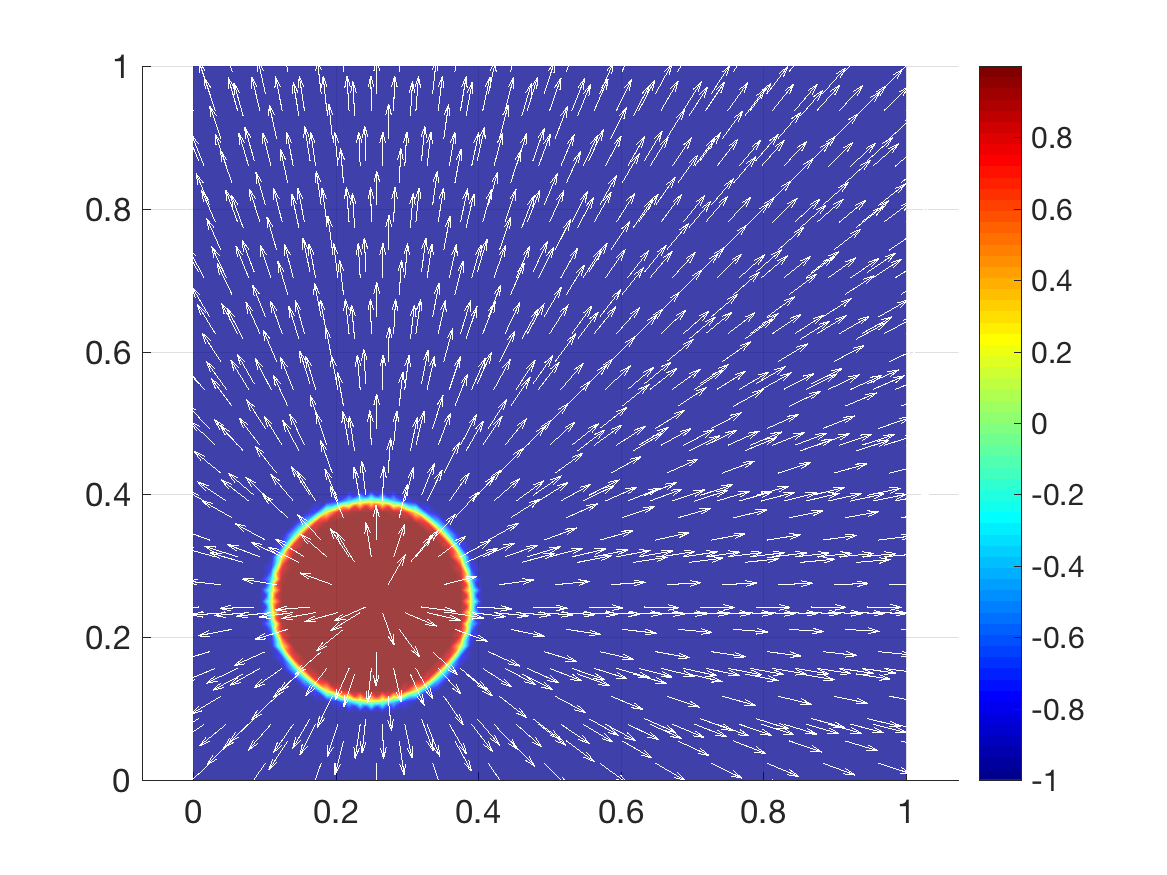}} \hspace{-2em}
\subfloat{\includegraphics[width = 2in]{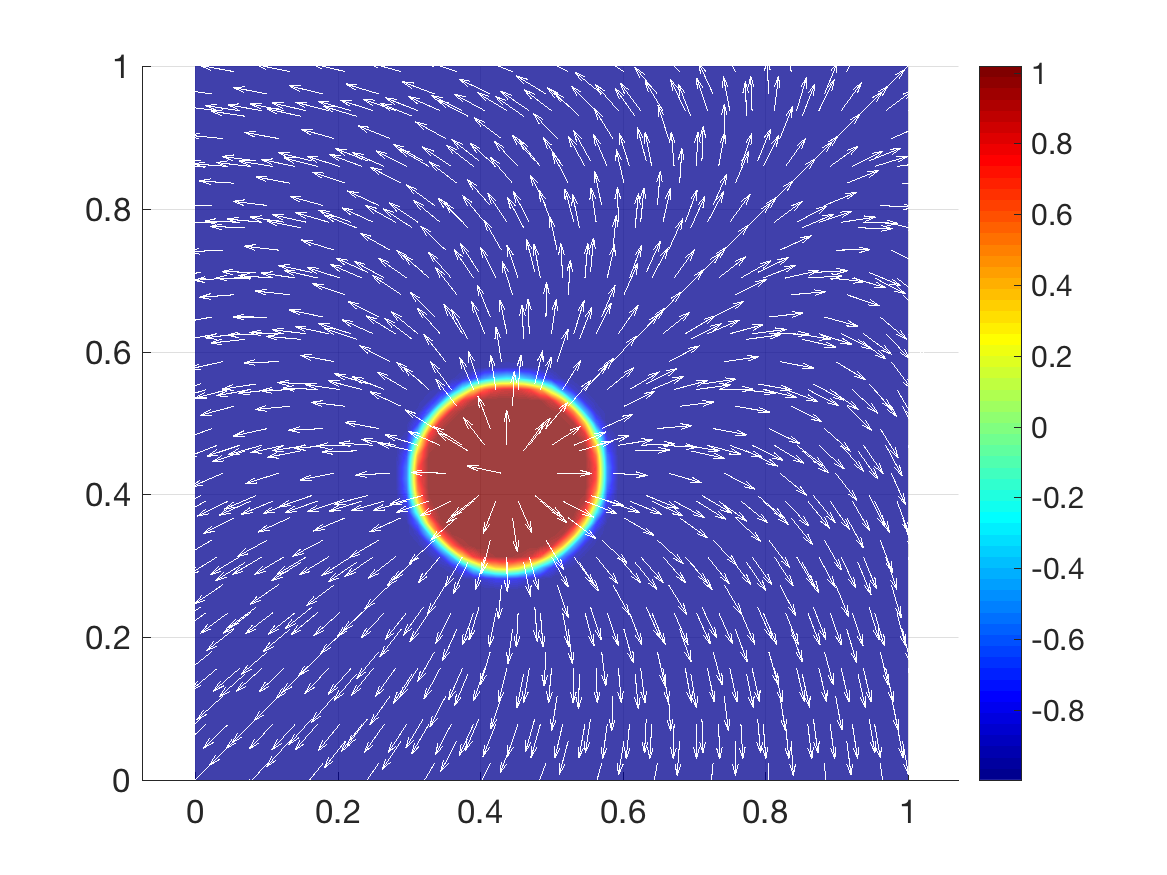}} \hspace{-2em}
\subfloat{\includegraphics[width = 2in]{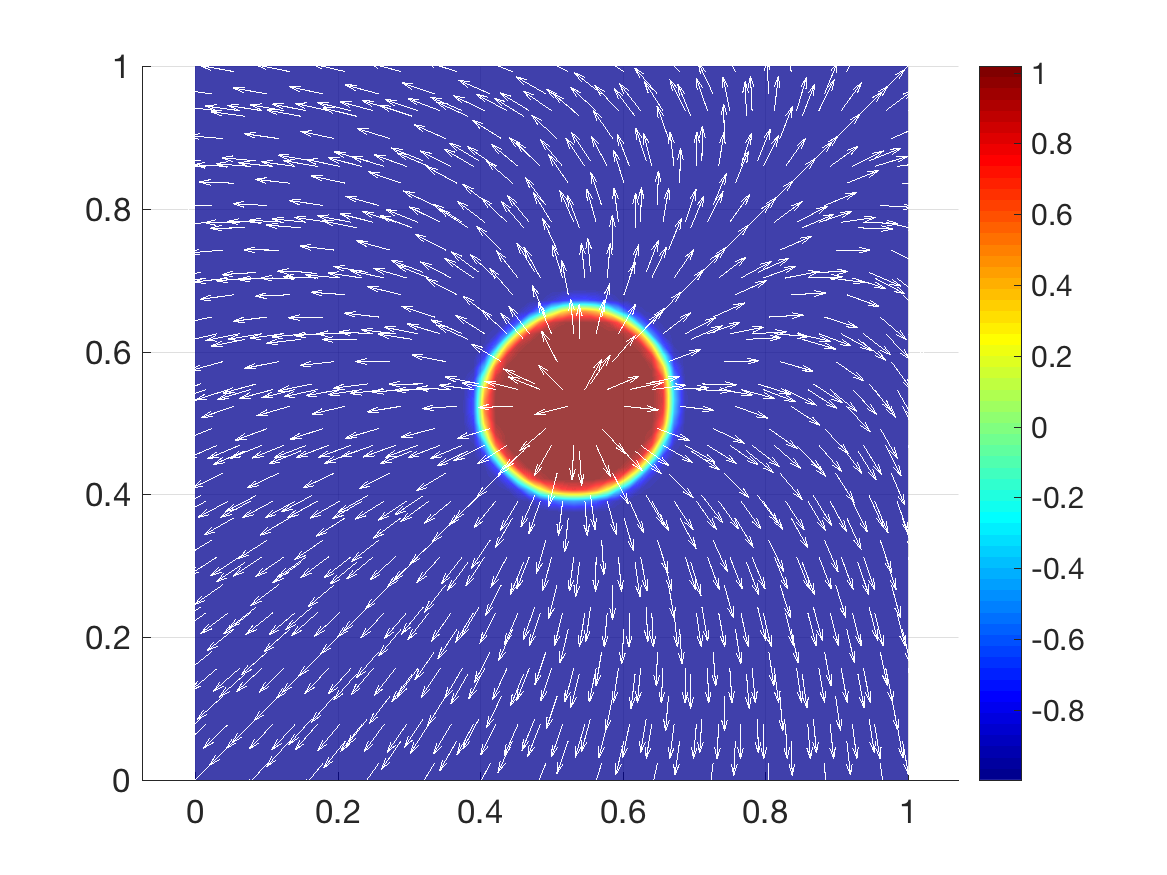}} \\[-4ex]
\subfloat{\includegraphics[width = 2in]{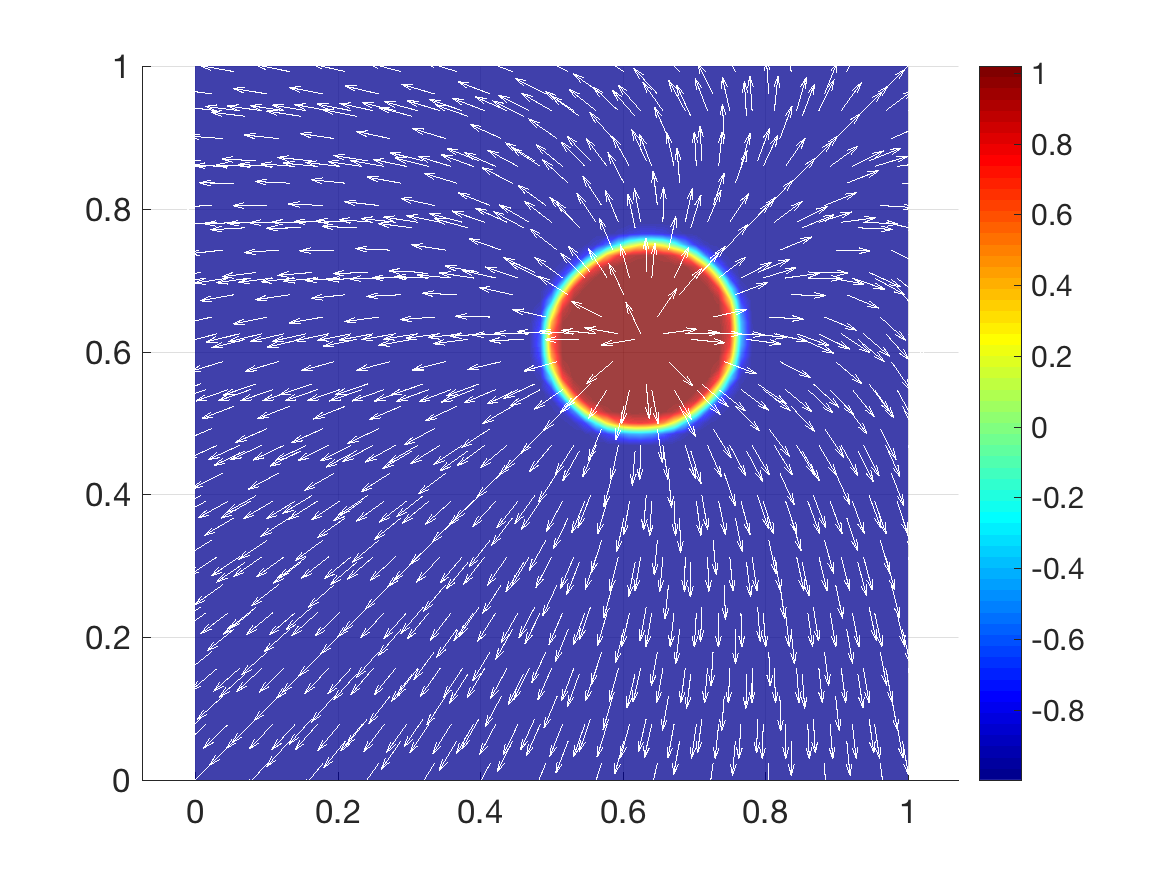}} \hspace{-2em}
\subfloat{\includegraphics[width = 2in]{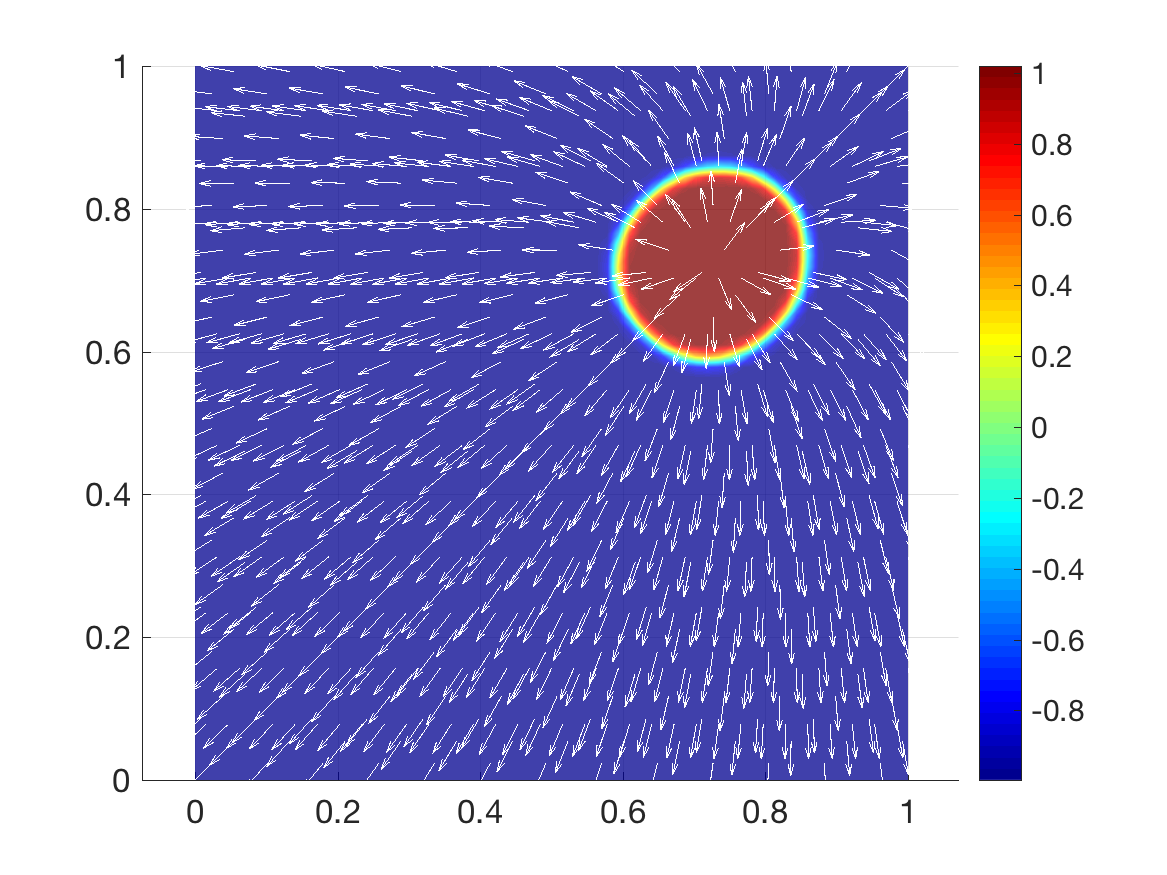}} \hspace{-2em}
\subfloat{\includegraphics[width = 2in]{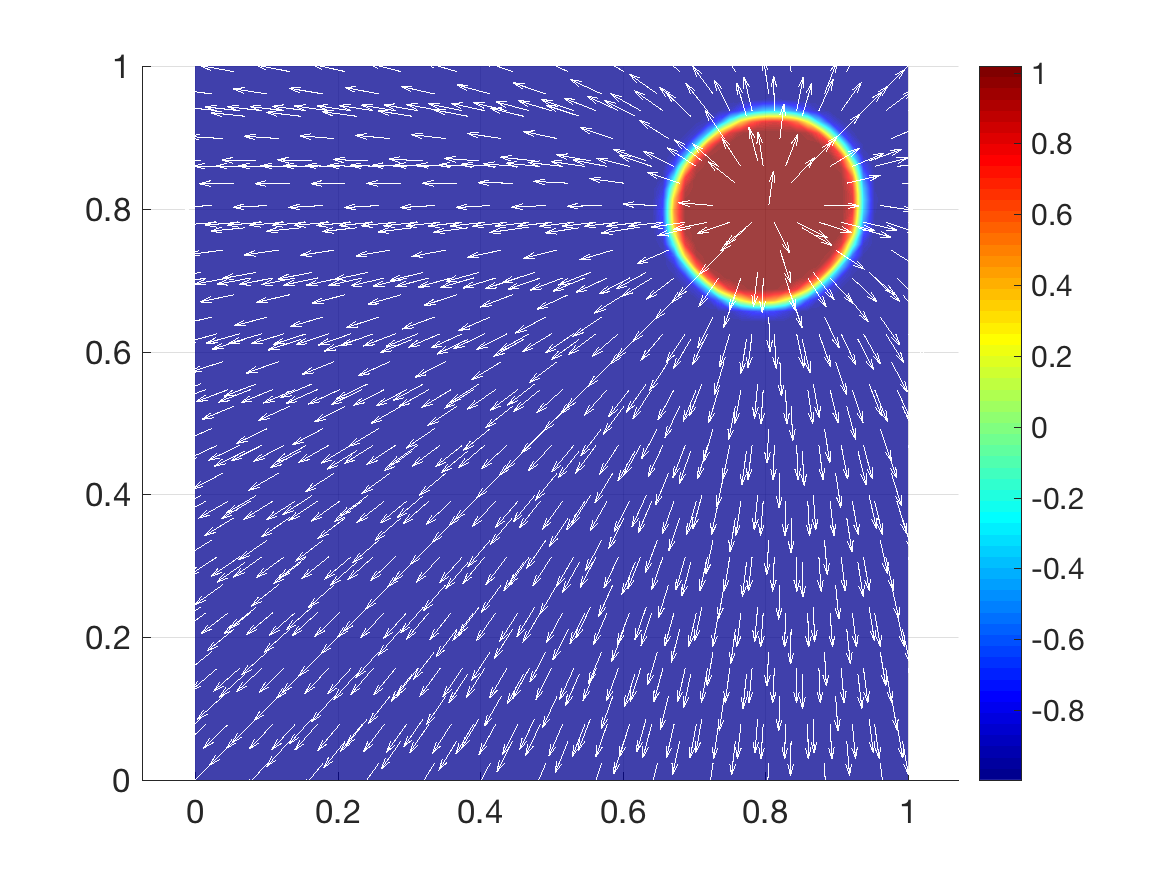}}
\caption{Droplet moving, $\Omega = [0,1]\times [0,1]$, $h = \sqrt{2} / 64$, $\tau = 0.002$ (Section \ref{sec:move-LC-droplet}). The times displayed are $t=0, t=4, t=8$ (top from left to right) and $t=12, t=16, t=20$ (bottom from left to right).}
\label{fig:droplet-moving}
\end{figure}

Figure \ref{fig:droplet-moving-energy} displays the energy decreasing property of the scheme for this experiment. We point out that the energy decreases dramatically at the beginning of the simulation due to the droplet adjusting to its equilibrium shape but then levels off.

\begin{figure}
\subfloat{\includegraphics[width = 3in]{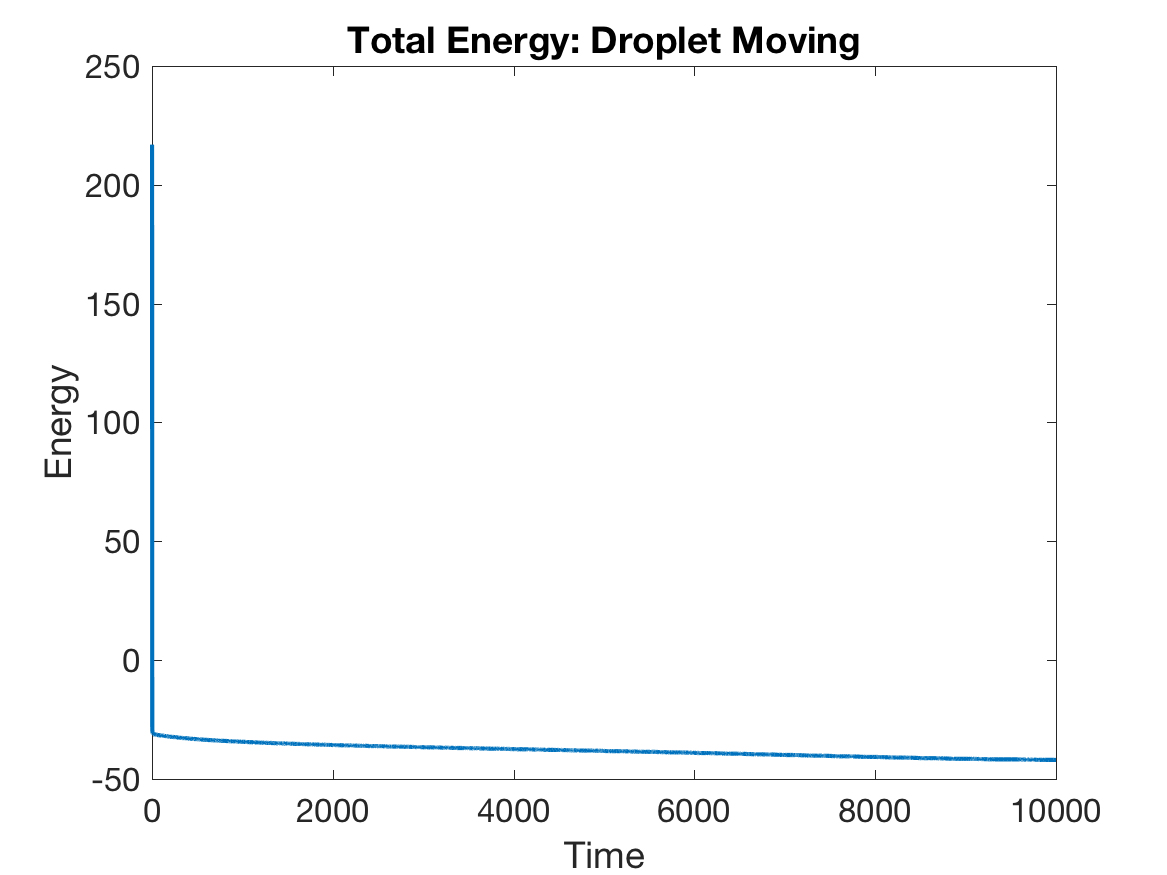}}
\subfloat{\includegraphics[width = 3in]{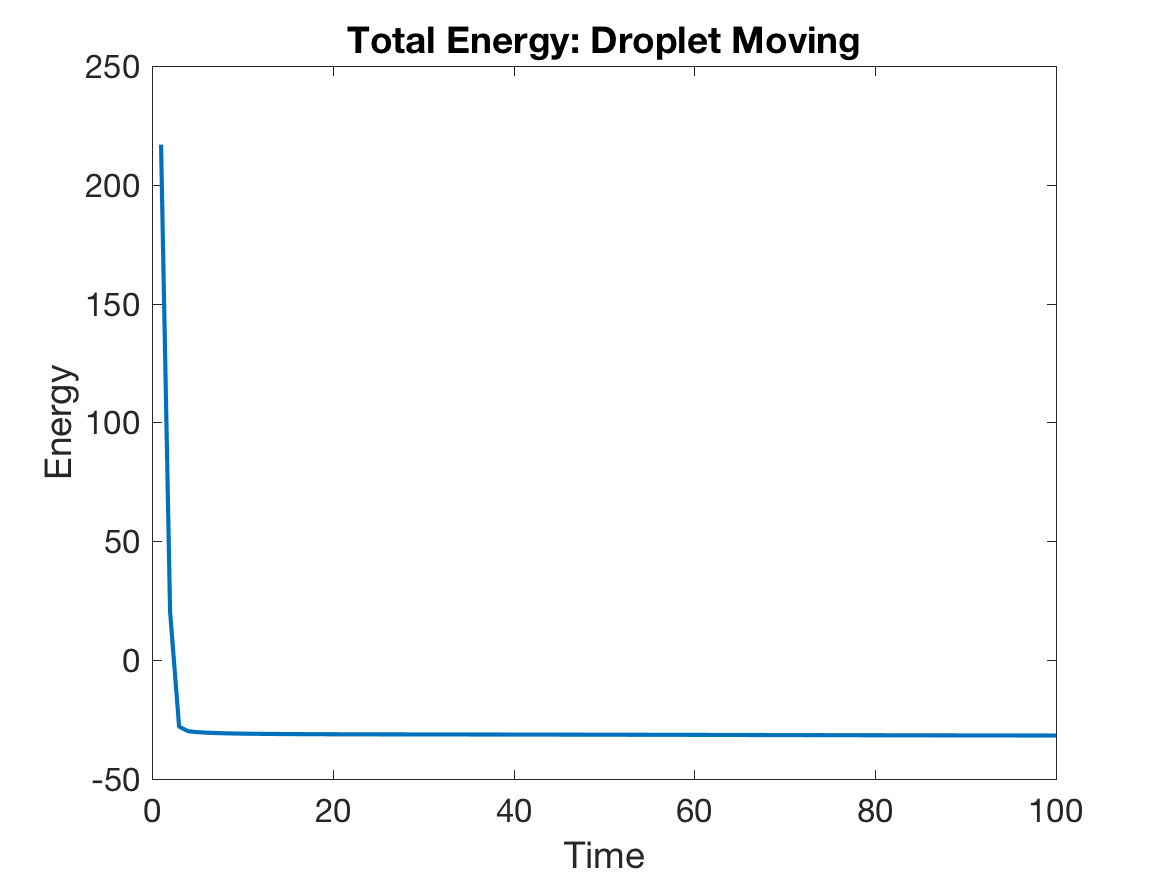}}
\caption{Total energy as a function of time for a moving droplet (Section \ref{sec:move-LC-droplet}).}
\label{fig:droplet-moving-energy}
\end{figure}

\newpage
\subsection{Cornering Effect of a Liquid Crystal Droplet}
\label{sec:cornering-LC-droplet}

The second numerical experiment demonstrates the ``cornering'' effect of a liquid crystal droplet. The initial conditions are as follows:
\begin{align*}
s_h^0 &= s^*,
\\
\bnh^0 &= (1,0),
\\
\phih^0 &= I_h\left\{-\tanh\left(\frac{(x-0.5)^2/0.02 + (y-0.5)^2/0.02 - 1}{2\varepsilon}\right)\right\}.
\end{align*}
The following Dirichlet boundary conditions on $\partial \Omega$ are imposed for $s$ and $\bn$:
\begin{align*}
s = s^*, \quad \bnh &= (1,0).
\end{align*}
The relevant parameters are $\kappa = 1, \rho = 1, \Werk =  1, \Wdw = 100, \Wchdw = 1, \Wchgd = 1 + \Wwan + \Wwas = 41, \Wwas = 20, \Wwan = 20$. The space step size is taken to be $h = \sqrt{2}/64$ and the time step size is taken to be $\tau = 0.002$ with a final stopping time of $T=2.0$. The interfacial width parameter is taken to be $\varepsilon = 3h/\sqrt{2}$. Figure \ref{fig:droplet-cornering} shows the evolution of the droplet over time. The top two rows display the evolution of the scalar degree of orientation parameter $s$. The bottom two rows show the evolution of the phase field parameter $\phi$ and the director field $\bn$.  The droplet takes on a ``lens'' shape with corners at the top and bottom.  Note that the cornering is not sharp due to having finite surface tension, as well as a finite interfacial width parameter $\varepsilon$.

\begin{figure}
\subfloat{\includegraphics[width = 2in]{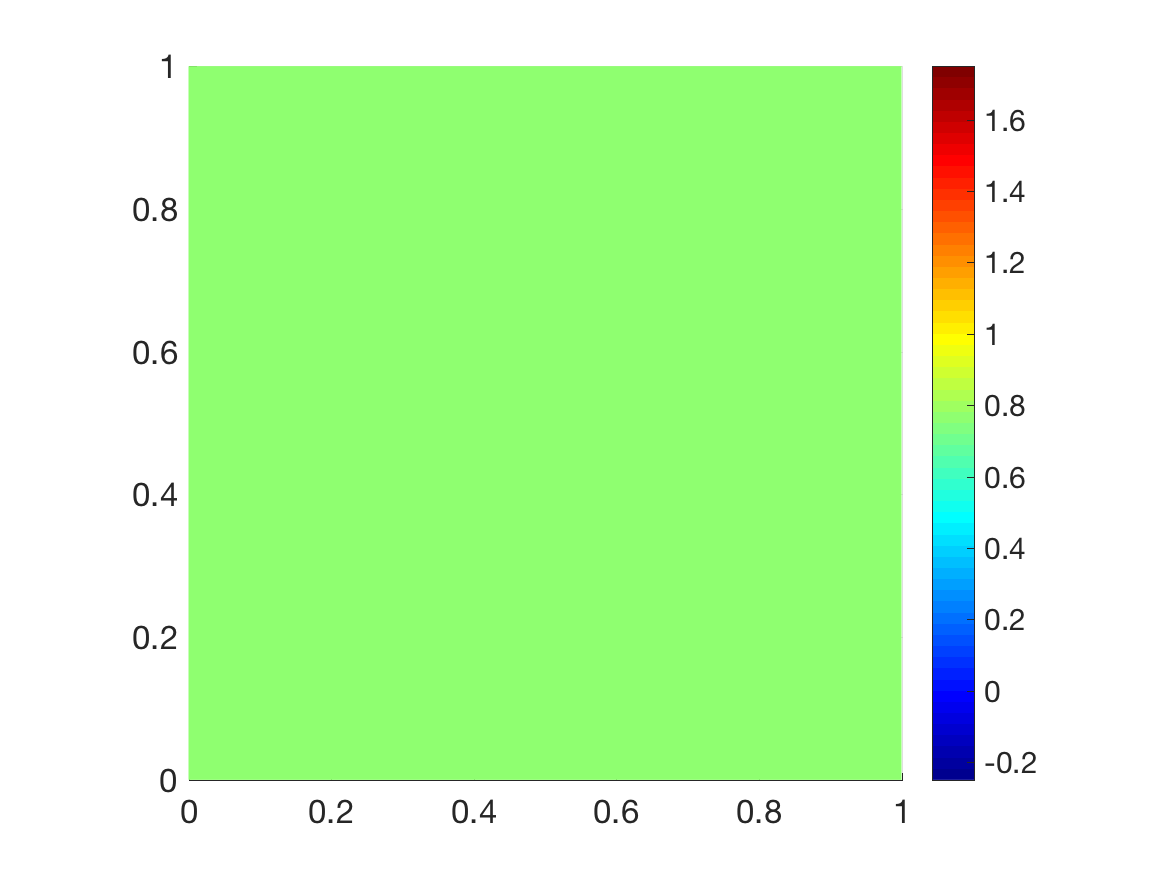}} \hspace{-2em}
\subfloat{\includegraphics[width = 2in]{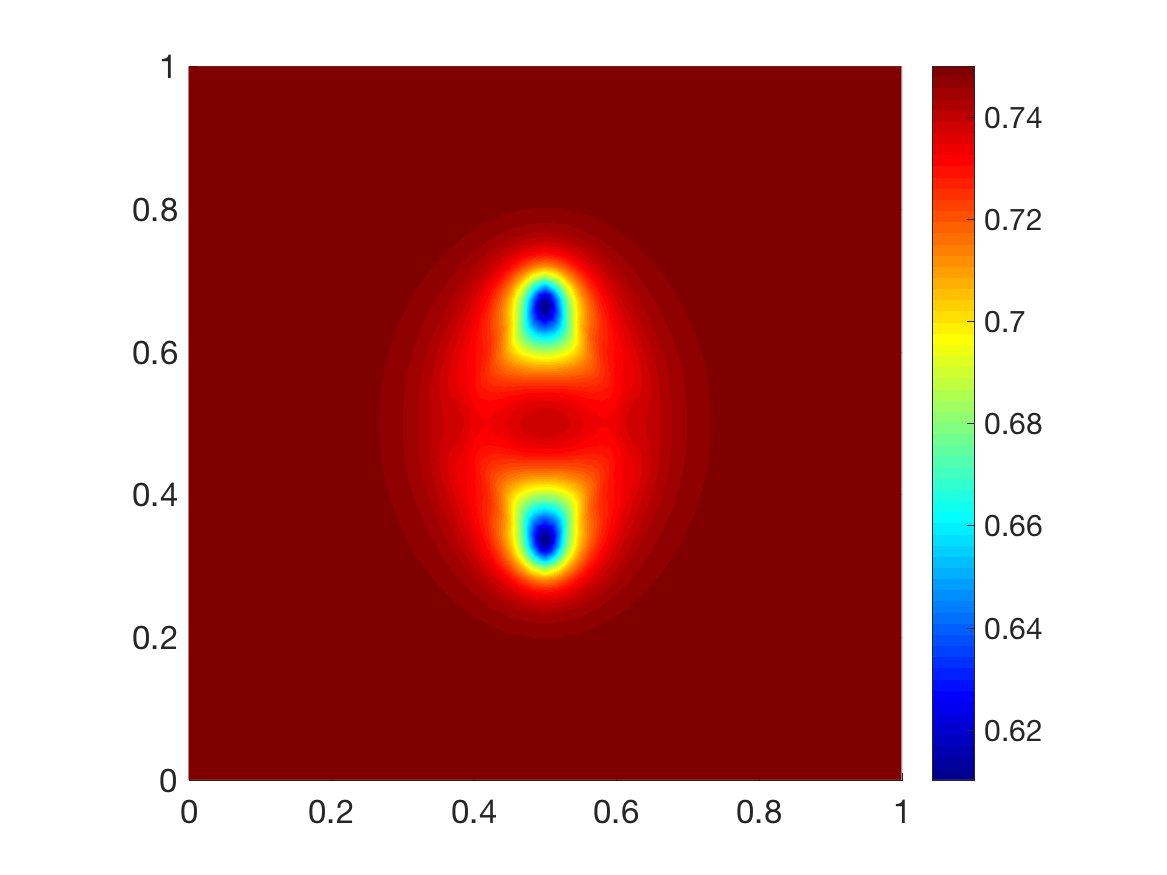}} \hspace{-2em}
\subfloat{\includegraphics[width = 2in]{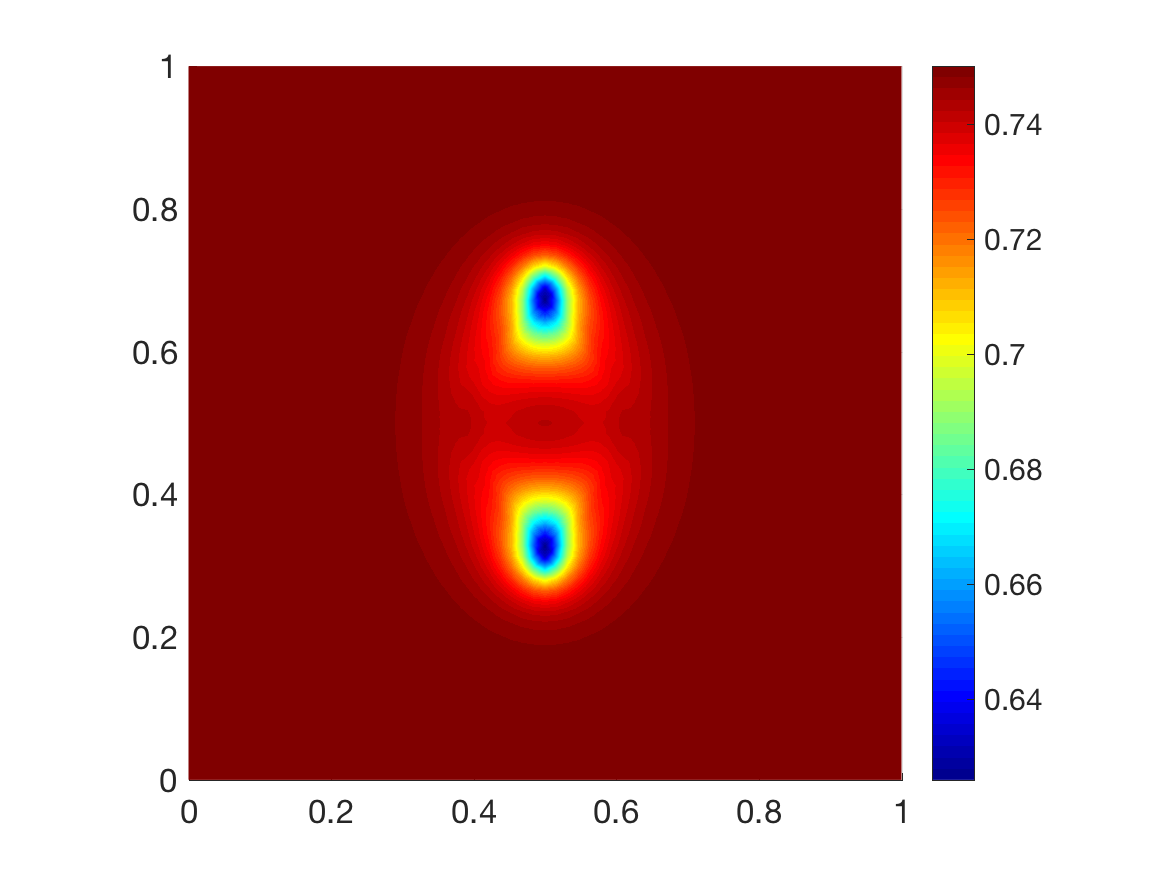}} \\[-4ex]
\subfloat{\includegraphics[width = 2in]{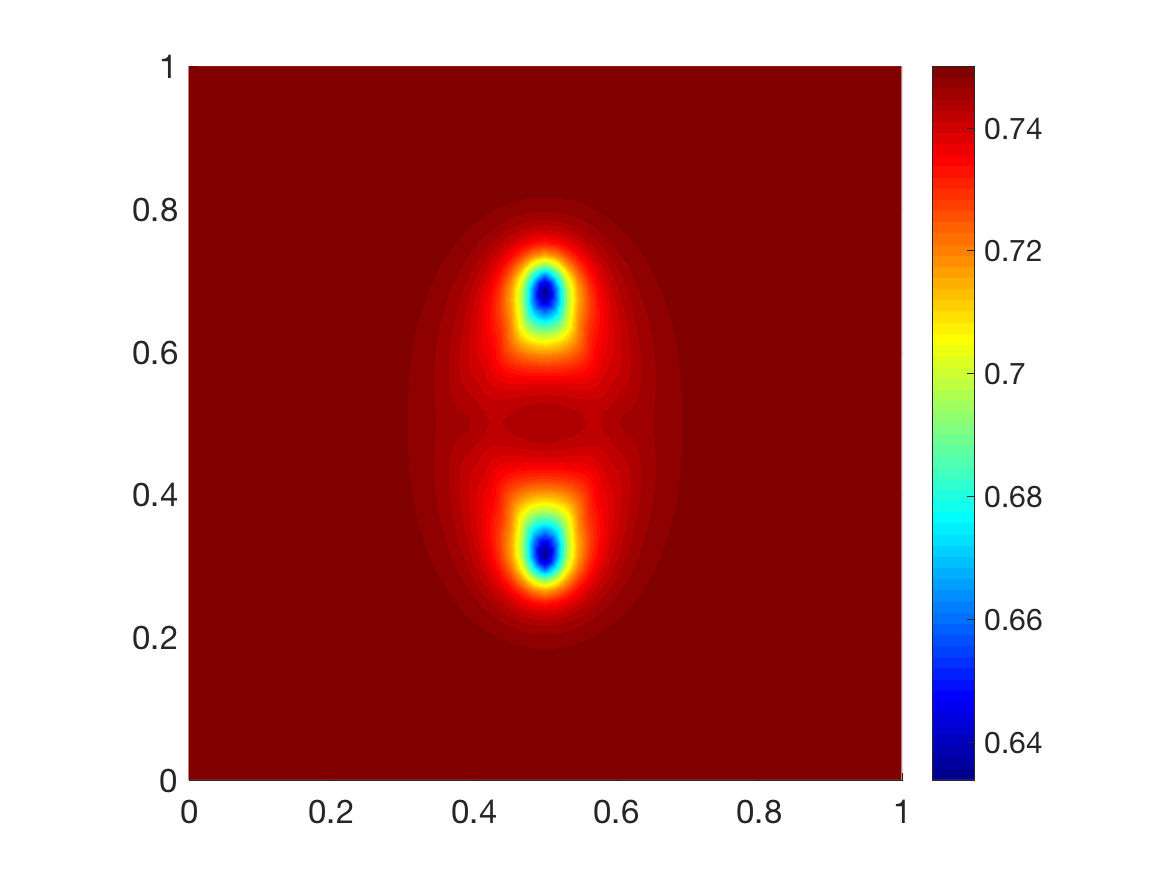}} \hspace{-2em}
\subfloat{\includegraphics[width = 2in]{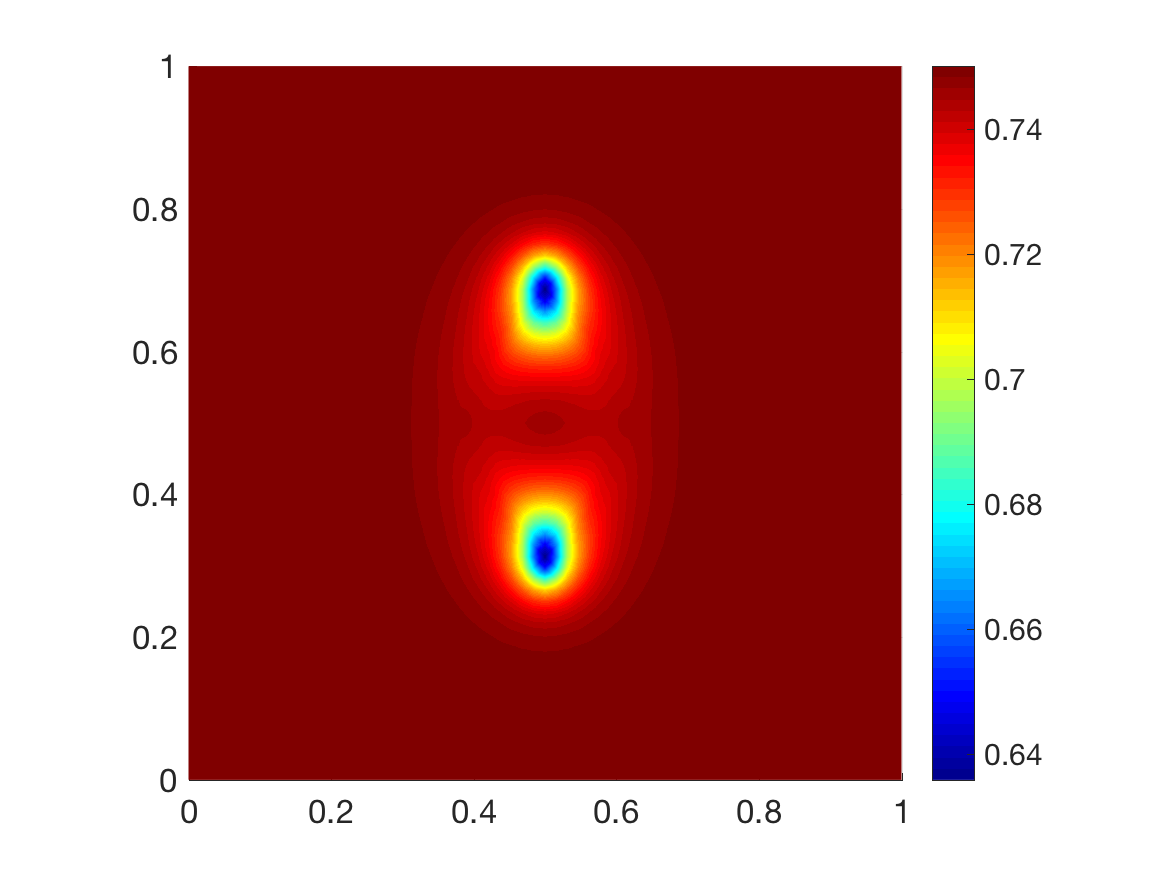}}\hspace{-2em}
\subfloat{\includegraphics[width = 2in]{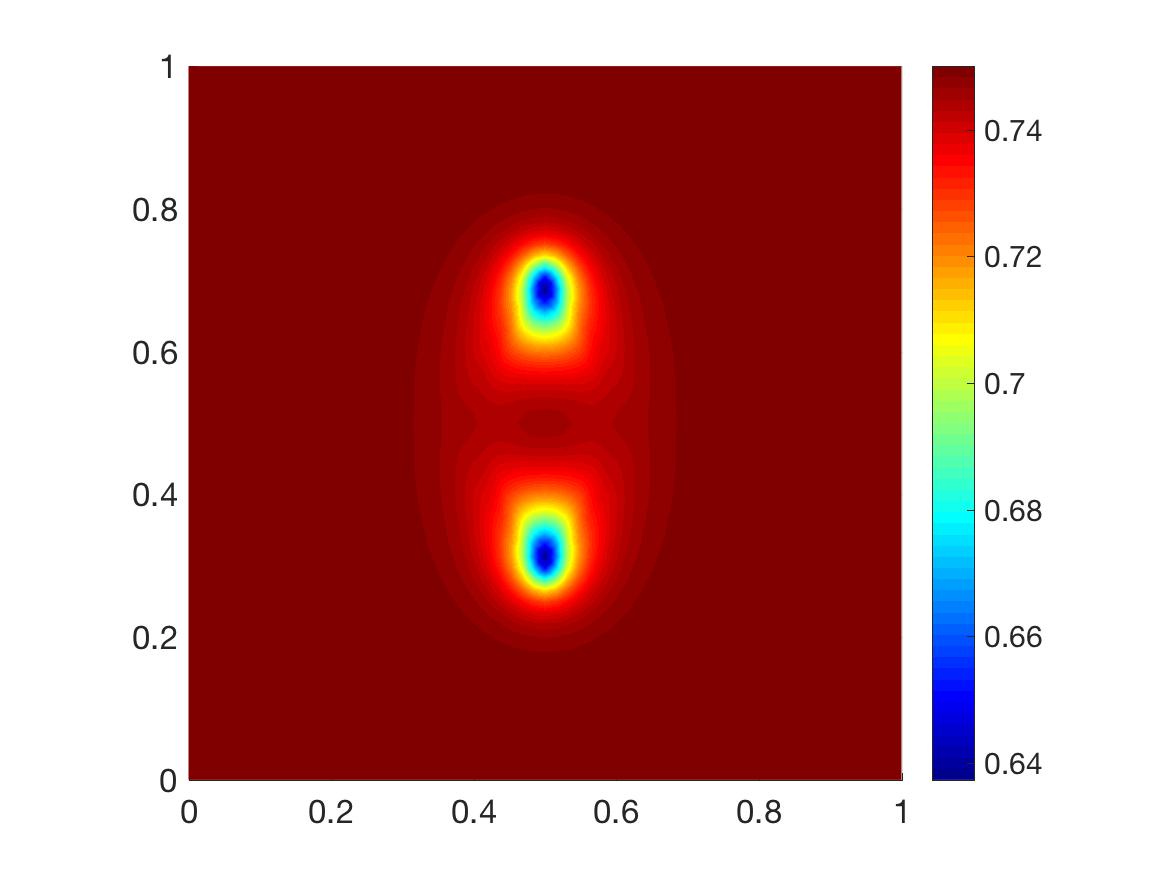}}\\[-4ex]
\subfloat{\includegraphics[width = 2in]{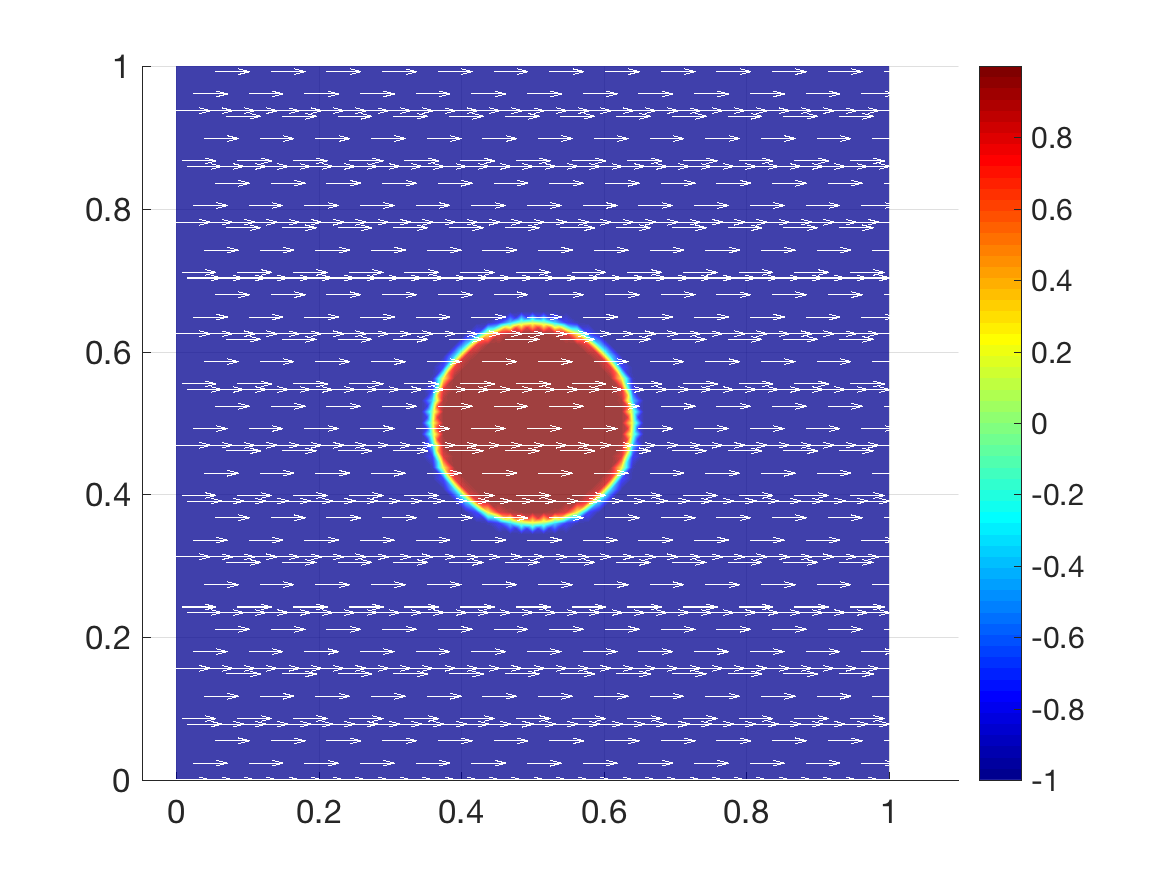}} \hspace{-2em}
\subfloat{\includegraphics[width = 2in]{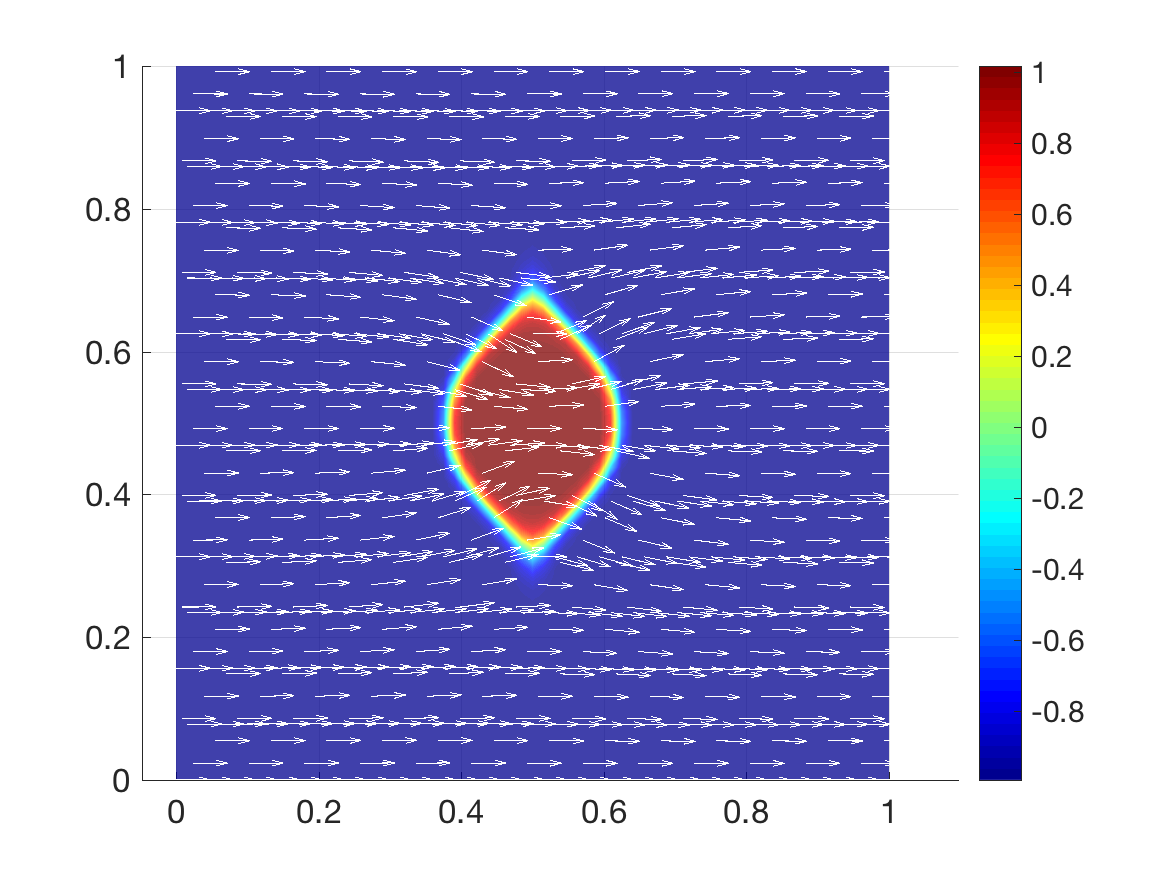}} \hspace{-2em}
\subfloat{\includegraphics[width = 2in]{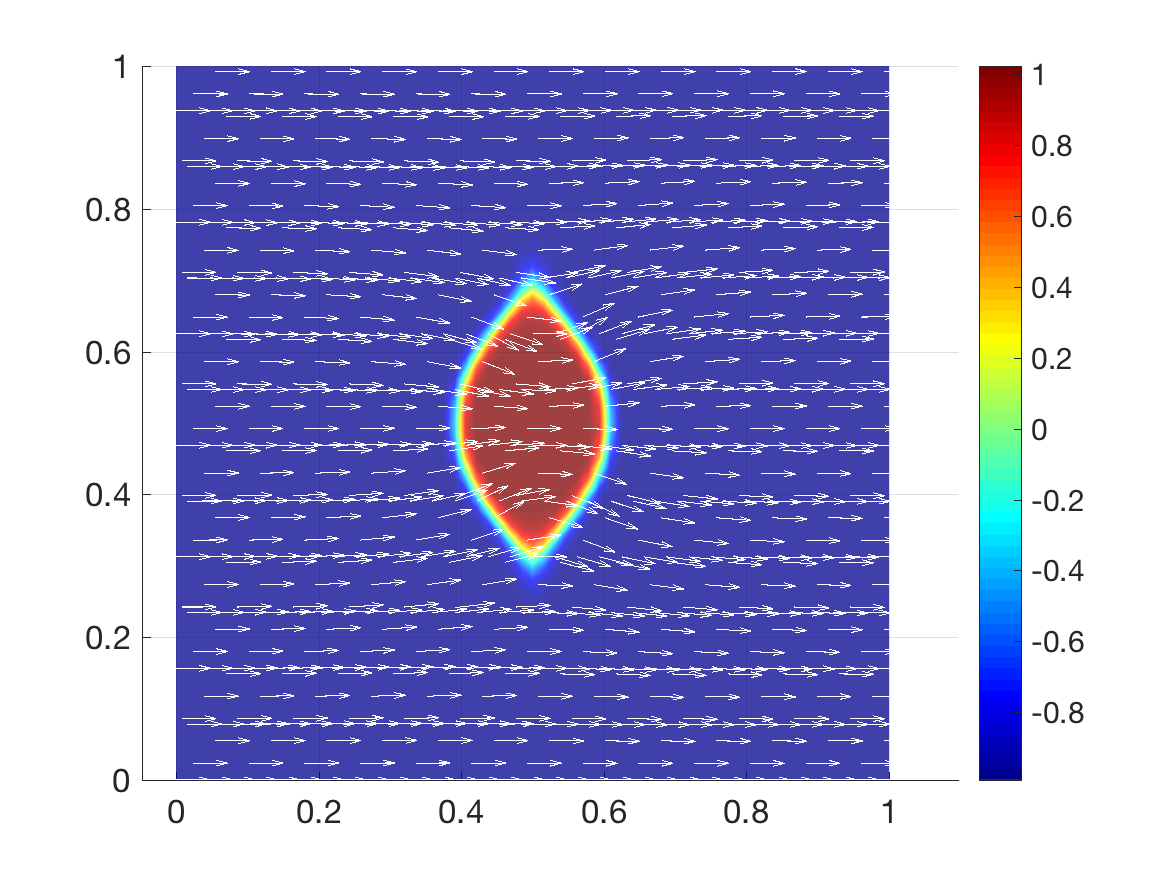}} \\[-4ex]
\subfloat{\includegraphics[width = 2in]{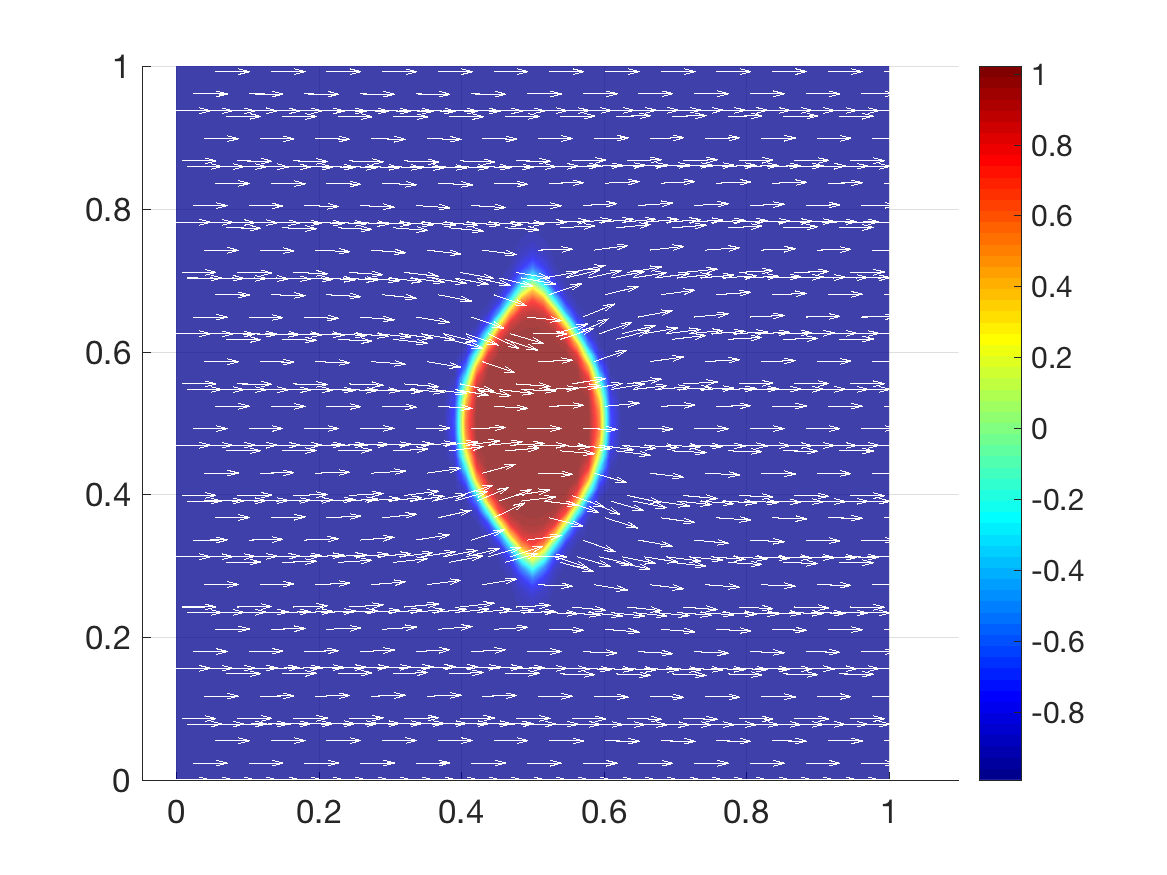}} \hspace{-2em}
\subfloat{\includegraphics[width = 2in]{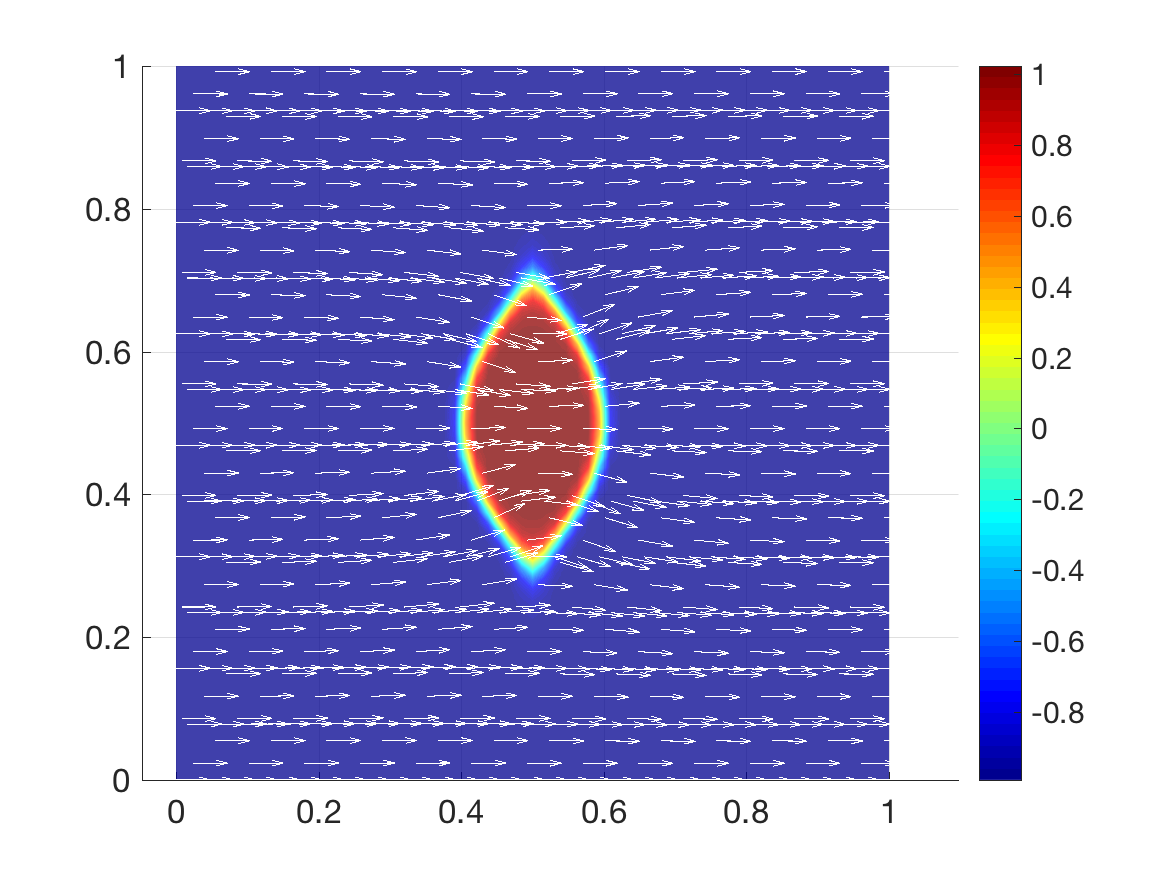}}\hspace{-2em}
\subfloat{\includegraphics[width = 2in]{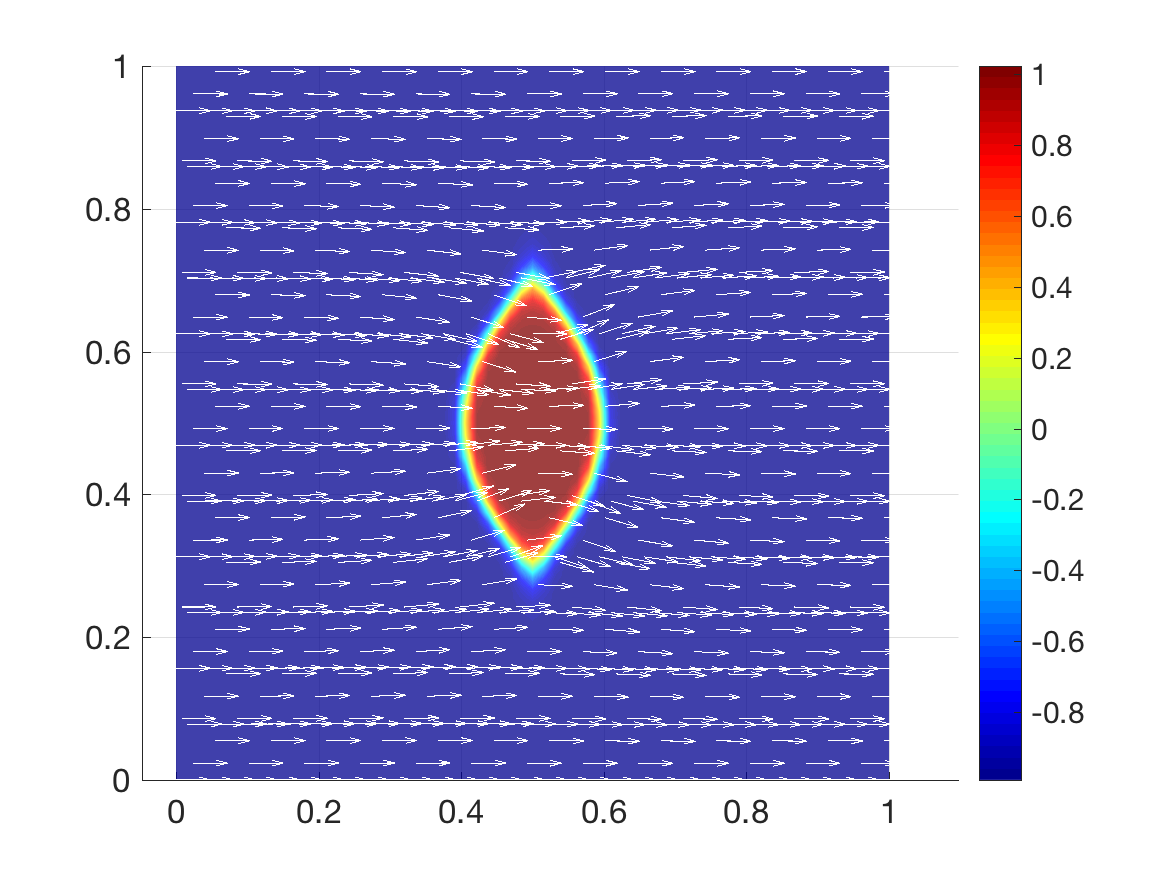}}
\caption{Droplet cornering, $\Omega = [0,1]\times[0,1]$, $h = \sqrt{2} / 64$, $\tau = 0.002$ (Section \ref{sec:cornering-LC-droplet}). The times displayed are $t=0, t=0.04, t=0.08$ (top from left to right) and $t=0.12, t=0.16, t=0.2$ (bottom from left to right).}
\label{fig:droplet-cornering}
\end{figure}

Figure \ref{fig:droplet-cornering-energy} displays the energy decreasing property of the scheme for this experiment.

\begin{figure}
\subfloat{\includegraphics[width = 3in]{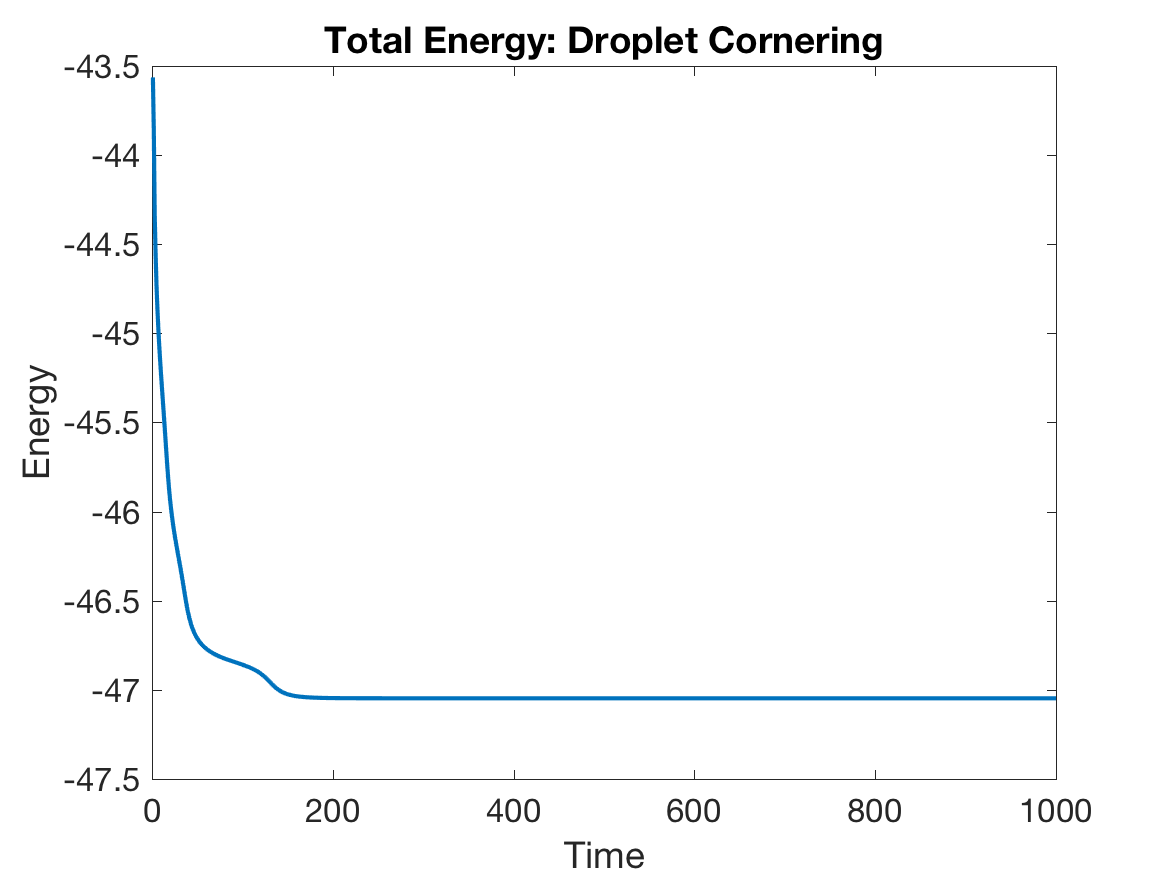}}
\subfloat{\includegraphics[width = 3in]{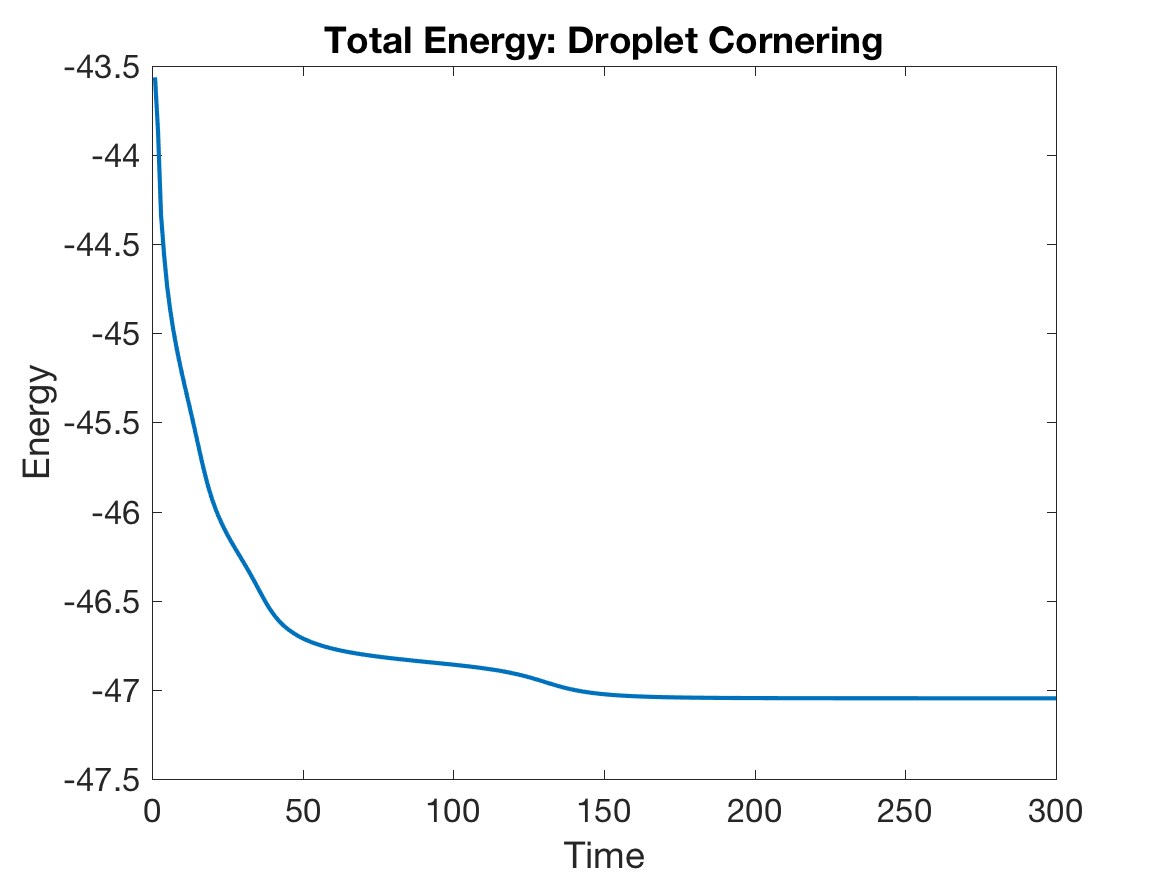}}
\caption{Total energy as a function of time for a droplet undergoing the cornering effect (Section \ref{sec:cornering-LC-droplet}).}
\label{fig:droplet-cornering-energy}
\end{figure}

\newpage

\subsection{Two Liquid Crystal Droplets Colliding}
\label{sec:collide-LC-droplet}

The third numerical experiment demonstrates two liquid crystal droplets colliding. The initial conditions are as follows:
\begin{align*}
s_h^0 &= s^*,
\\
\bnh^0 &=
\begin{cases}
 \frac{(x,y) - (0.3,0.5)}{|(x,y) - (0.3,0.5)|}, \quad x \le 0.5,
\\
\\
 \frac{-\left((x,y) - (0.7,0.5)\right)}{|(x,y) - (0.7,0.5)|}, \quad x > 0.5,
\end{cases}
\\
\phih^0 &=
\begin{cases}
I_h\left\{-\tanh\left(\frac{(x-0.3)^2/0.02 + (y-0.5)^2/0.02 - 1}{2\varepsilon}\right)\right\}, \quad x \le 0.5,
\\
I_h\left\{-\tanh\left(\frac{(x-0.7)^2/0.02 + (y-0.5)^2/0.02 - 1}{2\varepsilon}\right)\right\}, \quad x > 0.5,.
\end{cases}
\end{align*}
The following Dirichlet boundary conditions on $\partial \Omega$ are imposed for $s$ and $\bn$:
\begin{align*}
s = s^*, \quad \bnh &= (1,0).
\end{align*}
The relevant parameters are $\kappa = 1, \rho = 1, \Werk =  1, \Wdw = 100, \Wchdw = 1, \Wchgd = 1 + \Wwan + \Wwas = 21, \Wwas = 10, \Wwan = 10$. The space step size is taken to be $h = \sqrt{2}/64$ and the time step size is taken to be $\tau = 0.002$ with a final stopping time of $T=2.0$. The interfacial width parameter is taken to be $\varepsilon = 3h/\sqrt{2}$. Figure \ref{fig:droplet-collide} shows the evolution of the droplet over time. The top two rows display the evolution of the scalar degree of orientation parameter $s$. The bottom two rows show the evolution of the phase field parameter $\phi$ and the director field $\bn$.  Due to the boundary conditions for $\bn$, the defects inside the droplets are driven to annihilate; this is what forces the droplets to collide.  At equilibrium, no defects remain and the droplet takes on a lens shape.

\begin{figure}
\subfloat{\includegraphics[width = 1.75in]{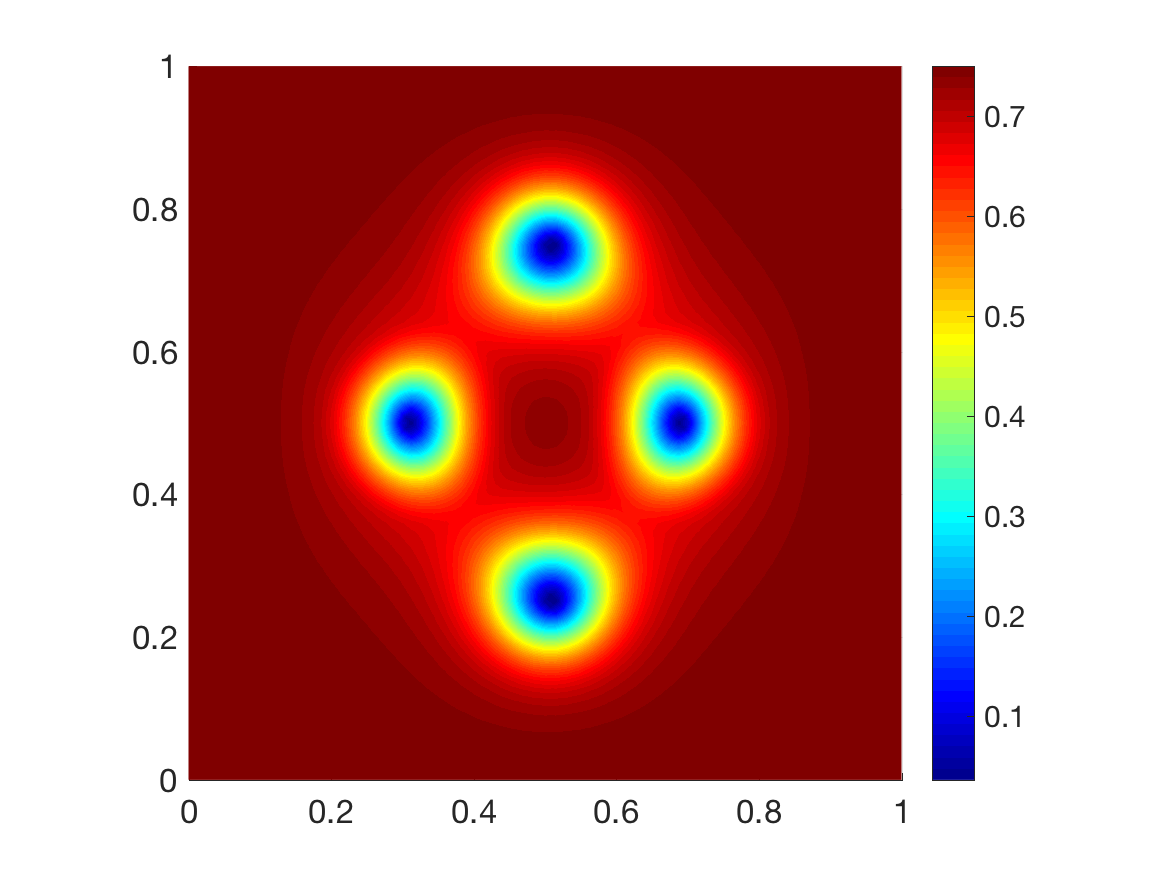}} \hspace{-2em}
\subfloat{\includegraphics[width = 1.75in]{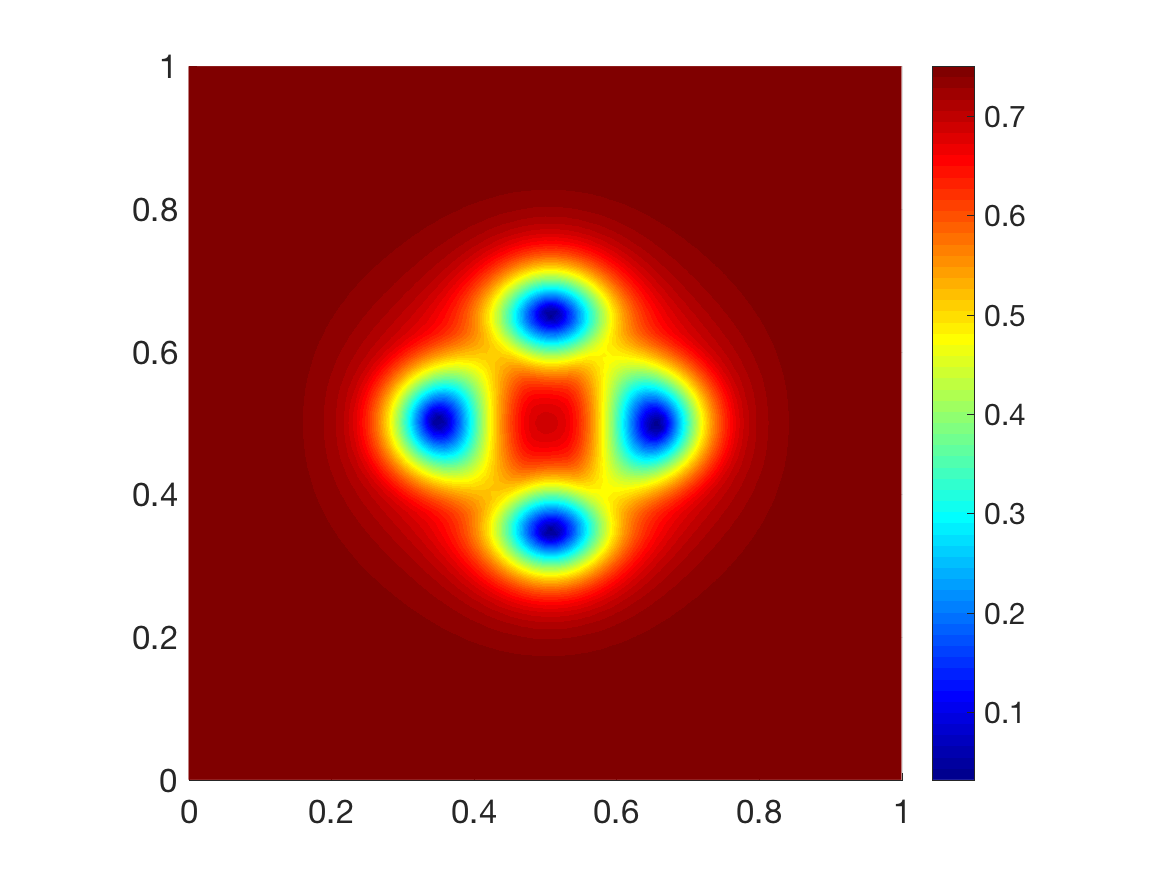}} \hspace{-2em}
\subfloat{\includegraphics[width = 1.75in]{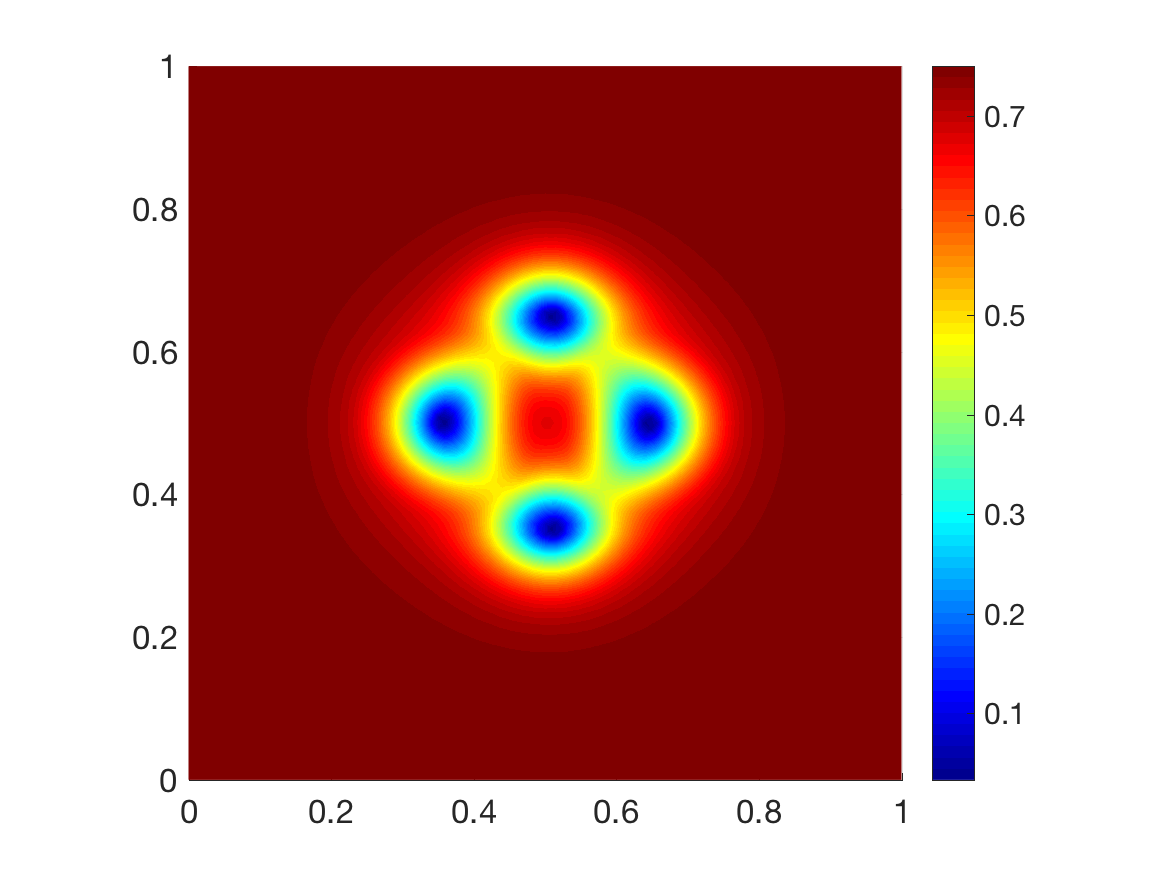}} \hspace{-2em}
\subfloat{\includegraphics[width = 1.75in]{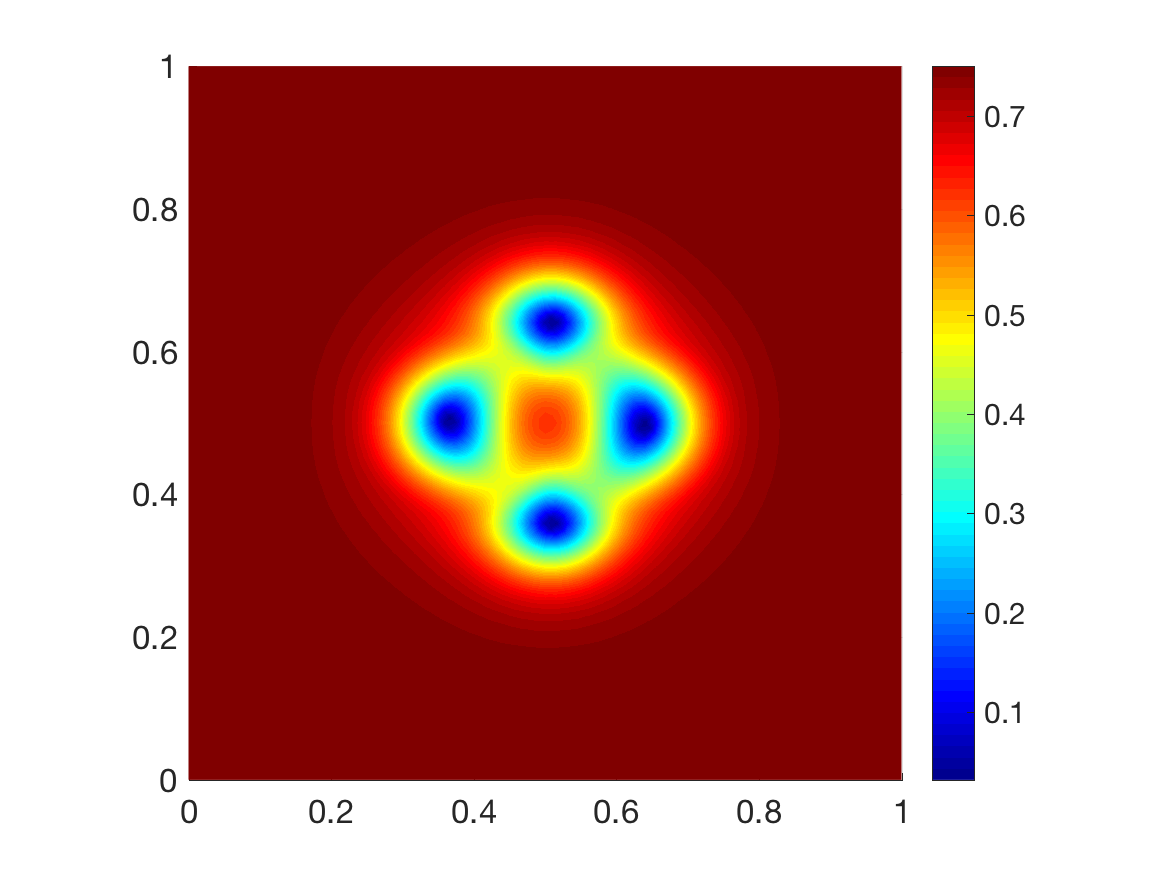}} \\[-4ex]
\subfloat{\includegraphics[width = 1.75in]{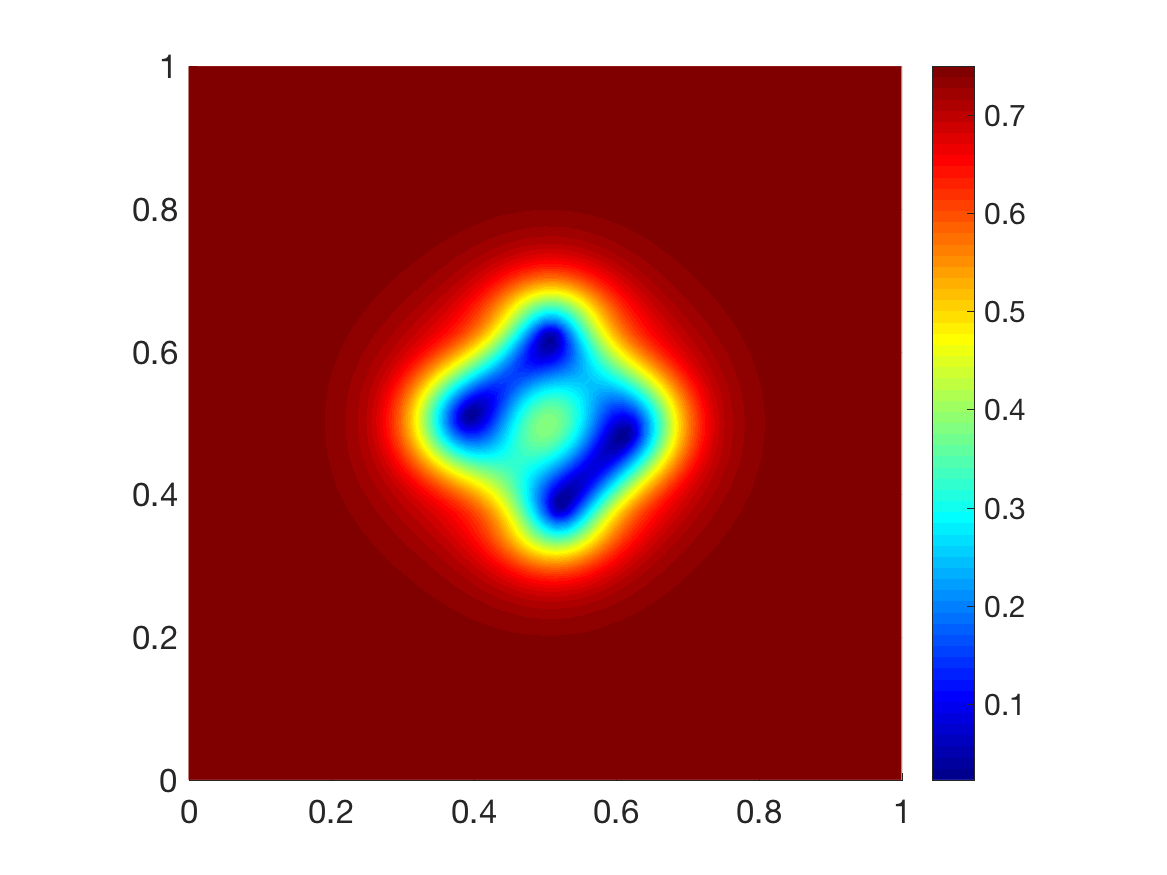}}\hspace{-2em}
\subfloat{\includegraphics[width = 1.75in]{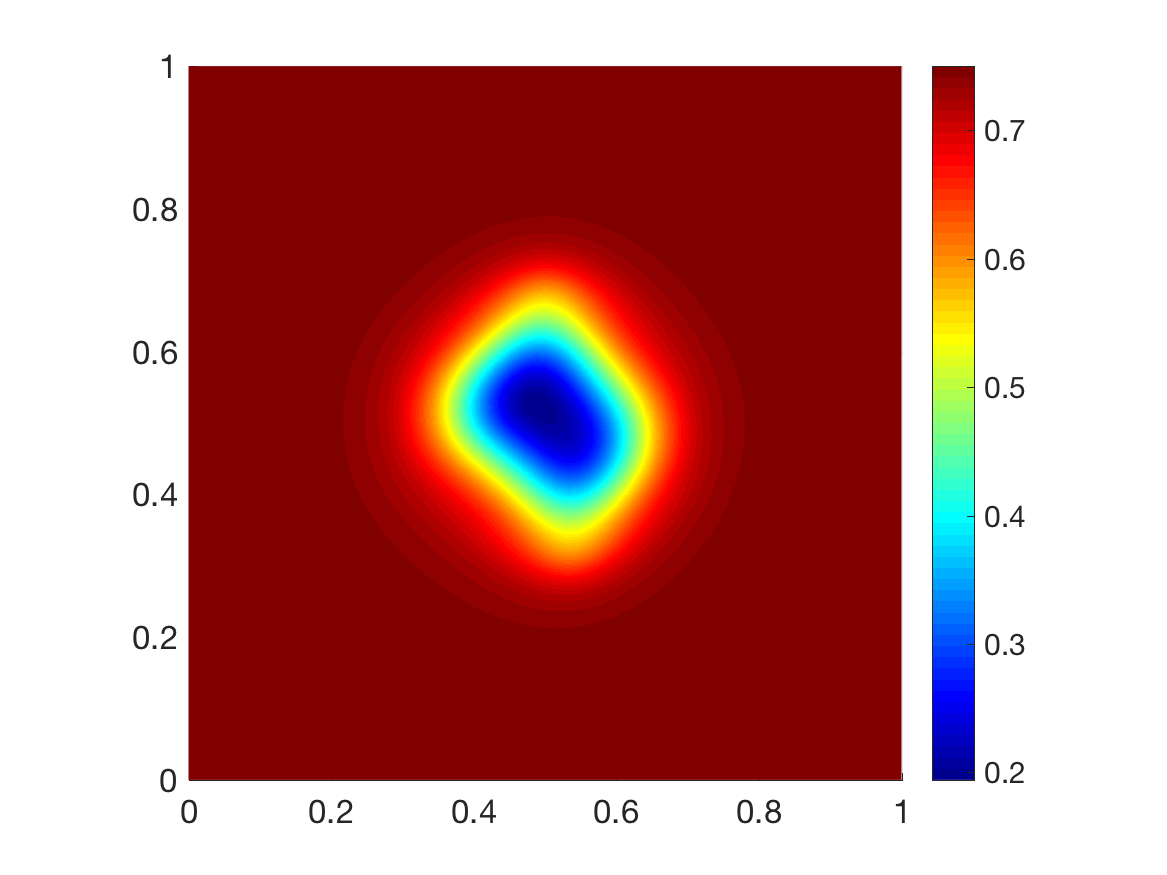}} \hspace{-2em}
\subfloat{\includegraphics[width = 1.75in]{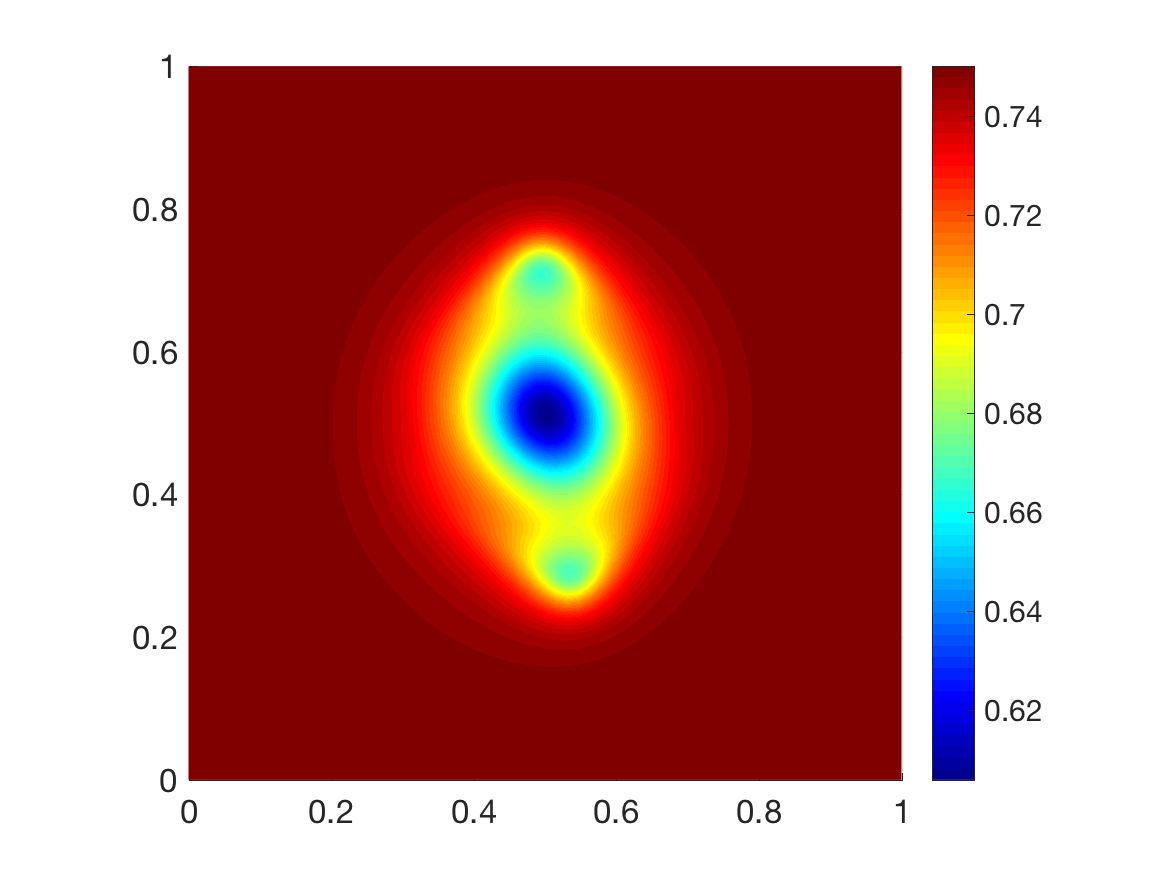}} \hspace{-2em}
\subfloat{\includegraphics[width = 1.75in]{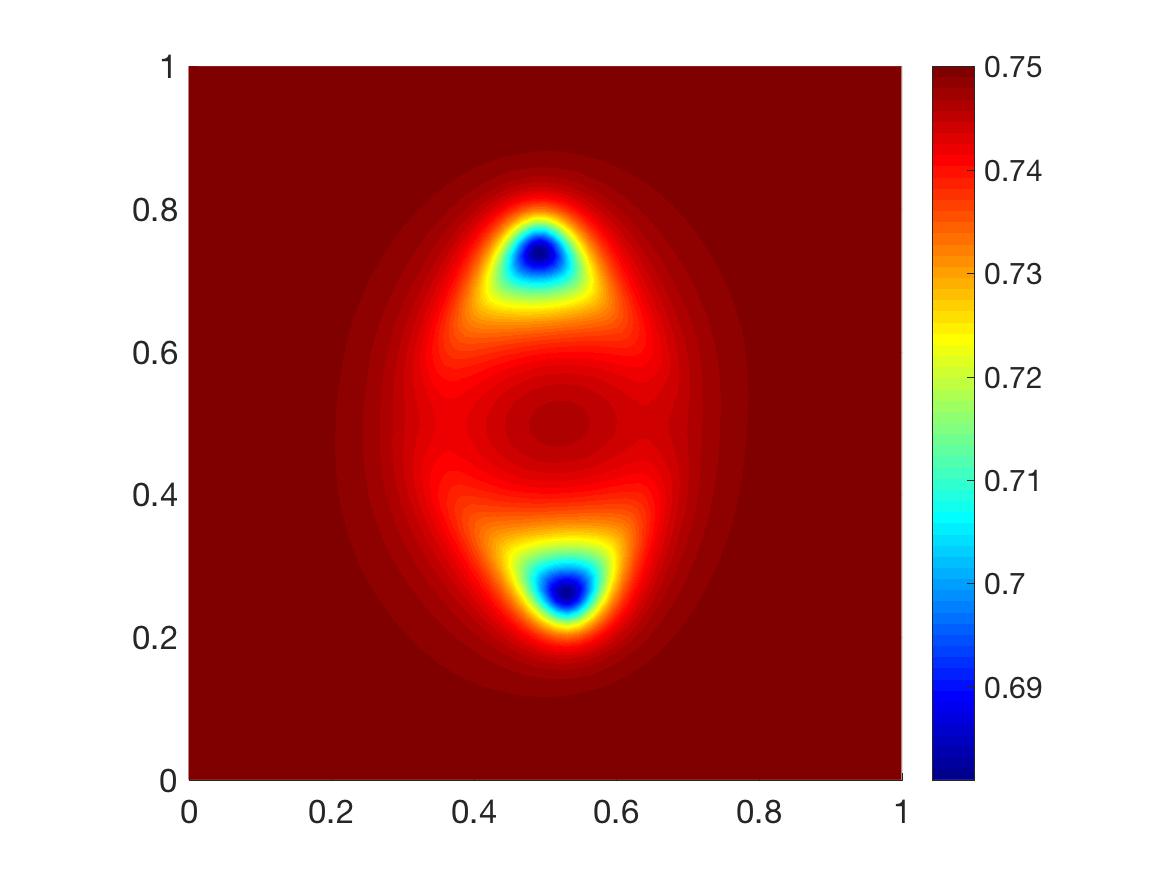}}\\[-4ex]
\subfloat{\includegraphics[width = 1.75in]{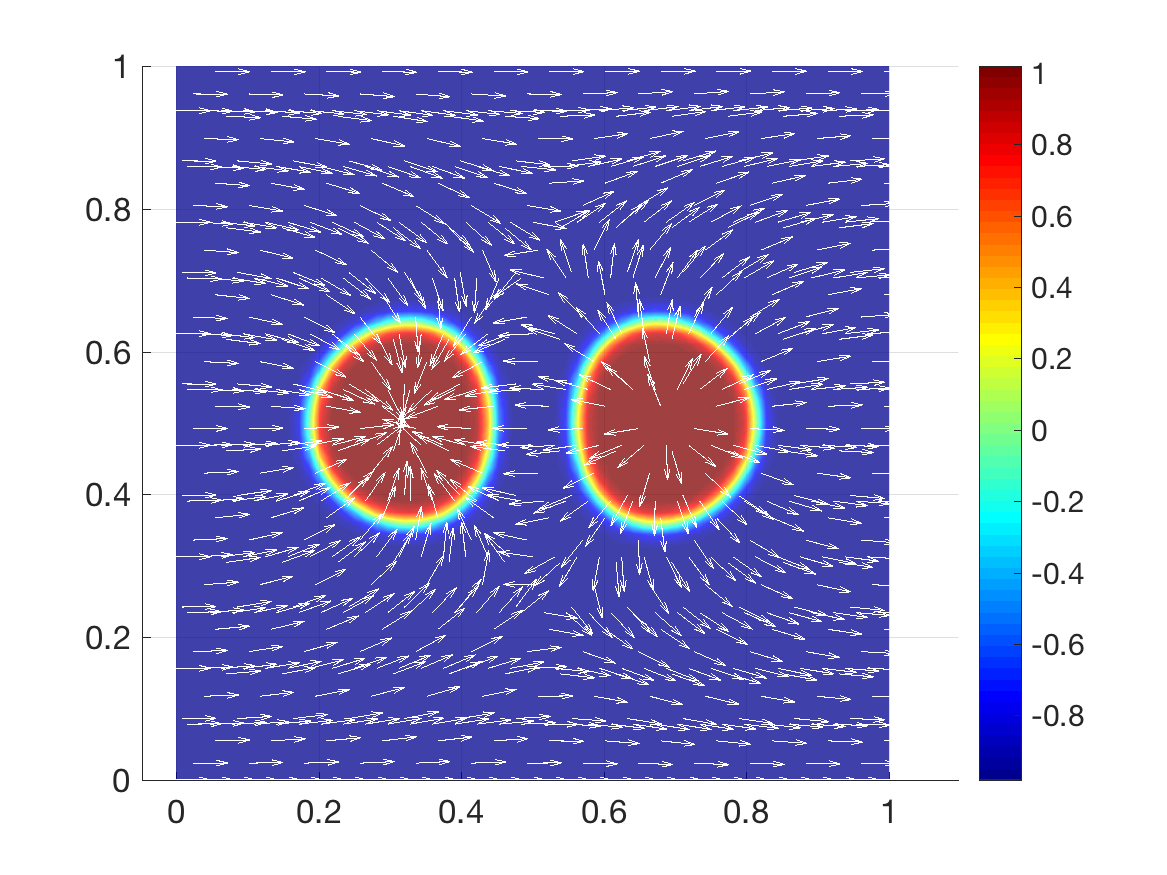}} \hspace{-2em}
\subfloat{\includegraphics[width = 1.75in]{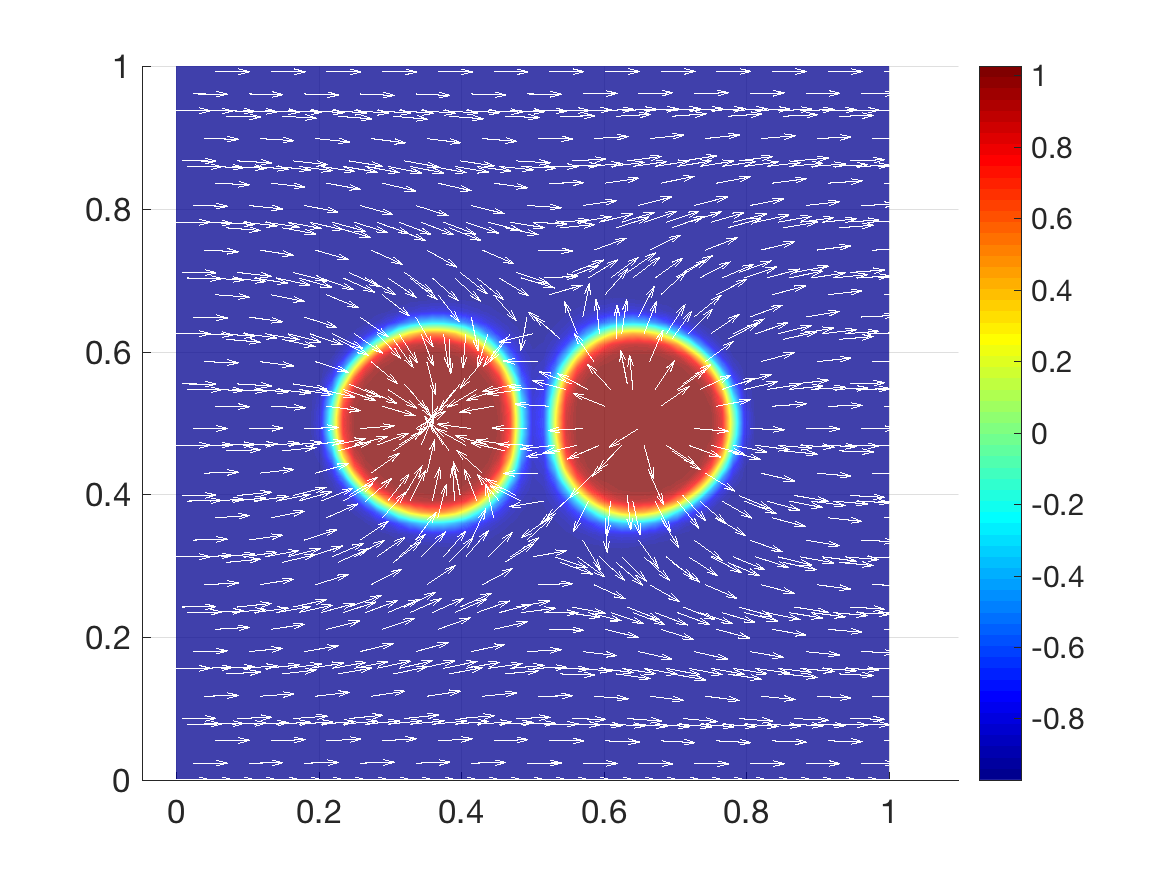}} \hspace{-2em}
\subfloat{\includegraphics[width = 1.75in]{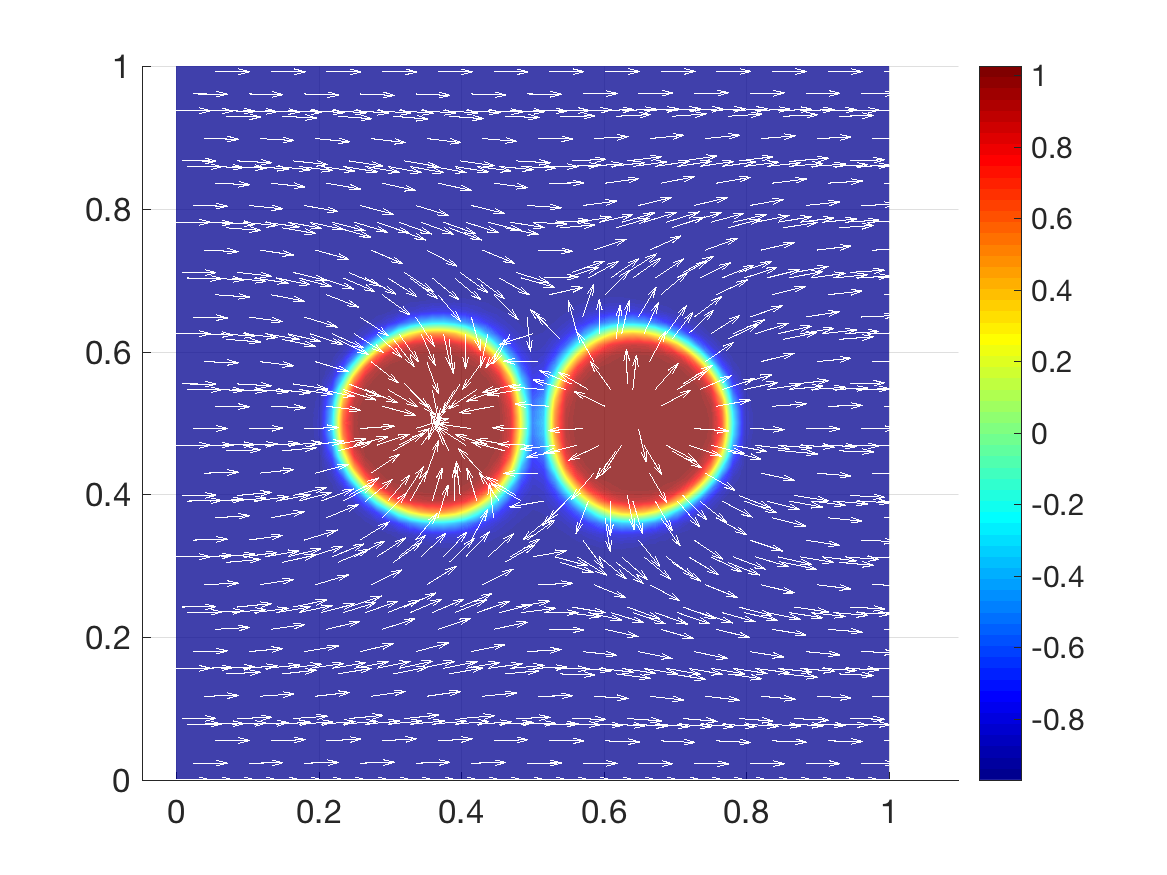}} \hspace{-2em}
\subfloat{\includegraphics[width = 1.75in]{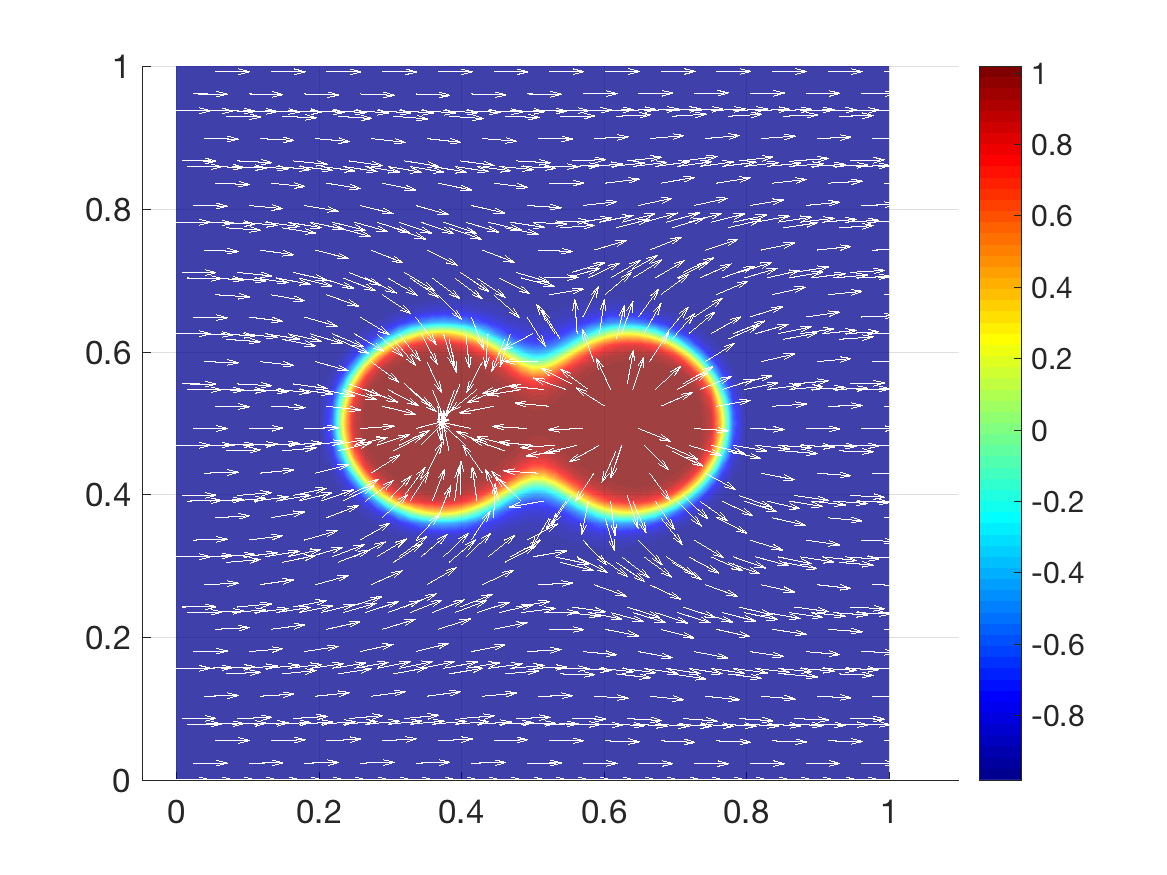}} \\[-4ex]
\subfloat{\includegraphics[width = 1.75in]{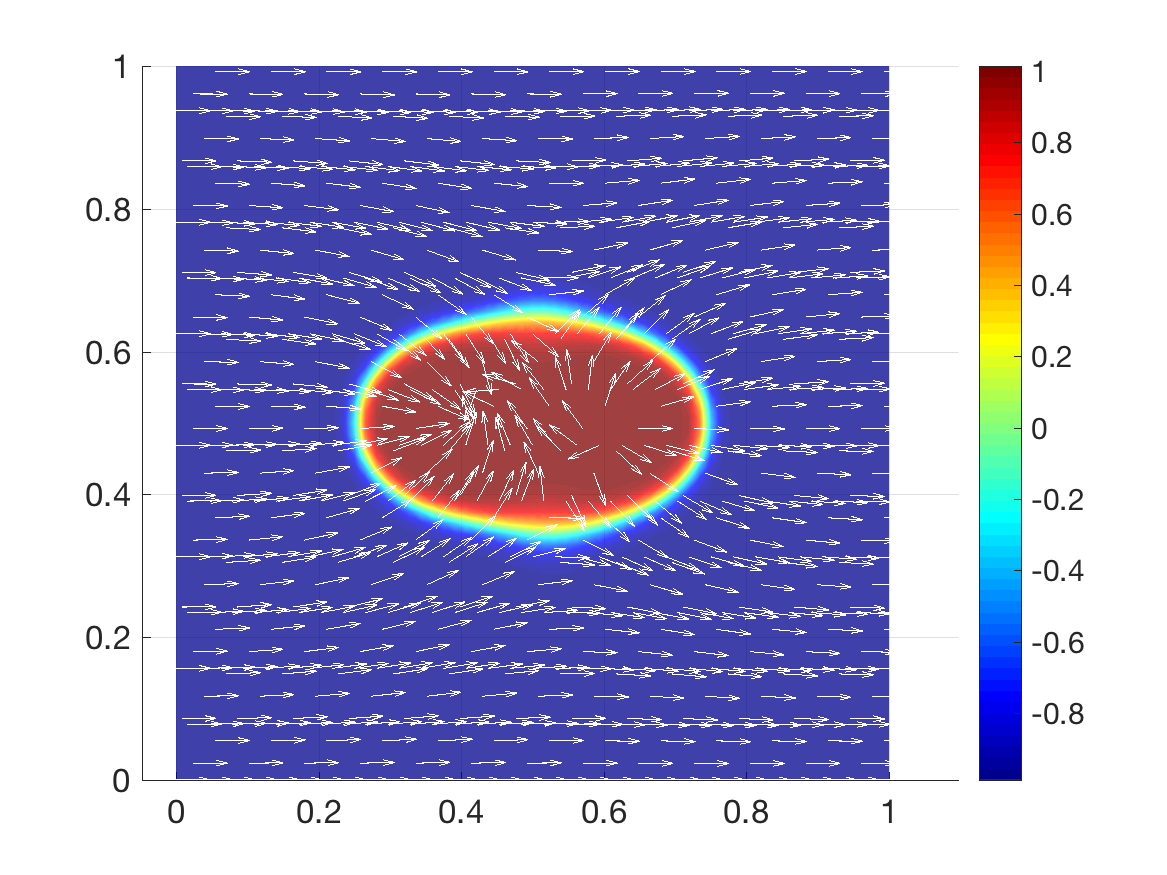}} \hspace{-2em}
\subfloat{\includegraphics[width = 1.75in]{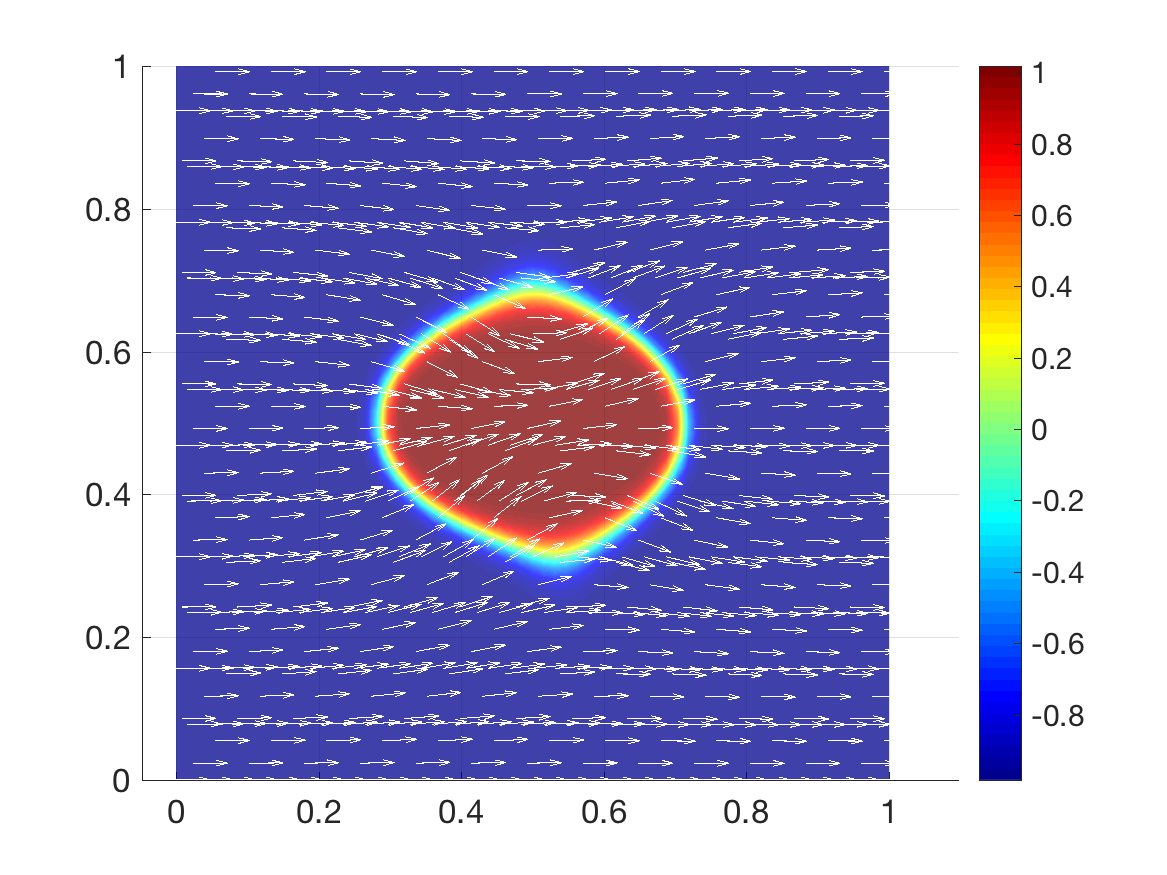}} \hspace{-2em}
\subfloat{\includegraphics[width = 1.75in]{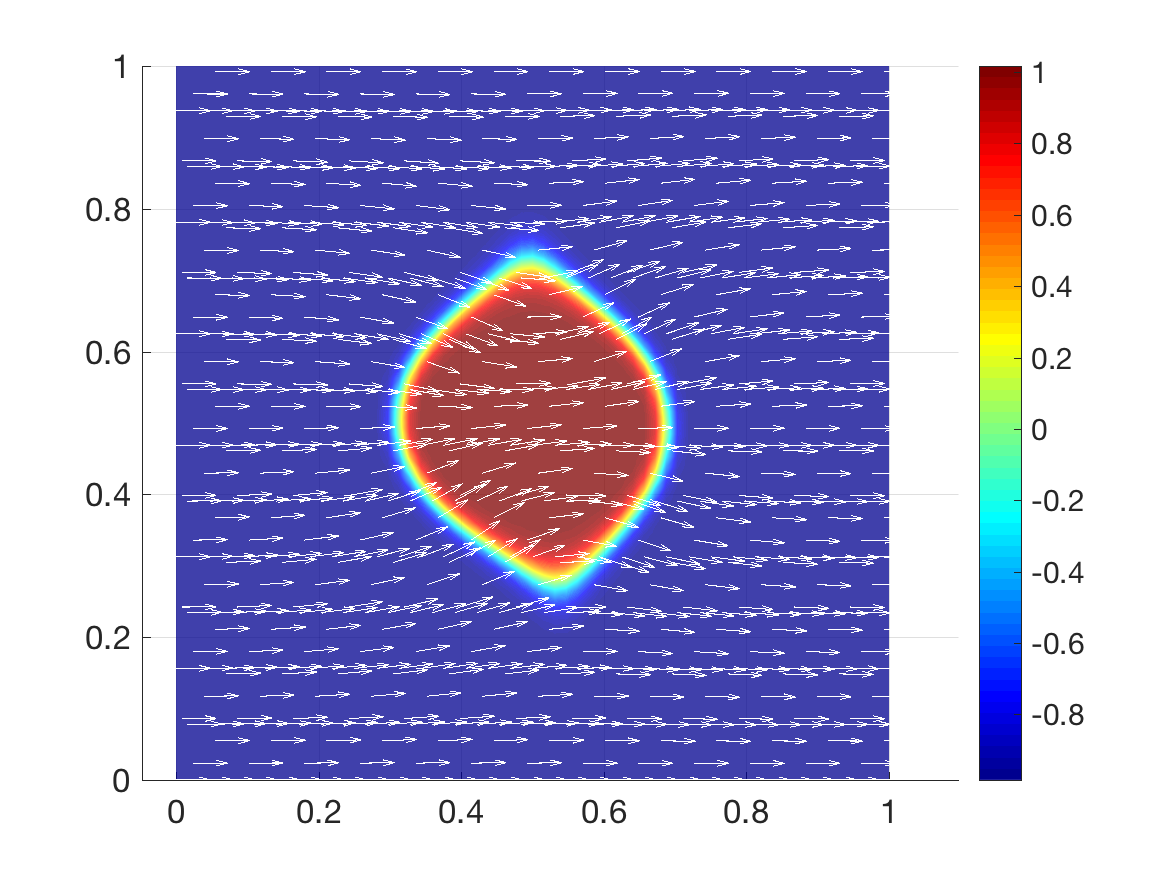}} \hspace{-2em}
\subfloat{\includegraphics[width = 1.75in]{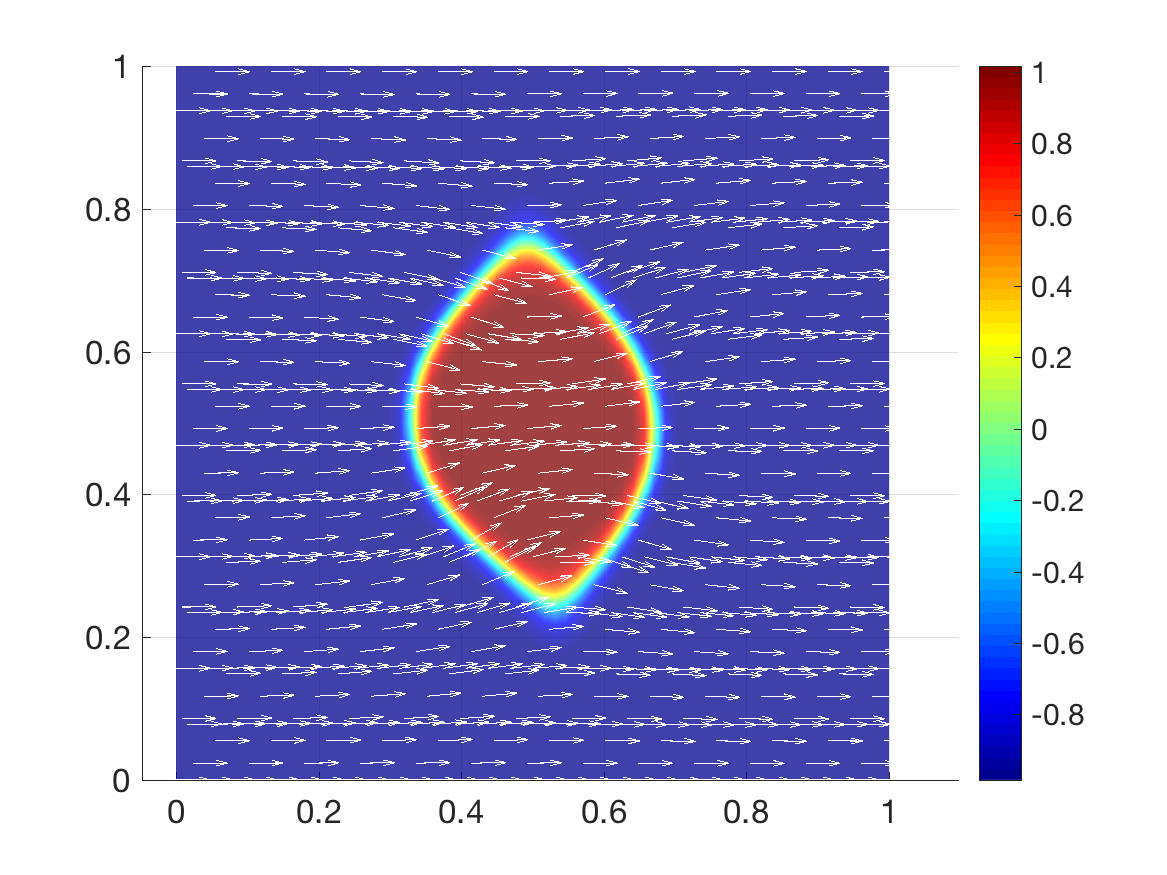}}\\[-4ex]
\caption{Droplet cornering, $\Omega = [0,1]\times[0,1]$, $h = \sqrt{2} / 64$, $\tau = 0.002$ (Section \ref{sec:collide-LC-droplet}). The times displayed are $t=0.2, t=0.4, t=0.42, t=0.44$ (top from left to right) and $t=0.48, t=0.52, t=0.56, t=0.6$ (bottom from left to right).}
\label{fig:droplet-collide}
\end{figure}

Figure \ref{fig:droplet-colliding-energy} displays the energy decreasing property of the scheme for this experiment.

\begin{figure}
\subfloat{\includegraphics[width = 3in]{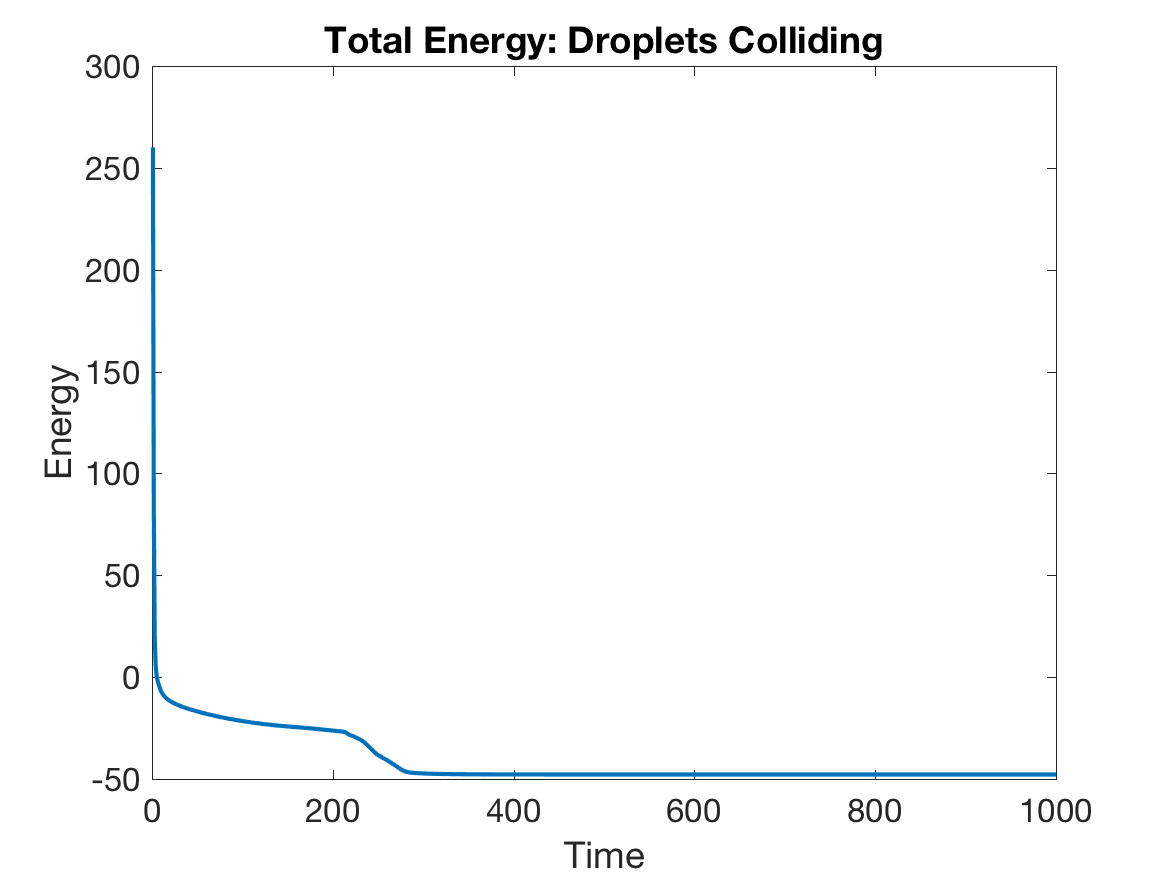}}
\subfloat{\includegraphics[width = 3in]{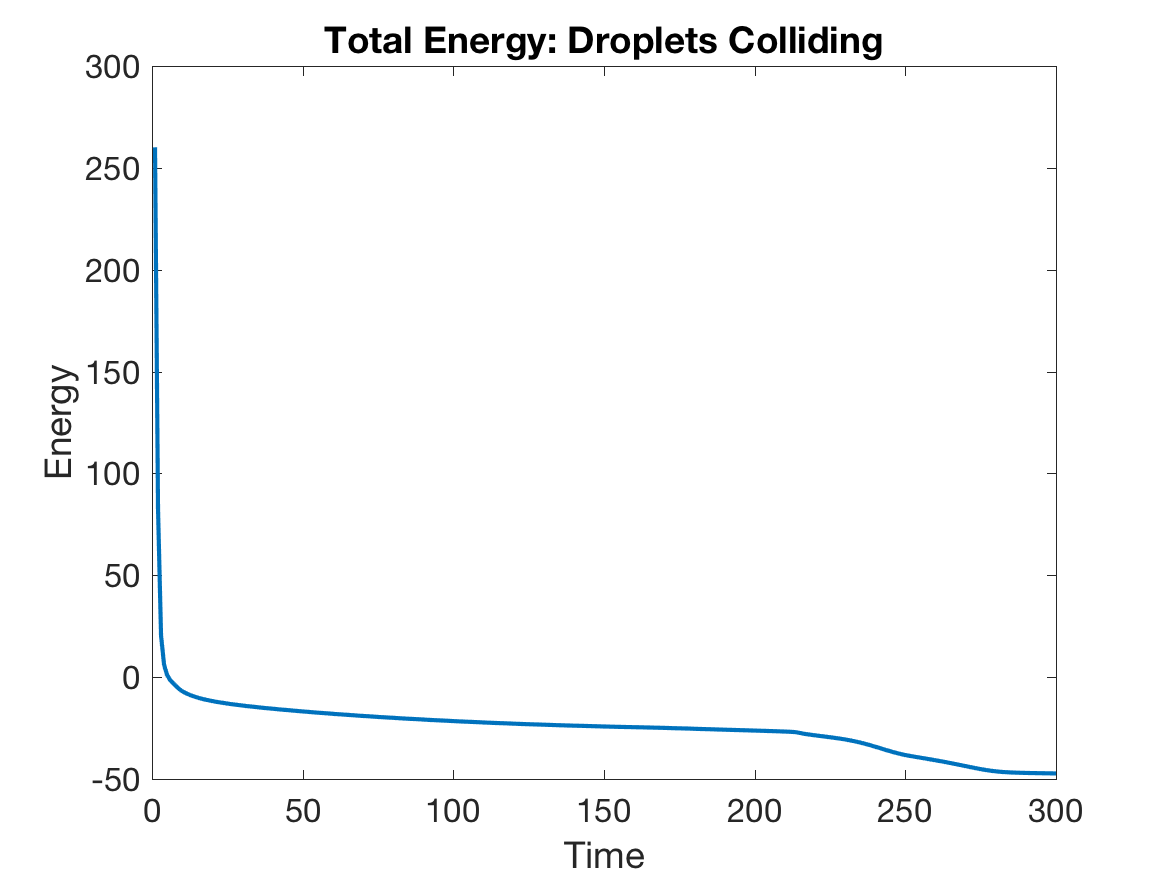}}
\caption{Total energy as a function of time for two droplets colliding (Section \ref{sec:collide-LC-droplet}).}
\label{fig:droplet-colliding-energy}
\end{figure}

\subsection{A Liquid Crystal Droplet Splitting}
\label{sec:splitting-LC-droplet}

The fourth numerical experiment demonstrates a liquid crystal droplet splitting into two droplets. The initial conditions are as follows:
\begin{align*}
s_h^0 &= s^*,
\\
\bnh^0 &=
\begin{cases}
 \frac{(x,y) - (0.35,0.5)}{|(x,y) - (0.35,0.5)|}, \quad x \le 0.5,
\\
\\
 \frac{-\left((x,y) - (0.65,0.5)\right)}{|(x,y) - (0.65,0.5)|}, \quad x > 0.5,
\end{cases}
\\
\phih^0 &= I_h\left\{-\tanh\left(\frac{(x-0.5)^2/0.03 + (y-0.5)^2/0.03 - 1}{2\varepsilon}\right)\right\}.
\end{align*}
The following Dirichlet boundary conditions on $\partial \Omega$ are imposed for $s$ and $\bn$:
\begin{align*}
s = s^*, \quad \bnh &= \begin{cases}
 \frac{(x,y) - (0.3,0.5)}{|(x,y) - (0.3,0.5)|}, \quad x \le 0.5,,
\\
\\
 \frac{-\left((x,y) - (0.7,0.5)\right)}{|(x,y) - (0.7,0.5)|}, \quad x > 0.5.
\end{cases}
\end{align*}

The relevant parameters are $\kappa = 1, \rho = 1, \Werk =  1, \Wdw = 100, \Wchdw = 1, \Wchgd = 1 + \frac{1}{4}(\Wwan + \Wwas) = 11, \Wwas = 20, \Wwan = 20$. The space step size is taken to be $h = \sqrt{2}/64$ and the time step size is taken to be $\tau = 0.002$ with a final stopping time of $T=2.0$. The interfacial width parameter is taken to be $\varepsilon = 3h/\sqrt{2}$. Figure \ref{fig:droplet-split} shows the evolution of the droplet over time. The top two rows display the evolution of the scalar degree of orientation parameter $s$. The bottom two rows show the evolution of the phase field parameter $\phi$ and the director field $\bn$.  The boundary conditions for $\bn$ \emph{induce} two defects in the domain (no annihilation), and the liquid crystal elastic energy acts to push the defects further apart. We note that the weighting on the Cahn-Hilliard gradient energy term $\Wchgd$ is lower than in the previous experiments effectively lowering surface tension on the droplet. If, for example, $\Wchgd = 1 + (\Wwan + \Wwas) = 21$ as before, then the droplet would hold together. Since surface tension is relatively weak in this example, the droplet splits to accommodate the separation of the defects.

\begin{figure}
\subfloat{\includegraphics[width = 2in]{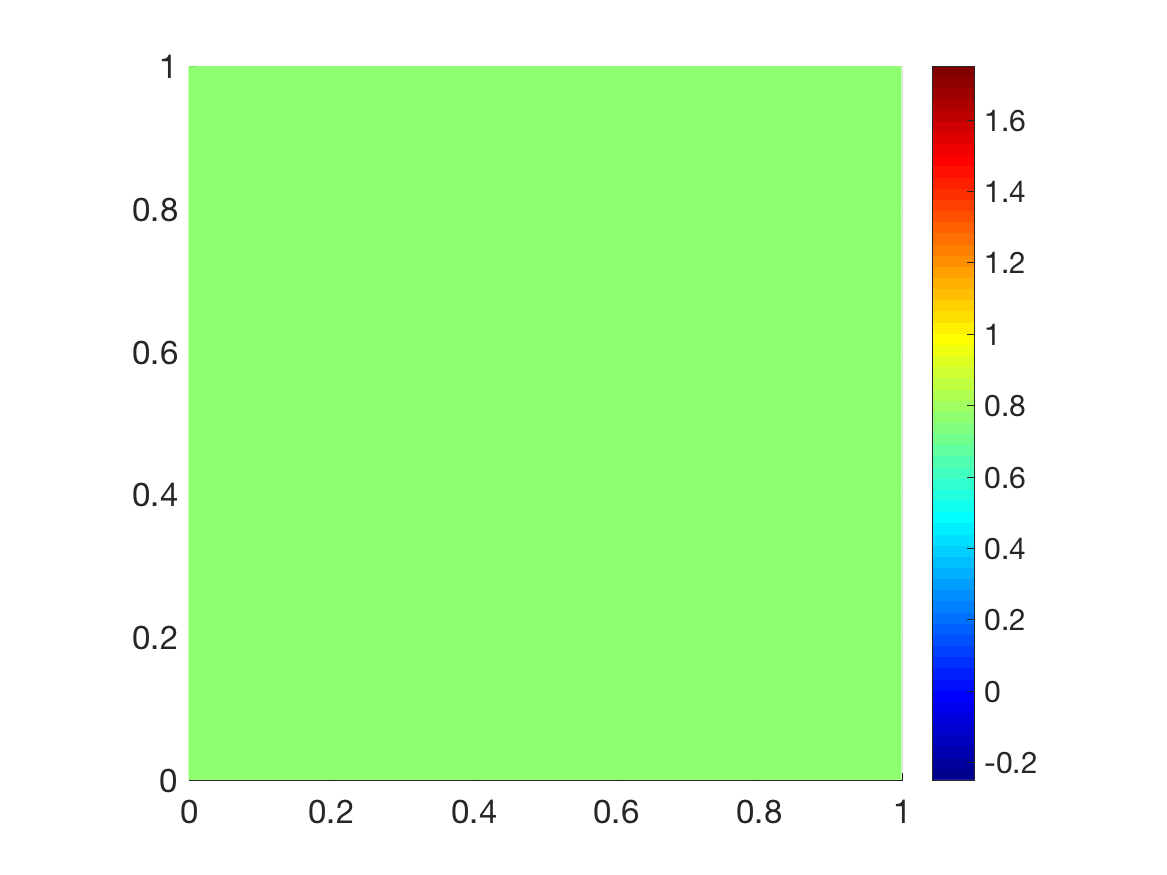}}
\subfloat{\includegraphics[width = 2in]{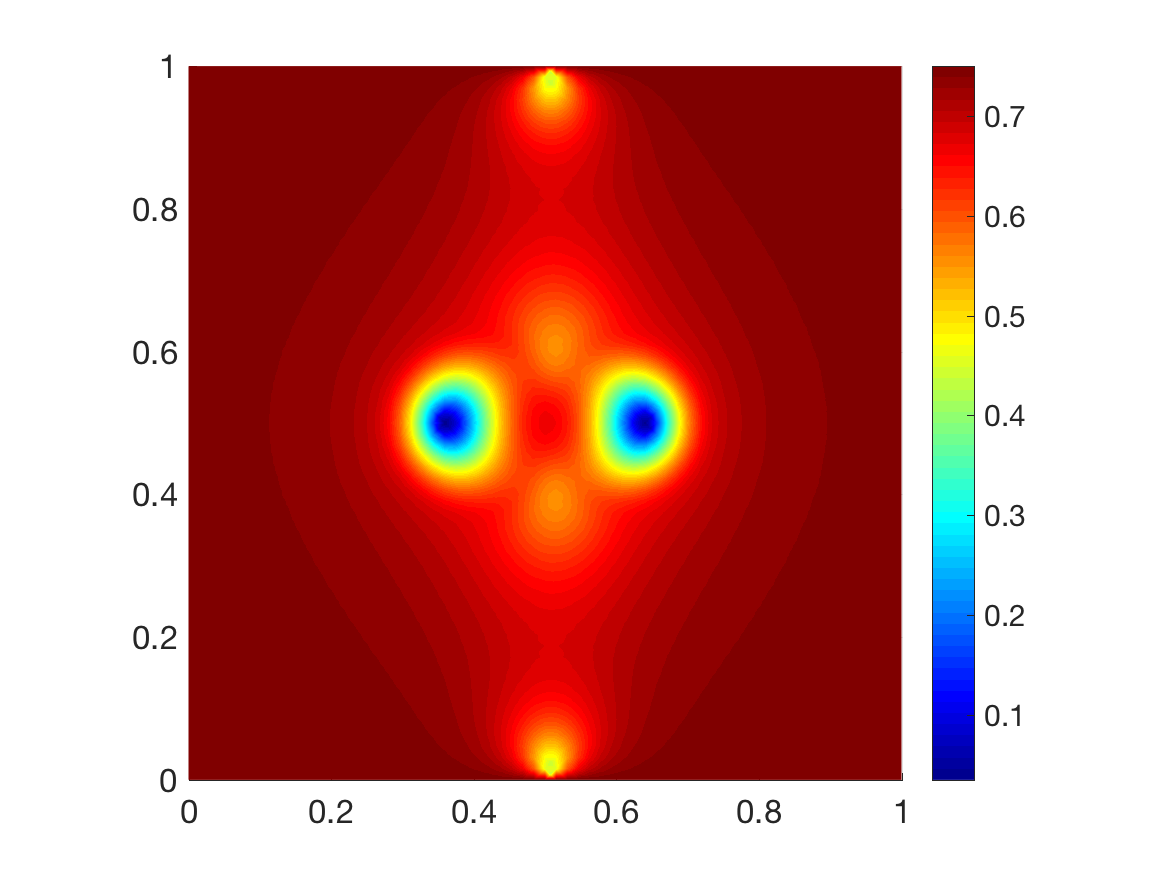}}
\subfloat{\includegraphics[width = 2in]{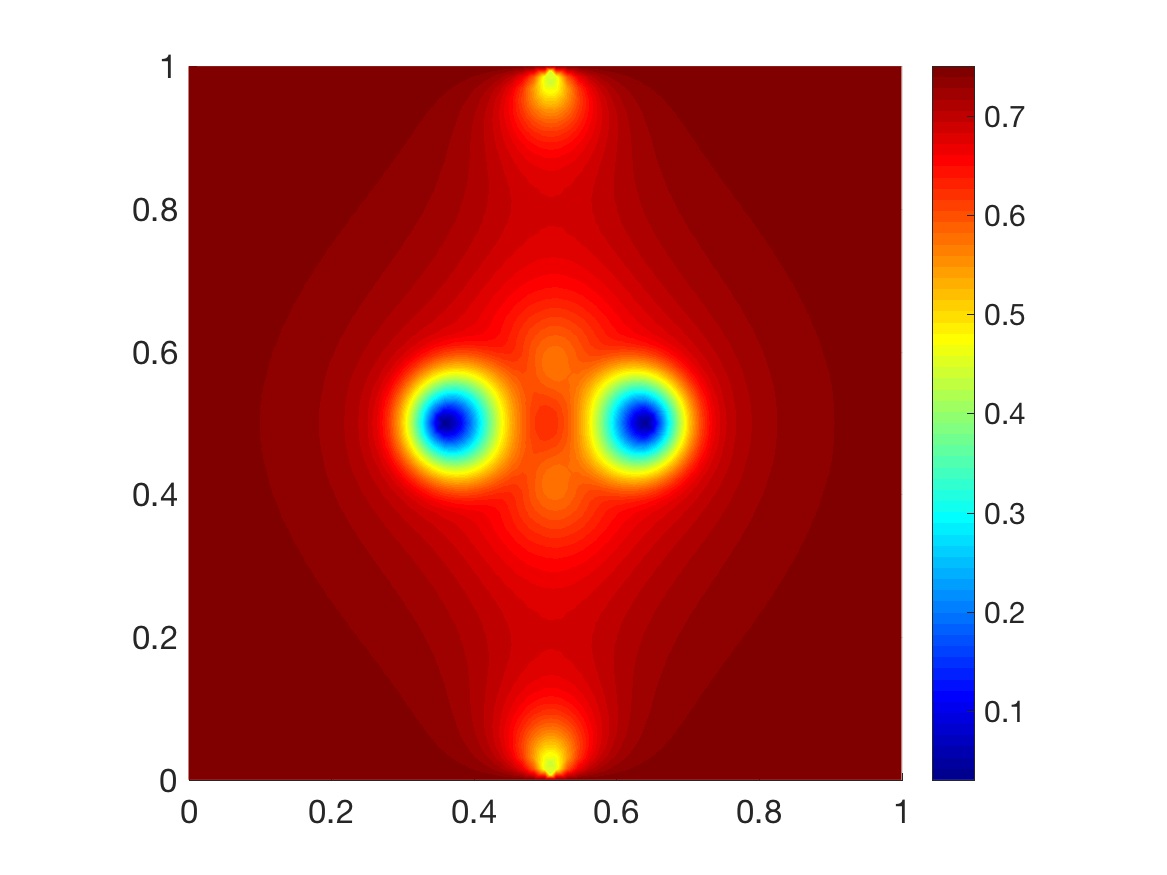}} \\[-4ex]
\subfloat{\includegraphics[width = 2in]{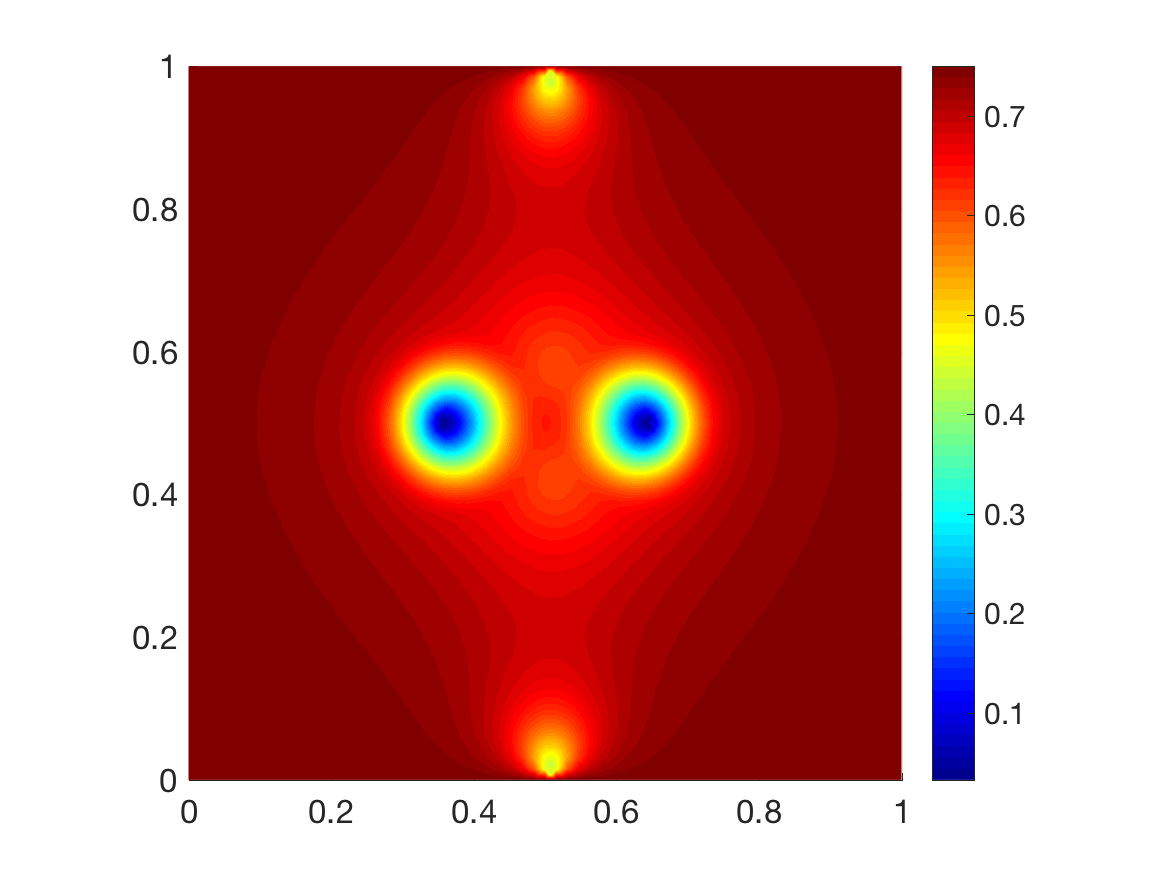}}
\subfloat{\includegraphics[width = 2in]{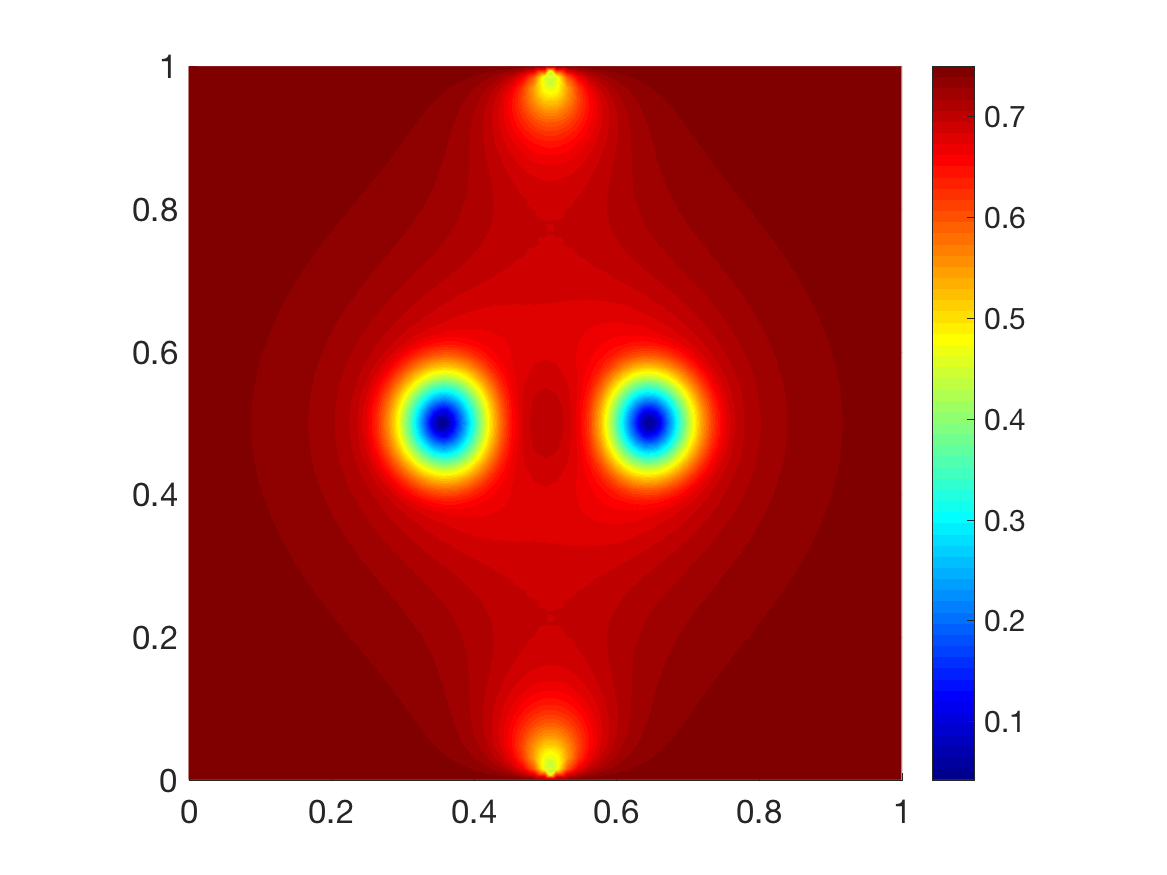}}
\subfloat{\includegraphics[width = 2in]{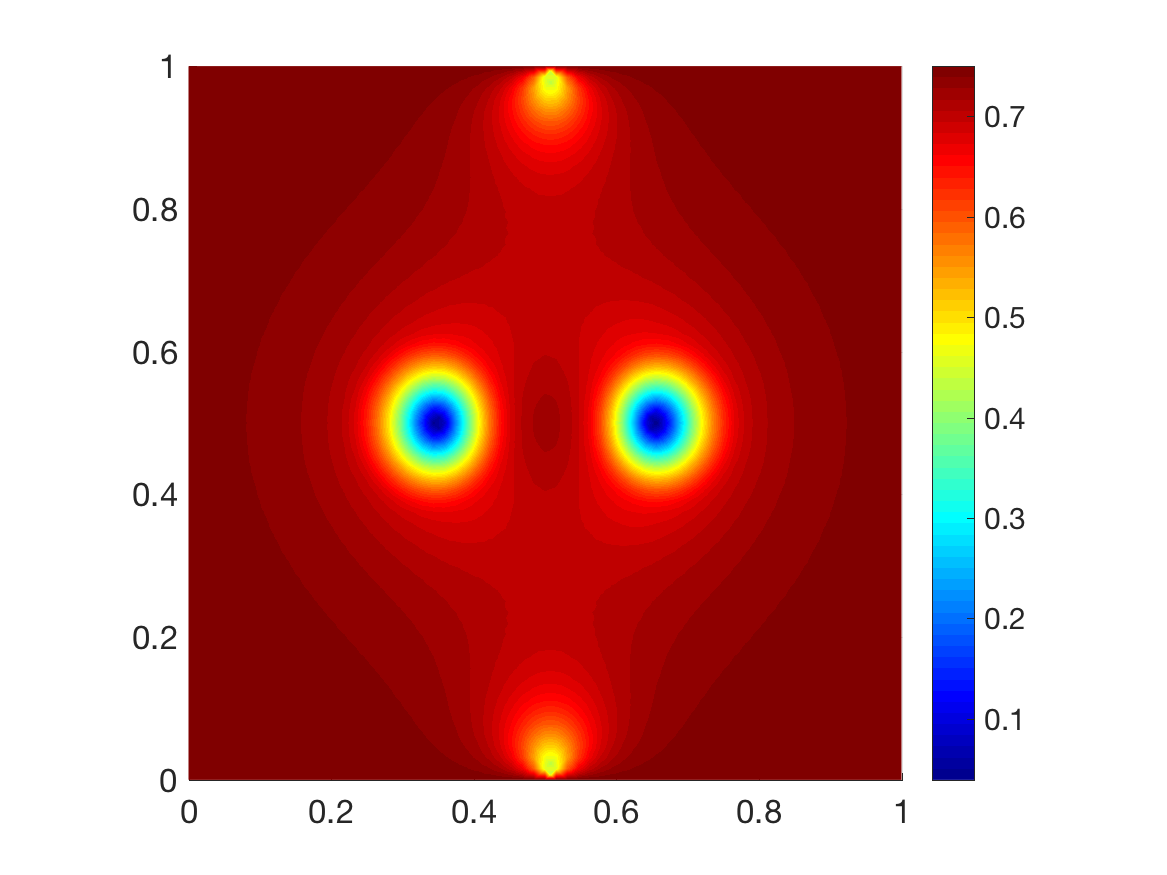}}\\[-4ex]
\subfloat{\includegraphics[width = 2in]{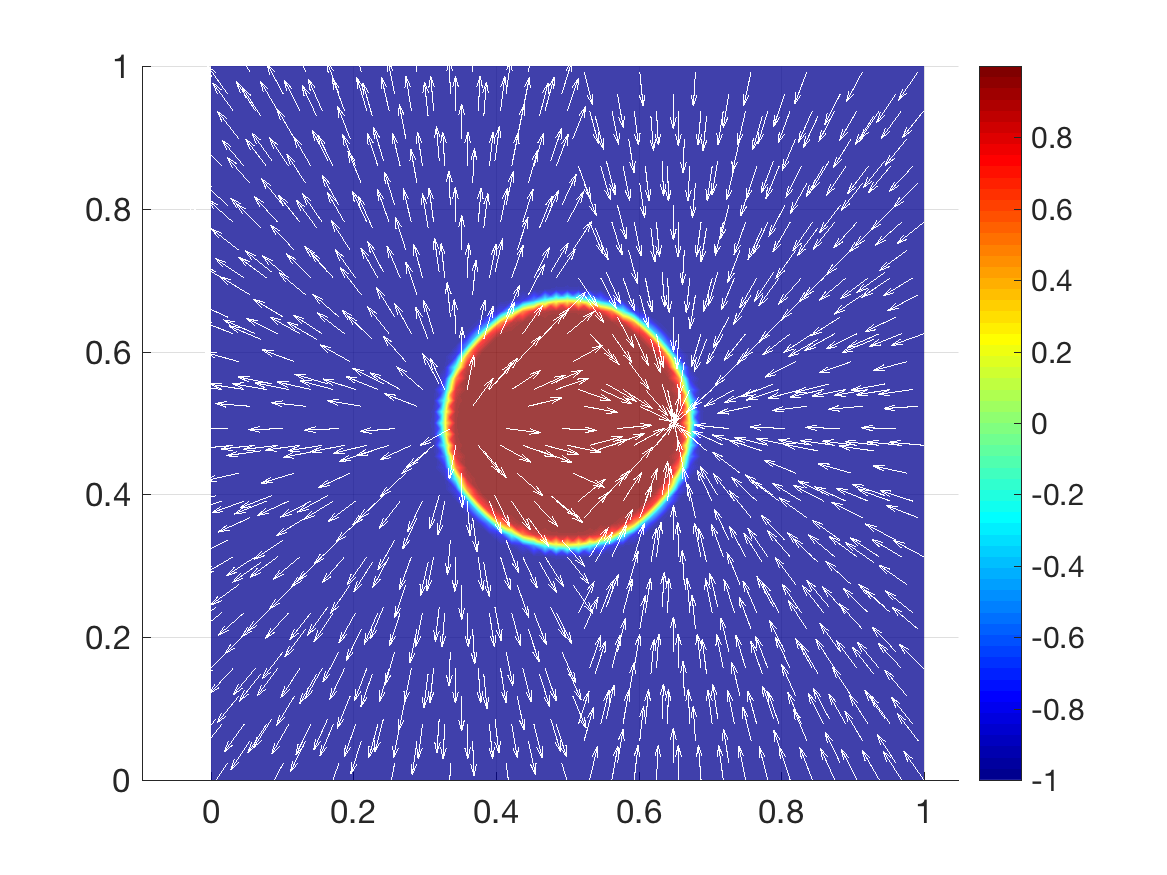}}
\subfloat{\includegraphics[width = 2in]{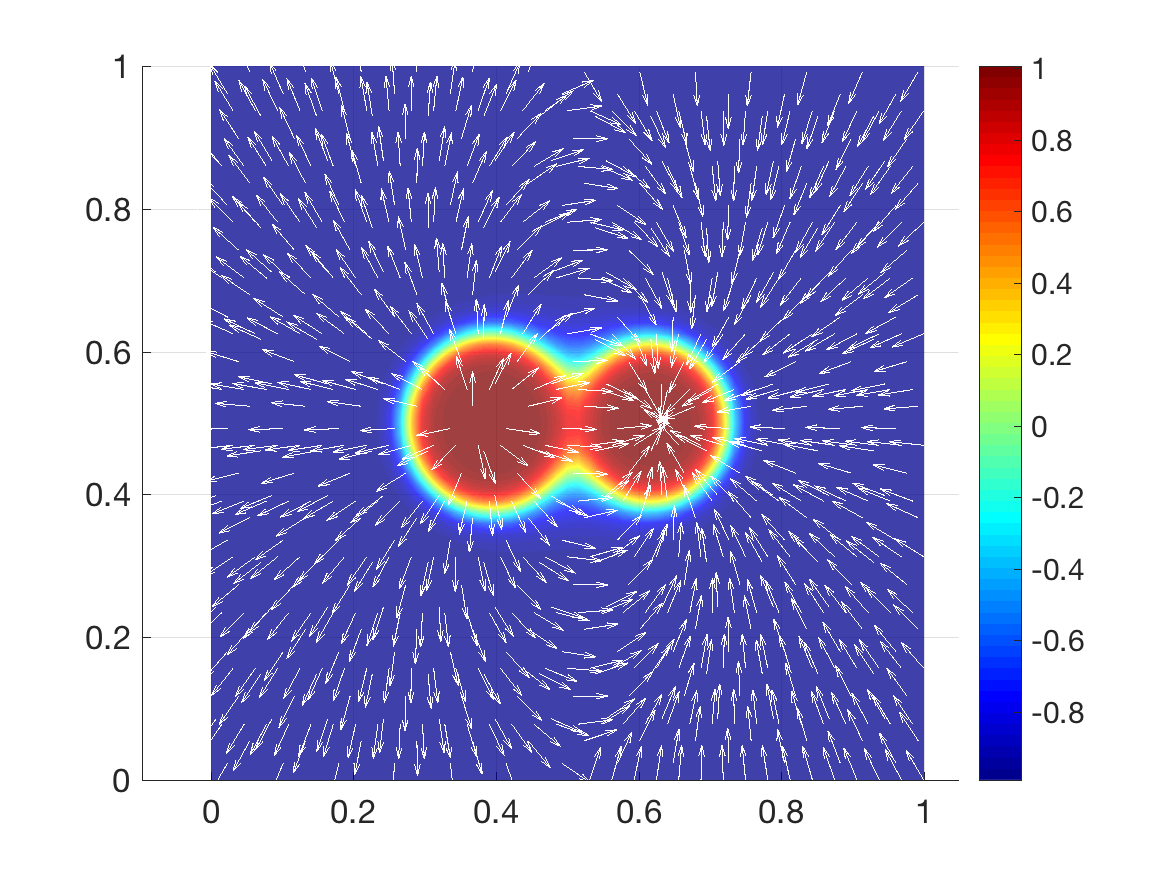}}
\subfloat{\includegraphics[width = 2in]{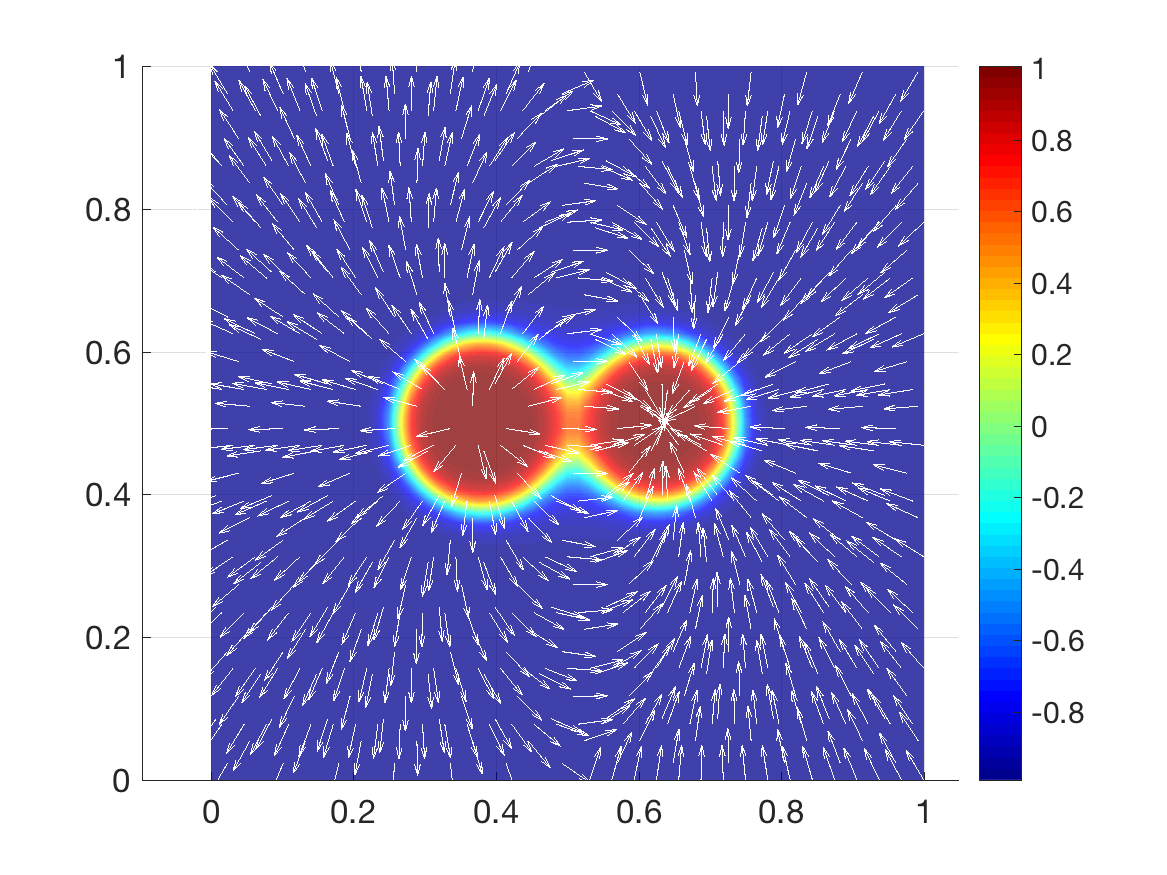}} \\[-4ex]
\subfloat{\includegraphics[width = 2in]{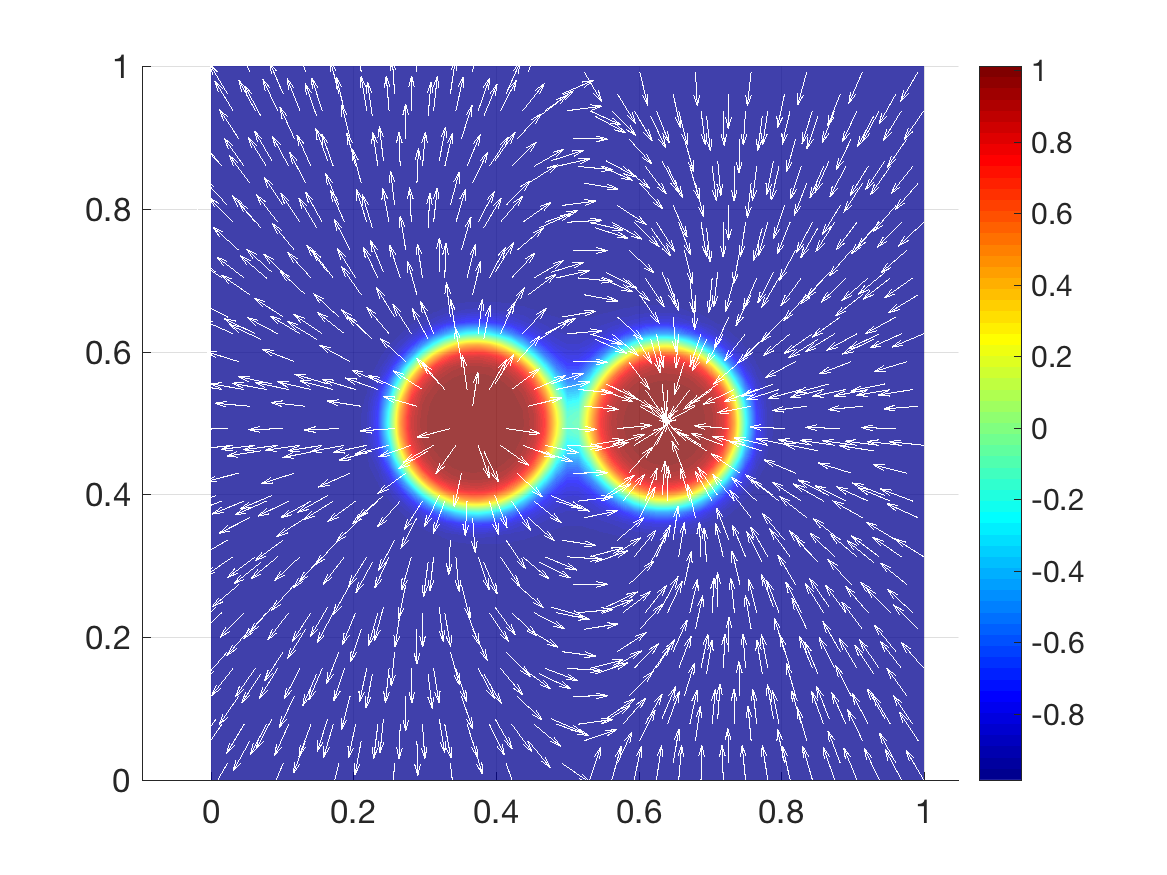}}
\subfloat{\includegraphics[width = 2in]{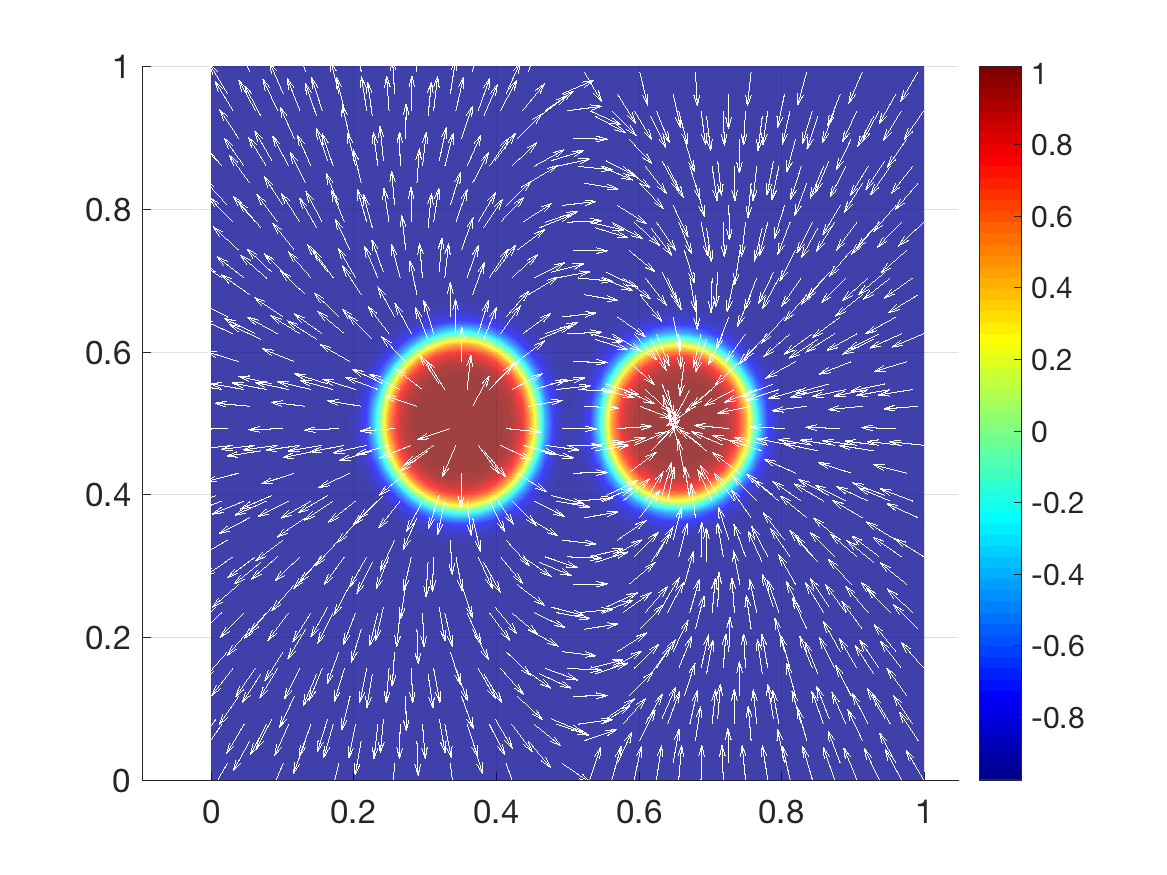}}
\subfloat{\includegraphics[width = 2in]{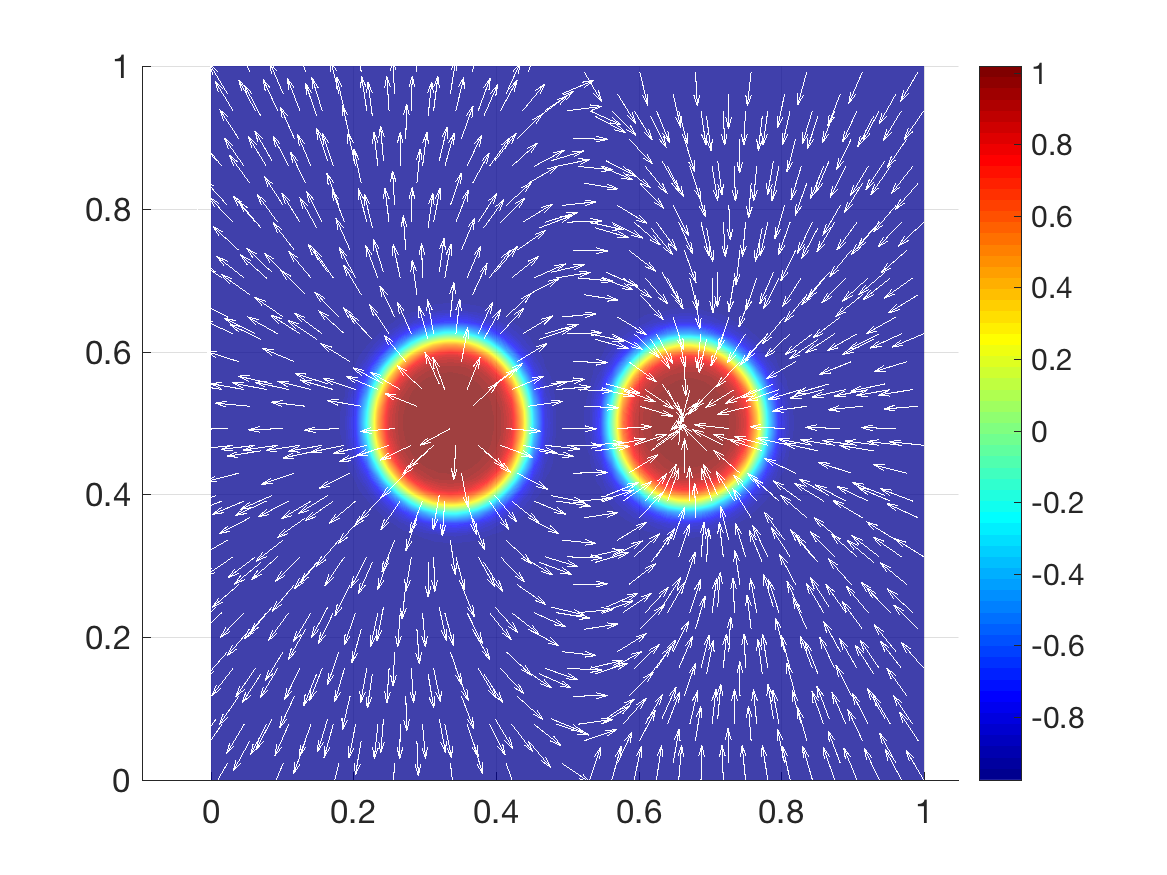}}
\caption{Droplet cornering, $\Omega = [0,1]\times[0,1]$, $h = \sqrt{2} / 64$, $\tau = 0.002$ (Section \ref{sec:splitting-LC-droplet}). The times displayed are $t=0, t=0.04, t=0.08$ (top from left to right) and $t=0.12, t=0.16, t=0.2$ (bottom from left to right).}
\label{fig:droplet-split}
\end{figure}

Figure \ref{fig:droplet-splitting-energy} displays the energy decreasing property of the scheme for this experiment.

\begin{figure}
\subfloat{\includegraphics[width = 3in]{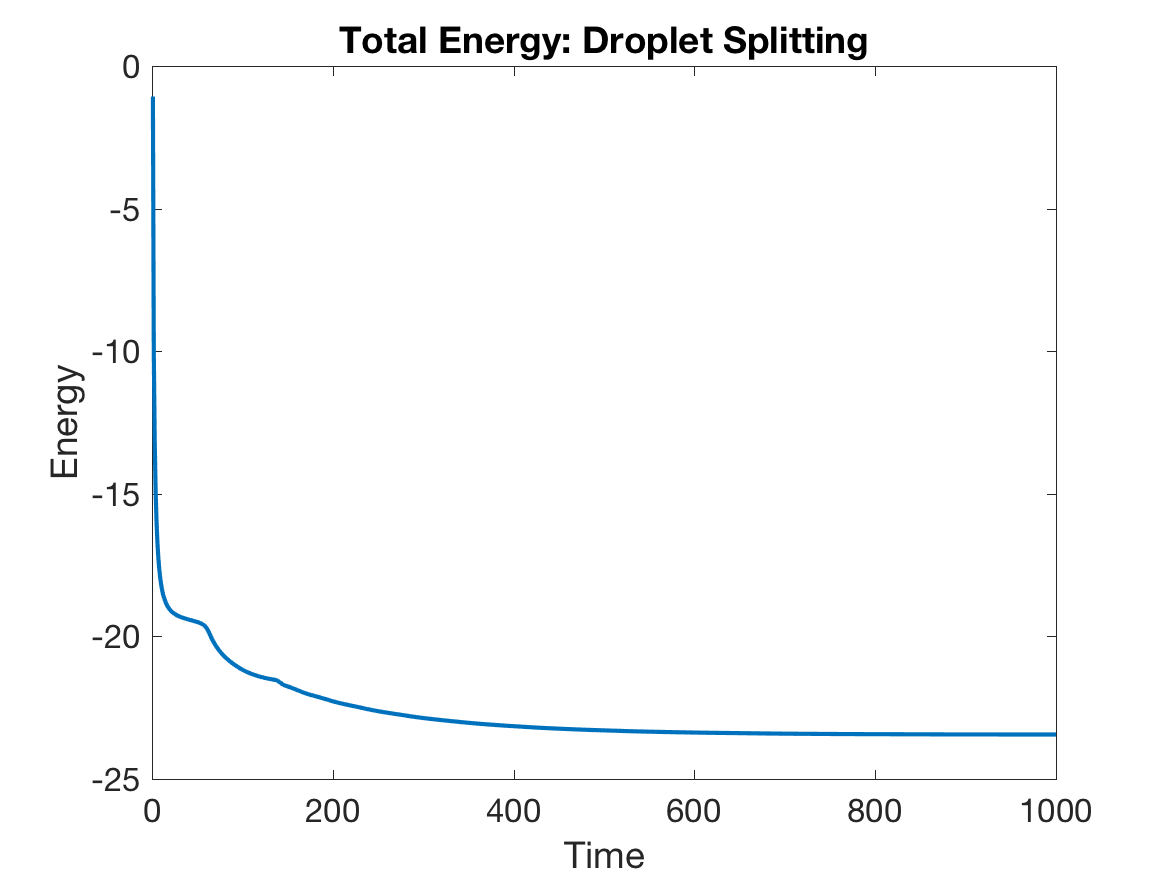}}
\subfloat{\includegraphics[width = 3in]{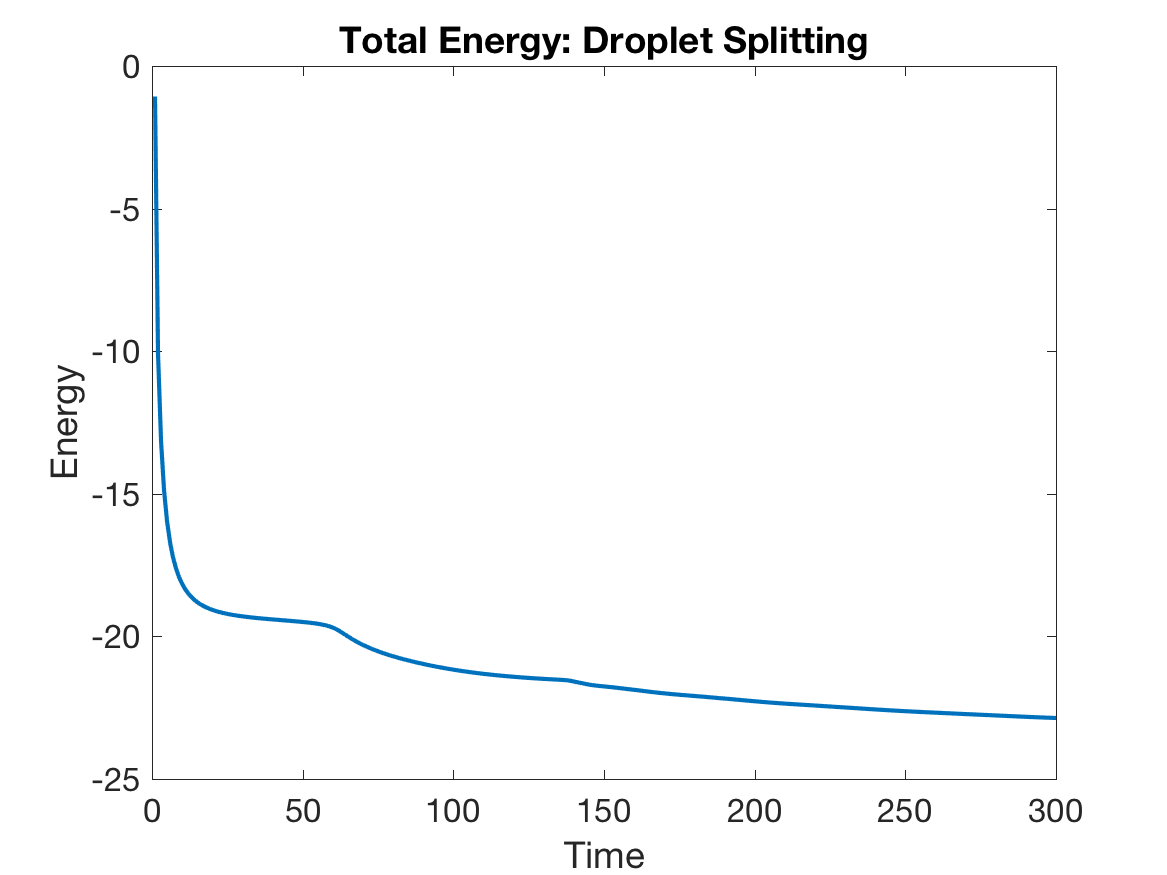}}
\caption{Total energy as a function of time for a droplet splitting (Section \ref{sec:splitting-LC-droplet}).}
\label{fig:droplet-splitting-energy}
\end{figure}

\section{Conclusion}\label{sec:conclusion}

We introduced a phase field model and finite element scheme for nematic liquid crystal droplets in a pure liquid crystal substance.  We presented a finite element method and gradient flow scheme, and used it to explore gradient flow dynamics for finding energy minimizers. We were able to show that the gradient flow method has a monotone energy decreasing property.  We also demonstrated that the discrete energy of the numerical scheme converges, in the sense of $\Gamma$-convergence, to the continuous free energy of the model. Finally, we presented numerical experiments demonstrating four different aspects of liquid crystal droplets:  movement/positioning, cornering, coalescence, and splitting.

Some extensions of this work are: include more general liquid crystal elastic energies, electro-static effects, and coupling to fluid dynamics (e.g. Stokes flow).  Moreover, development of a multi-grid solver for the Cahn-Hilliard equation \cite{brenner17} would enable computations in three dimensions; indeed, this would allow for investigating the connection between defect structures and droplet shapes. Furthermore,  our method could be used to model optimal shapes of liquid crystal droplets, e.g. tactoids \cite{Atherton_LC2016}, nematic droplets on fibers \cite{Batista_PRE2015}, and nematic shells \cite{Serra_LC2016}.  Other applications could be in optimal control of droplets and self-assembly of arrays of droplets.

\section*{Acknowledgements}

S.~W.~Walker acknowledges financial support by the NSF via DMS-1418994 and DMS-1555222 (CAREER).

\bibliographystyle{siam}
\bibliography{PHLC_BIB}
\end{document}